\documentclass[12pt]{amsart}

\usepackage[utf8]{inputenc}
\usepackage[T1]{fontenc}
\usepackage[english]{babel}

\usepackage{amsfonts}
\usepackage{amsmath}
\usepackage{amssymb}
\usepackage{amstext}
\usepackage{amsthm}

\usepackage{fullpage}

\usepackage[shortlabels]{enumitem}
\usepackage{graphicx}
\usepackage{hyperref}
\usepackage{mathrsfs}
\usepackage{mathtools}
\usepackage{tikz}
\usepackage{capt-of}
\usepackage{array}

\allowdisplaybreaks

\newcommand{\wtH}{\widetilde{H}}
\newcommand{\weyl}{\boldsymbol{\sigma}}
\newtheorem{theorem}{Theorem}[section]
\newtheorem{lemma}[theorem]{Lemma}
\newtheorem{proposition}[theorem]{Proposition}
\newtheorem{corollary}[theorem]{Corollary}
\newtheorem{conjecture}[theorem]{Conjecture}

\theoremstyle{definition}
\newtheorem{definition}[theorem]{Definition}
\newtheorem{example}[theorem]{Example}

\theoremstyle{remark}
\newtheorem{remark}[theorem]{Remark}

\theoremstyle{plain}

\newcommand{\area}{\operatorname{area}}
\newcommand{\dinv}{\operatorname{dinv}}
\newcommand{\inv}{\operatorname{inv}}
\newcommand{\doff}{\operatorname{doff}}
\newcommand{\pol}{\operatorname{pol}}
\newcommand{\ides}{\operatorname{ides}}
\newcommand{\comp}{\operatorname{comp}}

\DeclareMathOperator{\tdinv}{tdinv}
\DeclareMathOperator{\pdinv}{pdinv}
\DeclareMathOperator{\maxtdinv}{maxtdinv}
\DeclareMathOperator{\rank}{rank}

\newcommand{\D}{\operatorname{D}} 
\newcommand{\GL}{\operatorname{GL}} 
\newcommand{\DP}{\operatorname{DP}} 
\newcommand{\LD}{\operatorname{LD}} 
\newcommand{\lighthline}{\noalign{{\color{black!20}\hrule height\arrayrulewidth}}} 

\title{The nabla operator, Hall--Littlewood polynomials and LLT polynomials}

\author{Feng Gu}
\address{Center for Combinatorics, LPMC, Nankai University, Tianjin 300071, P. R. China}
\email{fgu@mail.nankai.edu.cn}
\thanks{}

\author{Mingyang Kang}
\address{Center for Combinatorics, LPMC, Nankai University, Tianjin 300071, P. R. China}
\email{1120260002@mail.nankai.edu.cn}
\thanks{}

\author{Dun Qiu}
\address{Center for Combinatorics, LPMC, Nankai University, Tianjin 300071, P. R. China}
\email{qiudun@nankai.edu.cn}
\thanks{Dun Qiu is supported in part by the National Natural Science Foundation of China (No. 12671390),
the Natural Science Foundation of Tianjin (No. 24JCZDJC01390), and the Fundamental Research Funds for the Central Universities (No. 105-63263093).}

\author{Minhao Zhang}
\address{Center for Combinatorics, LPMC, Nankai University, Tianjin 300071, P. R. China}
\email{carpentercheungmath@gmail.com}
\thanks{}

\date{\today}


\begin{document}

\begin{abstract}
Bergeron, Garsia, Haiman, and Tesler conjectured that $(-1)^{n-\ell(\mu)}\nabla\wtH_\mu[X;0,t]$ and $\nabla\omega\wtH_\mu[X;0,t^{-1}]$ are Schur positive, 
and that $\nabla\omega\wtH_\mu[X;0,t^{-1}]-\nabla\omega\wtH_\nu[X;0,t^{-1}]$ is Schur positive when $\mu$ dominates $\nu$.  The first statement follows from the compositional shuffle conjecture of Haglund, Morse, and Zabrocki, proved by Carlsson and Mellit.  We prove the other two statements.  
All three statements also hold for $\nabla^r$ with $r\geq1$.  The main new step is an identity: 
when $\mu$ covers $\nu$, the difference $\omega\wtH_\mu[X;0,t^{-1}]-\omega\wtH_\nu[X;0,t^{-1}]$ is minus a power of $t$ times one LLT polynomial.  
Using the Catalanimals and the Cauchy identity of Blasiak, Haiman, Morse, Pun, and Seelinger, we show that $\nabla^r$ of this LLT polynomial is, 
up to sign, a positive sum of LLT polynomials.  We also extend the main results to the Schiffmann algebra setting.
\end{abstract}

\maketitle

\section{Introduction}
The nabla operator, introduced by Bergeron and Garsia~\cite{BG99}, played a significant role in their study of
modified Macdonald polynomials.  The diagonal-harmonics
conjecture of Garsia and Haiman~\cite{GH96} asserts that the bigraded Frobenius
characteristic of the space of diagonal harmonics is $\nabla e_n[X]$.
Haiman~\cite{Hai02} proved this statement.  Consequently,
$\left\langle\nabla e_n[X],s_\lambda[X]\right\rangle\in\mathbb N[q,t]$.

Haglund, Haiman, Loehr, Remmel, and Ulyanov~\cite{HHLRU05}
conjectured a parking-function formula for $\nabla e_n[X]$ in 2005,
now called the shuffle conjecture.
Moreover, the shuffle conjecture implies a positive expansion of $\nabla e_n[X]$ in LLT polynomials. 
In 2012, Haglund, Morse, and Zabrocki~\cite{HMZ12} introduced the
functions $C_\alpha[X;t]$ and formulated the compositional shuffle
conjecture; Carlsson and Mellit~\cite{CM18} proved it.
Gorsky and Negu\c{t}~\cite{GN15} formulated the rational shuffle
conjecture, and Bergeron, Garsia, Leven, and Xin~\cite{BGLX16}
proposed the compositional $(km,kn)$-shuffle conjecture.
Mellit~\cite{Mel21} proved the latter conjecture.
Together with the Schur positivity of LLT polynomials
established by Grojnowski and Haiman \cite{GrojnowskiHaiman07}, these results
imply that
$\left\langle\nabla^r C_\alpha[X;t],s_\lambda[X]\right\rangle
\in\mathbb N[q,t]$.  Further developments in the shuffle theorem
appear in the work of Blasiak, Haiman, Morse, Pun, and Seelinger~\cite{shuffleline,LLTcat,dens}.

In this paper, we are interested in the action of $\nabla$ on certain
special symmetric functions.  Bergeron et al.~\cite{BGHT99} formulated several conjectures concerning this
action.  Their Conjecture~I~\cite{BGHT99} asserts that, 
for every positive integer $m$ and every partition $\lambda$, $\nabla^m s_\lambda[X]$ 
is Schur positive up to a sign that depends only on $\lambda$.  
The LLT polynomial formula for $\nabla^m s_\lambda[X]$ conjectured by Loehr and Warrington \cite{LW08} 
was proved by Blasiak et al.~\cite{dens}, and it proves Conjecture I.  
Conjecture~IV~\cite{BGHT99},
asserting the signed Schur positivity of $\nabla m_\mu[X]$, was recently
proved by Qiu and Zhang~\cite{QZ26}.  We focus on Conjectures II and III in this paper.
Conjecture II follows from the compositional shuffle conjecture of Haglund, Morse, and Zabrocki~\cite{HMZ12}, proved by Carlsson and Mellit~\cite{CM18}; see Theorem~\ref{thmHMZ} below.
Very recently, Qu~\cite{Qu26} proved these two conjectures when the
indexing partition has at most two columns, and extended them to all
powers of $\nabla$; we prove Conjecture III here for arbitrary partitions, and we prove both conjectures for all
powers of $\nabla$.  Unless another section is indicated, the notation
used in this introduction is defined in Sections 2--4.

\begin{conjecture}[{\cite[Conjecture~II]{BGHT99}}]\label{conjecture2}
For any partitions \(\mu,\lambda\vdash n\), we have
\begin{equation}
\left\langle
(-1)^{n-\ell(\mu)}
\nabla\wtH_\mu[X;0,t],
s_\lambda[X]
\right\rangle
\in
\mathbb N[q,t].
\end{equation}
\end{conjecture}

\begin{conjecture}[{\cite[Conjecture~III]{BGHT99}}]\label{conjecture3}
For every partition \(\mu\),
\(\nabla\omega\wtH_\mu[X;0,t^{-1}]\) is integral in $q$ and $t$ and Schur positive.  Moreover, if
\(\mu,\nu\vdash n\) satisfy \(\mu\unrhd\nu\), then
\[
\nabla\omega\wtH_\mu[X;0,t^{-1}]
-
\nabla\omega\wtH_\nu[X;0,t^{-1}]
\]
is Schur positive.
\end{conjecture}

We first recall a parking-function formula for \(\nabla\) on the modified Hall--Littlewood polynomials.  It is due to Haglund, Morse, and Zabrocki~\cite{HMZ12}.
It follows from their expression of \((-1)^{n-\ell(\mu)}\wtH_\mu[X;0,t]\) in terms of \(C_\mu[X;t]\) and from their compositional shuffle conjecture, which was proved by Carlsson and Mellit~\cite{CM18}.  In the conventions of this paper (see Section~\ref{subProofMain1}), it reads as follows.

\begin{theorem}[{Haglund--Morse--Zabrocki~\cite{HMZ12}, Carlsson--Mellit~\cite{CM18}}]\label{thmHMZ}
For any partition \(\mu\vdash n\), we have
\[
(-1)^{n-\ell(\mu)}
\nabla\wtH_\mu[X;0,t]
=
t^{n(\mu)+n-\ell(\mu)}
\sum_{\substack{\pi\in\LD_{n,n}\\
\comp(\pi)=\mu}}
q^{\area(\pi)}
t^{\dinv(\pi)}
F_{\ides(\pi)}[X].
\]
\end{theorem}

Since the Schur coefficients of \(\nabla C_\mu[X;t]\) lie in \(\mathbb N[q,t]\), as recalled at the beginning of this section (see also Proposition~\ref{QZ26schur}), Theorem~\ref{thmHMZ} implies Conjecture~\ref{conjecture2}.  The formula extends to arbitrary powers of \(\nabla\).

\begin{corollary}\label{thmmain1}
For any partition \(\mu\vdash n\) and every positive integer \(r\), we have
\begin{equation}
(-1)^{n-\ell(\mu)}
\nabla^r\wtH_\mu[X;0,t]
=
t^{n(\mu)+n-\ell(\mu)}
\sum_{\substack{\pi\in\LD_{nr,n}\\
\comp(\pi)=\mu}}
q^{\area(\pi)}
t^{\dinv(\pi)}
F_{\ides(\pi)}[X].
\end{equation}
\end{corollary}

Corollary~\ref{cor:hall-positivity} gives the
corresponding Schur positivity statement for every positive power of
\(\nabla\).  The proof begins with the identity of Haglund, Morse, and Zabrocki~\cite{HMZ12}, which expresses
\((-1)^{n-\ell(\mu)}\wtH_\mu[X;0,t]\) as a power of \(t\) times
\(C_\mu[X;t]\).  Applying \(\nabla^r\) and then Mellit's compositional
\((km,kn)\)-shuffle theorem gives the displayed parking-function formula.

We next record an explicit fundamental-quasisymmetric expansion for the first
symmetric function in Conjecture \ref{conjecture3}.  Haglund, Morse, and Zabrocki~\cite{HMZ12} conjectured this expansion in their conventions and normalization, and they showed that their conjecture is equivalent to their compositional shuffle conjecture.

\begin{theorem}[{Haglund--Morse--Zabrocki~\cite{HMZ12}, Carlsson--Mellit~\cite{CM18}}]\label{thmmain1.5}
For any partition \(\mu\vdash n\), we have
\begin{equation}
\begin{aligned}
\nabla\omega\wtH_\mu[X;0,t^{-1}]
&=
\sum_{\beta\leq\overleftarrow{\mu}}
\sum_{\substack{\pi\in\LD_{n,n}\\
\comp(\pi)=\beta}}
q^{\area(\pi)}
t^{
\doff_{\overleftarrow{\mu}}(\DP(\beta))
-
n(\mu)
+
\dinv(\pi)
}
F_{\ides(\pi)}[X].
\end{aligned}
\end{equation}
\end{theorem}

The \(C\)-expansion of
\(\omega\wtH_\mu[X;0,t^{-1}]\) due to Haglund, Morse, and Zabrocki~\cite{HMZ12}, together with the compositional shuffle theorem proved by Carlsson and Mellit \cite{CM18}, proves Theorem
\ref{thmmain1.5}.  A lower bound for \(\dinv(\pi)\) shows that every
displayed power of \(t\) is nonnegative.  Proposition~\ref{QZ26schur}
places the relevant Schur coefficients in
\(\mathbb N[q,t,t^{-1}]\), whereas the fundamental expansion together
with the integrality of the inverse Kostka matrix places them in
\(\mathbb Z[q,t]\).  Combining these two facts yields the following
result.

\begin{theorem}\label{thmmain2}
For any partitions \(\mu,\lambda\vdash n\), we have
\begin{equation}
\begin{aligned}
\left\langle
\nabla\omega\wtH_\mu[X;0,t^{-1}],
s_\lambda[X]
\right\rangle
&\in
\mathbb N[q,t].
\end{aligned}
\end{equation}
\end{theorem}

Thus Theorem \ref{thmmain2} proves the first assertion of Conjecture
\ref{conjecture3}.  To prove its dominance assertion, it is enough to
consider a dominance cover \(\mu\gtrdot\nu\).  We first identify
$t^{n(\nu)}(\wtH_\nu[X;0,t^{-1}]-\wtH_\mu[X;0,t^{-1}])$ with a Catalan function
$H_{i,\nu}[X;t]$ in the sense of Chen \cite{Chen10} and Blasiak, Morse, Pun, and Summers \cite{BMPS19}, where $i$ is the smallest index with $\mu_i\neq\nu_i$, and then use the representation-theoretic
interpretation of LLT polynomials introduced by Grojnowski and Haiman~\cite{GrojnowskiHaiman07} and
reformulated by Blasiak et al.~\cite{shuffleline,dens}.  This produces an LLT polynomial expression for
$\wtH_\mu[X;0,t^{-1}]-\wtH_\nu[X;0,t^{-1}]$.  The Catalanimal formula of LLT polynomials by Blasiak et al.~\cite[Corollary~8.4.2]{LLTcat}, together with their Cauchy identity~\cite[Theorem~6.1.3]{dens}, then gives a positive LLT expansion of $\nabla^r
\left(
\omega\wtH_\mu[X;0,t^{-1}]
-
\omega\wtH_\nu[X;0,t^{-1}]
\right)$.  Here a Catalanimal, in the sense of Blasiak et al.~\cite{LLTcat}, is the Weyl symmetrization of a monomial times a product of factors $1-qtz^\alpha$, divided by factors $1-qz^\alpha$ and $1-tz^\alpha$, where $\alpha$ runs over given sets of positive roots; see Section \ref{secCat}.

\begin{theorem}\label{thmintroNablaHallLLT}
Let $\mu,\nu\vdash n$ satisfy $\mu\gtrdot\nu$, and let
$\boldsymbol\nu^{\mu,\nu}$ be the tuple of skew diagrams defined in
Section~\ref{subLLTcat}.  For a tuple $\boldsymbol\tau$ of skew diagrams, let
$G_{\boldsymbol\tau}[X;t]$ denote its LLT polynomial, defined in Section~\ref{subLLT}.
\begin{enumerate}
\item The difference
\[
\omega\wtH_\mu[X;0,t^{-1}]
-\omega\wtH_\nu[X;0,t^{-1}]
\]
is the negative of an integral power of $t$ times
$G_{\boldsymbol\nu^{\mu,\nu}}[X;t]$.
\item For every positive integer $r$, the symmetric function
\[
\nabla^r\!\Bigl(
  \omega\wtH_\mu[X;0,t^{-1}]
  -\omega\wtH_\nu[X;0,t^{-1}]
\Bigr)
\]
is a finite sum of terms $q^at^bG_{\boldsymbol\tau}[X;t]$, where
$a\geq1$ and $b$ are integers and $\boldsymbol\tau$ is a tuple of
skew diagrams.  Theorem~\ref{thmNablaHallLLT} gives these terms
explicitly.
\end{enumerate}
\end{theorem}

Part (1) is the main new identity of this paper; see Theorem \ref{thmHallLLT}.  Since LLT polynomials are Schur positive by the theorem of Grojnowski and Haiman \cite{GrojnowskiHaiman07}, part (2) gives
\[
\left\langle
\nabla^r\omega\wtH_\mu[X;0,t^{-1}]
-
\nabla^r\omega\wtH_\nu[X;0,t^{-1}],
s_\lambda[X]
\right\rangle
\in
\mathbb N[q,t,t^{-1}].
\]
On the other hand, Theorem \ref{thmmain2}, together with the integrality of
\(\nabla\) proved by Bergeron et al.~\cite{BGHT99}, gives integrality over
\(\mathbb Z[q,t]\).  Combining these two facts eliminates the negative powers of
\(t\) and yields the following stronger form of the dominance assertion.

\begin{theorem}\label{thmmain3}
For any partitions \(\mu,\nu,\lambda\vdash n\) with
\(\mu\unrhd\nu\), and every positive integer \(r\), we have
\begin{equation}\label{eqmain3}
\left\langle
\nabla^r\omega\wtH_\mu[X;0,t^{-1}]
-
\nabla^r\omega\wtH_\nu[X;0,t^{-1}],
s_\lambda[X]
\right\rangle
\in
\mathbb N[q,t].
\end{equation}
\end{theorem}

For \(r=1\), Theorem \ref{thmmain3} proves the second assertion of
Conjecture \ref{conjecture3}.  Taking $\nu=(1^n)$ in Theorem \ref{thmmain3}, and using $\omega\wtH_{(1^n)}[X;0,t^{-1}]=C_{(1^n)}[X;t]$ together with Proposition \ref{QZ26schur}, we also extend Theorem \ref{thmmain2} to all powers of $\nabla$.

\begin{corollary}\label{coromain4}
For any partitions \(\mu,\lambda\vdash n\) and every positive integer \(r\), we have
\begin{equation}
\left\langle
\nabla^r\omega\wtH_\mu[X;0,t^{-1}],
s_\lambda[X]
\right\rangle
\in
\mathbb N[q,t].
\end{equation}
\end{corollary}

For coprime positive $(r,s)$, Section~\ref{sec:schiffmann-coprime}
extends Corollary~\ref{thmmain1}, Theorem~\ref{thmmain1.5},
Theorem~\ref{thmNablaHallLLT}, Theorem~\ref{thmmain3}, and Corollary~\ref{coromain4}
to the Schiffmann algebra setting.

The paper is organized as follows.  Section 2 reviews the necessary background,
including symmetric functions, LLT polynomials, and the combinatorial objects
used throughout the paper.  Section 3 recalls the \(C\) operators of
Haglund, Morse, and Zabrocki~\cite{HMZ12} and establishes their connection with Hall--Littlewood
polynomials.  Section 4 first relates Hall--Littlewood polynomials to
Catalan functions, and then relates the latter to LLT
polynomials.  Section 5 uses the Catalanimal construction and the
Cauchy identity developed by Blasiak et al.~\cite{LLTcat,dens} to complete the proofs of
Theorems~\ref{thmintroNablaHallLLT} and~\ref{thmmain3}, as well as Corollary~\ref{coromain4}.
Section 6 gives these generalizations to the Schiffmann algebra setting.

\section{Background}
In Section~\ref{sub1}, we review the relevant material on symmetric functions.
For further details, we refer the reader to the books of Haglund \cite{haglund-book} and Macdonald \cite{Mac95}.
In Section~\ref{subLLT}, we adopt the notation of
Blasiak et al.~\cite{shuffleline,LLTcat,dens} to define LLT polynomials and recall the
properties that will be used later.
In Section~\ref{sub3}, we introduce the combinatorial objects used in this
paper.  For further details, see Bergeron, Garsia, Leven, and Xin~\cite{BGLX16}.

\subsection{Symmetric functions}\label{sub1}
A partition is a weakly decreasing sequence of positive integers $\lambda=(\lambda_1,\lambda_2,\cdots,\lambda_m)$ 
satisfying $\lambda_1\geq\lambda_2\geq\cdots\geq\lambda_m>0$.
If $|\lambda|=\lambda_1+\cdots+\lambda_m=n$, we write $\lambda\vdash n$.
We denote by $\ell(\lambda)$ the length of $\lambda$, and by $\lambda'$ its conjugate partition.
For a partition $\lambda=(\lambda_1,\cdots,\lambda_m)$, we write $\overleftarrow\lambda=(\lambda_m,\cdots,\lambda_1)$.
For partitions $\mu,\nu\vdash n$, write $\mu\unrhd\nu$ if $$\sum_{i=1}^{k}\mu_i\geq\sum_{i=1}^{k}\nu_i$$ for every $k\geq1$, and write $\mu\gtrdot\nu$ if $\mu$ covers $\nu$ in this order.

A composition of $n$ is a sequence $\alpha=(\alpha_1,\cdots,\alpha_m)$ of positive integers with $|\alpha|=\alpha_1+\cdots+\alpha_m=n$; we write $\alpha\models n$ and $\ell(\alpha)=m$. 
Its descent set is $\operatorname{Des}(\alpha)=\{\alpha_1,\alpha_1+\alpha_2,\cdots,\alpha_1+\cdots+\alpha_{m-1}\}$. 
For $\alpha,\beta\models n$, we write $\beta\leq\alpha$ if $\operatorname{Des}(\alpha)\subseteq\operatorname{Des}(\beta)$. 
Thus, if $\beta\leq\alpha$, then $\beta$ is uniquely of the form $\beta=\beta^{(1)}\mid\cdots\mid\beta^{(m)}$, where $\beta^{(k)}\models\alpha_k$.  

Let $\Lambda=\bigoplus_{n\geq 0}\Lambda^n$ be the ring of symmetric functions with coefficients in $\mathbb{Q}(q,t)$.
In this paper, we use the standard bases $m_\lambda, e_\lambda, h_\lambda, p_\lambda, s_\lambda$, 
which denote the monomial, elementary, complete homogeneous, power-sum, and Schur symmetric functions, respectively.
We also write $F_S(X)$ for the fundamental quasisymmetric function, where $S\subseteq \{1,2,\cdots, n-1\}$ and 
$$F_S(X) = \sum_{\substack{i_1\leq \cdots \leq i_n \\ j\in S \Rightarrow i_j < i_{j+1}}} x_{i_1}\cdots x_{i_n}.$$

The Hall inner product is defined by $\langle p_\lambda,p_\mu\rangle=z_\lambda\delta_{\lambda,\mu}$, where $z_\lambda=\prod_{i\geq1}i^{m_i}m_i!$ and $m_i$ is the number of parts of $\lambda$ equal to $i$. 
We denote by $\omega$ the standard involution on $\Lambda$, which acts by $\omega(s_\lambda)=s_{\lambda'}$.
For any $f\in\Lambda$, we denote by $f^\perp$ the Hall-adjoint operator to multiplication by $f$.
We write $\mathbb N=\{0,1,2,\cdots\}$.  A symmetric function $f$ is Schur positive if every coefficient $\langle f,s_\lambda\rangle$ lies in $\mathbb{N}[q,t]$.

We also adopt the standard plethystic notation.  Throughout the paper, we set $M=(1-q)(1-t)$.  For a formal sum of variables $E$, 
the expression $f[E]$ is obtained by replacing each $p_r$ in $f$ with $p_r[E]$, where $p_r[E]$ is obtained from $E$ by replacing every variable, including $q$ and $t$, by its $r$th power.
We also define $$\Omega[E]=\exp\left(\sum_{r\geq 1}\frac{p_r[E]}{r}\right)=\sum_{n\geq 0}h_n[E].$$

We use the modified Macdonald polynomials $\widetilde H_\mu[X;q,t]=t^{n(\mu)}J_\mu[X/(1-t^{-1});q,t^{-1}]$, as normalized by Bergeron et al.~\cite{BGHT99} and Haglund \cite{haglund-book}, where $J_\mu$ is Macdonald's integral form \cite[Chapter VI, (8.3)]{Mac95} and $n(\mu)=\sum_i(i-1)\mu_i$.  Their specializations $\widetilde H_\mu[X;0,t]$ and $\widetilde H_\mu[X;0,t^{-1}]$ are the modified Hall--Littlewood polynomials occurring below.  With this normalization, $\widetilde H_\mu[X;q,t]=\widetilde H_{\mu'}[X;t,q]$ (see Haglund \cite{haglund-book}), 
and the nabla operator is the $\mathbb Q(q,t)$-linear operator determined by 
$$\nabla\widetilde H_\mu[X;q,t]=T_\mu \widetilde H_\mu[X;q,t], \text{ where } T_\mu = t^{n(\mu)}q^{n(\mu')}.$$

\subsection{LLT polynomials}\label{subLLT}
For a box $a=(i,j)$ in column $i$ and row $j$ of a Young diagram in French notation, its ordinary content is $c(a)=i-j$. 
Let $\boldsymbol\mu=(\mu^{(1)},\cdots,\mu^{(k)})$ be a tuple of skew diagrams, and fix $\epsilon>0$ with $k\epsilon<1$.  If $a\in\mu^{(h)}$, its adjusted content is $c_e(a)=c(a)+h\epsilon$.  An adjusted diagonal is the set of boxes with a fixed adjusted content.  The LLT reading order lists boxes by increasing adjusted content, breaking ties from southwest to northeast.  An attacking pair is a pair $(a,b)$ with $a<b$ in this order and $0<c_e(b)-c_e(a)<1$.  Write $A(\boldsymbol\mu)$ for the total number of attacking pairs.

A semistandard tableau on $\boldsymbol\mu$ is a filling by positive integers which weakly increases along rows and strictly increases from bottom to top in every column of each component.  Write $\operatorname{SSYT}(\boldsymbol\mu)$ for the set of such tableaux.  For $T\in\operatorname{SSYT}(\boldsymbol\mu)$, let $\inv(T)$ be the number of attacking pairs $(a,b)$ satisfying $T(a)>T(b)$, and put $x^T=\prod_{a\in\boldsymbol\mu}x_{T(a)}$.  The associated LLT polynomial is
$$
G_{\boldsymbol\mu}[X;t]
=
\sum_{T\in\operatorname{SSYT}(\boldsymbol\mu)}
t^{\inv(T)}x^T.
$$
Lascoux, Leclerc, and Thibon \cite{LLT97} showed that it is a symmetric function.  Grojnowski and Haiman \cite{GrojnowskiHaiman07} proved it is Schur positive with coefficients in $\mathbb N[t]$.

For every component $\mu^{(h)}$, let $(\mu^{(h)})^R$ be the $180^\circ$-rotation of its transpose, so that a box and its corresponding box have the same ordinary content.  Set $\boldsymbol\mu^R=((\mu^{(1)})^R,\cdots,(\mu^{(k)})^R)$; clearly $A(\boldsymbol\mu^R)=A(\boldsymbol\mu)$.
We also define the transpose--reversal of $\boldsymbol\mu$ by
\[
\boldsymbol\mu^\dagger
=
\bigl((\mu^{(k)})',\cdots,(\mu^{(1)})'\bigr).
\]

\begin{example}\label{ex:llt-rotation}
Take the three-component tuple
\[
\boldsymbol\mu=
\bigl((4,4,2,1)/(2,1),\ (3,3,2)/(1,1),\
(4,4,3,1)/(2,2,1)\bigr).
\]
The three diagrams in Figure~\ref{fig:llt-rotation} show
$\boldsymbol\mu$, $\boldsymbol\mu^R$, and
$\boldsymbol\mu^\dagger$.  Within each diagram, the components
ascend from left to right.
Boxes are drawn in French notation.
\begin{center}
\begin{tikzpicture}[x=.31cm,y=.31cm]
\begin{scope}
\begin{scope}
\clip (-.5,-.5) rectangle (12.5,12.5);
\draw[blue,densely dashed,semithick] (-.5,-.5)--(12.5,12.5);
\end{scope}
\foreach \x/\y in {0/0,1/0,0/1}{\draw[black,densely dashed] (\x,\y) rectangle ++(1,1);}
\foreach \x/\y in {2/0,3/0,1/1,2/1,3/1,0/2,1/2,0/3}{\draw (\x,\y) rectangle ++(1,1);}
\begin{scope}[shift={(4,4)}]
\foreach \x/\y in {0/0,0/1}{\draw[black,densely dashed] (\x,\y) rectangle ++(1,1);}
\foreach \x/\y in {1/0,2/0,1/1,2/1,0/2,1/2}{\draw (\x,\y) rectangle ++(1,1);}
\end{scope}
\begin{scope}[shift={(8,8)}]
\foreach \x/\y in {0/0,1/0,0/1,1/1,0/2}{\draw[black,densely dashed] (\x,\y) rectangle ++(1,1);}
\foreach \x/\y in {2/0,3/0,2/1,3/1,1/2,2/2,0/3}{\draw (\x,\y) rectangle ++(1,1);}
\end{scope}
\node at (6,-2) {$(a)\ \boldsymbol\mu$};
\end{scope}
\begin{scope}[shift={(15,0)}]
\begin{scope}
\clip (-.5,-.5) rectangle (12.5,12.5);
\draw[blue,densely dashed,semithick] (-.5,-.5)--(12.5,12.5);
\end{scope}
\foreach \x/\y in {0/0,1/0,0/1,1/1,0/2}{\draw[black,densely dashed] (\x,\y) rectangle ++(1,1);}
\foreach \x/\y in {2/0,3/0,2/1,3/1,1/2,2/2,0/3,1/3}{\draw (\x,\y) rectangle ++(1,1);}
\begin{scope}[shift={(4,4)}]
\draw[black,densely dashed] (0,0) rectangle ++(1,1);
\foreach \x/\y in {1/0,2/0,0/1,1/1,2/1,0/2}{\draw (\x,\y) rectangle ++(1,1);}
\end{scope}
\begin{scope}[shift={(8,8)}]
\foreach \x/\y in {0/0,1/0,0/1,0/2}{\draw[black,densely dashed] (\x,\y) rectangle ++(1,1);}
\foreach \x/\y in {2/0,3/0,1/1,2/1,3/1,1/2,0/3}{\draw (\x,\y) rectangle ++(1,1);}
\end{scope}
\node at (6,-2) {$(b)\ \boldsymbol\mu^R$};
\end{scope}
\begin{scope}[shift={(30,0)}]
\begin{scope}
\clip (-.5,-.5) rectangle (12.5,12.5);
\draw[blue,densely dashed,semithick] (-.5,-.5)--(12.5,12.5);
\end{scope}
\foreach \x/\y in {0/0,1/0,2/0,0/1,1/1}{\draw[black,densely dashed] (\x,\y) rectangle ++(1,1);}
\foreach \x/\y in {3/0,2/1,0/2,1/2,2/2,0/3,1/3}{\draw (\x,\y) rectangle ++(1,1);}
\begin{scope}[shift={(4,4)}]
\foreach \x/\y in {0/0,1/0}{\draw[black,densely dashed] (\x,\y) rectangle ++(1,1);}
\foreach \x/\y in {2/0,0/1,1/1,2/1,0/2,1/2}{\draw (\x,\y) rectangle ++(1,1);}
\end{scope}
\begin{scope}[shift={(8,8)}]
\foreach \x/\y in {0/0,1/0,0/1}{\draw[black,densely dashed] (\x,\y) rectangle ++(1,1);}
\foreach \x/\y in {2/0,3/0,1/1,2/1,0/2,1/2,0/3,1/3}{\draw (\x,\y) rectangle ++(1,1);}
\end{scope}
\node at (6,-2) {$(c)\ \boldsymbol\mu^\dagger$};
\end{scope}
\end{tikzpicture}

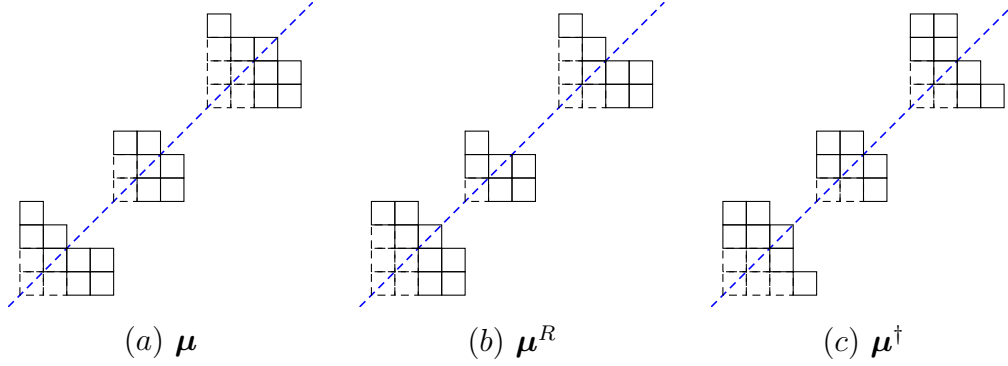
\captionof{figure}{A skew tuple and its transforms.}\label{fig:llt-rotation}
\end{center}
The second and third diagrams have component shapes
\[
\begin{aligned}
\boldsymbol\mu^R
&=\bigl((4,4,3,2)/(2,2,1),\ (3,3,1)/(1),\
(4,4,2,1)/(2,1,1)\bigr),\\
\boldsymbol\mu^\dagger
&=\bigl((4,3,3,2)/(3,2),\ (3,3,2)/(2),\
(4,3,2,2)/(2,1)\bigr).
\end{aligned}
\]
For the first and third components, $R$ sends a box $(i,j)$ to
$(5-j,5-i)$; for the second it sends $(i,j)$ to $(4-j,4-i)$.
\end{example}

The involution $\omega$ relates an LLT polynomial to the rotated
skew tuple as follows.
\begin{proposition}[{\cite[Proposition~4.1.6]{shuffleline}}]\label{propomegaLLT}
For every tuple $\boldsymbol\mu$ of skew diagrams, we have
$$
\omega G_{\boldsymbol\mu}[X;t]
=
t^{A(\boldsymbol\mu)}
G_{\boldsymbol\mu^R}[X;t^{-1}].
$$
\end{proposition}
The transpose--reversal gives another form of this identity.
\begin{proposition}[{\cite{HHL05}}]\label{propomegaLLTdagger}
For every tuple $\boldsymbol\mu$ of skew diagrams, we have
$$
\omega G_{\boldsymbol\mu}[X;t]
=
t^{A(\boldsymbol\mu)}
G_{\boldsymbol\mu^\dagger}[X;t^{-1}].
$$
Consequently,
$$
A(\boldsymbol\mu^\dagger)=A(\boldsymbol\mu)
\qquad\text{and}\qquad
G_{\boldsymbol\mu^\dagger}[X;t]=G_{\boldsymbol\mu^R}[X;t].
$$
\end{proposition}

For $r\geq1$, let $\boldsymbol\nu_r$ be the $r$-stretching $\boldsymbol\nu(r,\mathbf o)$ of $\boldsymbol\nu$ with constant offsets $\mathbf o=(0,\cdots,0)$, defined by Blasiak et al.~\cite[Section~8.1]{LLTcat}:
each box of ordinary content $c$ is replaced by $r$ boxes in the same column, of contents $rc,rc-1,\cdots,rc-r+1$.
The result $\boldsymbol\nu_r$ is again a tuple of skew diagrams; we call it the constant-offset $r$-stretch of $\boldsymbol\nu$.
Define $A_r(\boldsymbol\nu)$ by assigning weight $r$ to each attacking pair of $\boldsymbol\nu$ of equal ordinary content and weight $1$ to every other attacking pair of $\boldsymbol\nu$.  Note that $A_r(\boldsymbol\nu)$ is computed on $\boldsymbol\nu$, not on $\boldsymbol\nu_r$.

\subsection{Combinatorial objects}\label{sub3}
For positive integers $A$ and $B$, an $(A,B)$-Dyck path is a lattice path from $(0,0)$ to $(A,B)$, with north and east steps, which remains weakly above the line $y=(B/A)x$.  We write $\D_{A,B}$ for the set of such paths.

We write $m'$ and $n'$ for the path parameters to avoid confusion with
the length $m$ and size $n$ of a composition in Section~\ref{sub1}.  Throughout, $m'$ and $n'$ are coprime positive integers.

A $(km',kn')$-parking function $\pi$ consists of a path $D\in\D_{km',kn'}$
together with an assignment of the labels $1,\cdots,kn'$ to its north
steps, such that the labels in each column are strictly increasing from
bottom to top.  We write $\LD_{km',kn'}$ for the set of these labeled paths.
The composition $\comp(\pi)$ records, starting from $(0,0)$, the numbers of north steps
between consecutive contacts of $D$ with the diagonal; every such contact
is of the form $(dm',dn')$ for some $0\leq d\leq k$.

Let $u_i$ be the $x$-coordinate of the lower endpoint of the $i$th north step of $D$, counted from bottom to top, and let $w_i$ be its label.  Then $0=u_1\leq\cdots\leq u_{kn'}$, $u_i\leq\frac{m'}{n'}(i-1)$, and $u_i=u_{i+1}$ implies $w_i<w_{i+1}$.  Set $\rank_i=\rank(w_i)=km'(i-1)-kn'u_i+u_i/(km'+1)$.  We define $$\tdinv(\pi)=\#\{(i,j):w_i<w_j,\ \rank_i<\rank_j<\rank_i+km'\}.$$  The pairs $(i,j)$ counted here are called the temporary diagonal inversions of $\pi$.

Let $\lambda(D)$ denote the Ferrers diagram in the $km'\times kn'$ rectangle above $D$.  We define
\[
  \pdinv(D)=
  \#\left\{
  c\in\lambda(D):
  \frac{a(c)}{\ell(c)+1}\leq\frac{m'}{n'}
  <
  \frac{a(c)+1}{\ell(c)}
  \right\},
\]
where $a(c)$ and $\ell(c)$ are the arm and leg of $c$ in $\lambda(D)$, that is, the numbers of cells of $\lambda(D)$ in the same row strictly to the right of $c$ and in the same column strictly below $c$,
and $(a(c)+1)/\ell(c)$ is interpreted as $+\infty$ when $\ell(c)=0$.
Let $\maxtdinv(D)$ denote the maximum possible value of $\tdinv$ over all
standard positive labelings of the fixed path $D$.  The $\dinv$ statistic is
$$\dinv(\pi)=\tdinv(\pi)+\pdinv(D)-\maxtdinv(D).$$

The area is the number of full cells between $D$ and the diagonal, that is, $$\area(\pi)=\area(D)=\frac{km'kn'-km'-kn'+k}{2}-\sum_{i=1}^{kn'}u_i.$$
Reading the labels in decreasing rank gives the diagonal word, and $\ides(\pi)=\{a\in[kn'-1]:a+1\text{ precedes }a\}$, where $[m]=\{1,\cdots,m\}$.

In the square case $(n,n)$, for $\alpha=(\alpha_1,\cdots,\alpha_m)\models n$, let $\DP(\alpha)$ be the Dyck path $N^{\alpha_1}E^{\alpha_1}\cdots N^{\alpha_m}E^{\alpha_m}$, where $N$ and $E$ denote north and east steps.  The area sequence of a square Dyck path $D$ is $(a_1,\cdots,a_n)$, where $a_i=i-1-u_i$ is the number of full cells in the $i$th row between $D$ and the diagonal; for square Dyck paths $D$ and $D'$, we write $D\leq D'$ if $a_i(D)\leq a_i(D')$ for every $i$.  The area sequence of $\DP(\alpha)$ is $(0,1,\cdots,\alpha_1-1\mid\cdots\mid0,1,\cdots,\alpha_m-1)$.  For every square Dyck path $D$, we have $D\leq\DP(\alpha)$ if and only if $\comp(D)\leq\alpha$. 

If $D\leq\DP(\alpha)$ in componentwise area-sequence order, write the area sequence of $D$ as $A_1(D)\mid\cdots\mid A_m(D)$, where $A_k(D)$ has length $\alpha_k$; we call $A_k(D)$ the $k$th $\alpha$-block of $D$.  Let $z_k(D)$ be the number of zero entries of $A_k(D)$.  Following Haglund, Morse, and Zabrocki~\cite[Section~2.1]{HMZ12}, define $$\doff_\alpha(D)=\sum_{k=1}^{m}(m-k)z_k(D).$$

\begin{example}\label{ex:parking-statistics-doff}
Figure~\ref{fig:parking-statistics-doff} shows a labeled $(6,8)$-path
and a labeled $(6,6)$-path.  The shaded cells are those counted by
$\pdinv(D)$, and each number beside a north step is its label.
\begin{center}
\begin{minipage}[t]{.48\textwidth}\centering
\begin{tikzpicture}[x=.43cm,y=.43cm]
\foreach \x/\y in {2/4,1/5,2/5,0/6,1/6,0/7}
  {\fill[blue!18] (\x,\y) rectangle ++(1,1);}
\foreach \x in {0,...,6}{\draw[gray!55,thin] (\x,0)--(\x,8);}
\foreach \y in {0,...,8}{\draw[gray!55,thin] (0,\y)--(6,\y);}
\draw[densely dashed,blue!45] (0,0)--(6,8);
\draw[very thick,blue!85!black]
  (0,0)--(0,4)--(3,4)--(3,8)--(6,8);
\foreach \y/\lab in {0/2,1/4,2/6,3/8}
  {\node[right,font=\scriptsize] at (0,\y+0.5) {\lab};}
\foreach \y/\lab in {4/1,5/3,6/5,7/7}
  {\node[right,font=\scriptsize] at (3,\y+0.5) {\lab};}
\end{tikzpicture}\\[-.2em]
$(a)$ Rational $(6,8)$-path: $D=N^4E^3N^4E^3$
\end{minipage}\hfill
\begin{minipage}[t]{.48\textwidth}\centering
\begin{tikzpicture}[x=.43cm,y=.43cm]
\foreach \x/\y in {1/3,2/3,0/4,1/4,0/5}
  {\fill[blue!18] (\x,\y) rectangle ++(1,1);}
\foreach \x in {0,...,6}{\draw[gray!55,thin] (\x,0)--(\x,6);}
\foreach \y in {0,...,6}{\draw[gray!55,thin] (0,\y)--(6,\y);}
\draw[densely dashed,blue!45] (0,0)--(6,6);
\draw[very thick,blue!85!black]
  (0,0)--(0,3)--(3,3)--(3,6)--(6,6);
\foreach \y/\lab in {0/2,1/4,2/6}
  {\node[right,font=\scriptsize] at (0,\y+0.5) {\lab};}
\foreach \y/\lab in {3/1,4/3,5/5}
  {\node[right,font=\scriptsize] at (3,\y+0.5) {\lab};}
\end{tikzpicture}\\[-.2em]
$(b)$ Square $(6,6)$-path: $D=\DP(3,3)$
\end{minipage}
\captionof{figure}{Examples of parking functions and their statistics.}
\label{fig:parking-statistics-doff}
\end{center}

For the $(6,8)$-path in Figure~\ref{fig:parking-statistics-doff}(a),
the six shaded cells give $\pdinv(D)=6$.  The rank condition permits
the seven ordered pairs
\(
\{(1,5),(2,6),(3,7),(4,8),(5,2),(6,3),(7,4)\}.
\)
The displayed labeling counts the last three, so $\tdinv(\pi)=3$.
A legal labeling counts all seven, giving $\maxtdinv(D)=7$ and
$\dinv(\pi)=3+6-7=2$.  Here $\area(\pi)=6$,
$\comp(\pi)=(4,4)$, and $\ides(\pi)=\{2,4,6\}$.

For the $(6,6)$-path in Figure~\ref{fig:parking-statistics-doff}(b),
the five shaded cells give $\pdinv(D)=5$.  The five pairs permitted
by the rank condition are
\(
\{(1,4),(2,5),(3,6),(4,2),(5,3)\}.
\)
The displayed labeling counts the last two, while a legal labeling
counts all five.  Thus $\tdinv(\pi)=2$,
$\maxtdinv(D)=\pdinv(D)=5$, and $\dinv(\pi)=2$, in agreement with
Blasiak et al.~\cite[Proposition~8.4.10]{nonsymshuffle}.
For $\mu=\alpha=(3,3)$, the area sequence
$(0,1,2\mid0,1,2)$ has one zero in each $\alpha$-block.  Hence
$\doff_{(3,3)}(D)=1$ and
$\doff_{(3,3)}(D)-n(\mu)+\dinv(\pi)=1-3+2=0$.
Here $\area(\pi)=6$, $\comp(\pi)=(3,3)$, and
$\ides(\pi)=\{2,4\}$.
\end{example}

\section{The C operators and Hall--Littlewood polynomials}
In this section, we recall the \(C\) operators of Haglund, Morse, and Zabrocki~\cite{HMZ12} and relate them to Hall--Littlewood polynomials.
Combining the compositional shuffle theorem proved by Carlsson and
Mellit~\cite{CM18} and Mellit~\cite{Mel21} with the Schur positivity
of LLT polynomials established by Grojnowski and
Haiman~\cite{GrojnowskiHaiman07}, we complete the proofs of
Corollary~\ref{thmmain1}, Theorem~\ref{thmmain1.5}, and Theorem~\ref{thmmain2}.

\subsection{The C operators}
Haglund, Morse, and Zabrocki~\cite[Equation~(3.1)]{HMZ12} defined the $C$ operators as follows,
\begin{equation}\label{eq:Ca}
    C_a(q)f[X]
    = \left(-\frac{1}{q}\right)^{a-1}
      \langle z^a \rangle
      f\left[X-\frac{1-1/q}{z}\right]
      \sum_{r\geq 0} h_r[X] z^r,
\end{equation}
where $\langle z^a\rangle G(z)$ denotes the coefficient of $z^a$ in the formal Laurent series $G(z)$.

For any composition $\alpha = (\alpha_1, \cdots, \alpha_m)$, we use the notation
$$
  C_\alpha[X;q] = C_{\alpha_1}(q) C_{\alpha_{2}}(q) \cdots C_{\alpha_m}(q)(1).
$$
For the empty composition $\emptyset$, we adopt the convention that $C_\emptyset(q)f = f$.  Note that $C_a(q)$ maps $\Lambda^n$ to $\Lambda^{n+a}$, so $C_\alpha[X;q] \in \Lambda^{|\alpha|}$.

Following Haglund, Morse, and Zabrocki~\cite{HMZ12}, we can use the functions $C_\alpha$ to express the Hall--Littlewood polynomials as follows.

\begin{proposition}[{\cite{HMZ12}}]
Let \(\mu \) be a partition of \(n\).  Then the following two identities hold:
\begin{enumerate}
\item
\begin{equation}\label{eqhallC}
(-1)^{n-\ell(\mu)} \widetilde{H}_\mu[X;0,t] = t^{n(\mu)+n-\ell(\mu)} C_{\mu}[X;t],
\end{equation}
and moreover \(n(\mu)+n-\ell(\mu) \geq 0\);
\item
\begin{equation}\label{eqhallomegaC}
\omega \widetilde{H}_\mu[X;0,t^{-1}] = \sum_{\beta \leq \overleftarrow{\mu}} t^{\doff_{\overleftarrow{\mu}}(\DP(\beta)) - n(\mu)} C_\beta[X;t],
\end{equation}
where the sum is over all compositions \(\beta\) satisfying \(\beta \leq \overleftarrow{\mu}\).
\end{enumerate}
\end{proposition}

\begin{proof}
(1) By Haglund, Morse, and Zabrocki~\cite[pp.~828, 830]{HMZ12},
\[
\widetilde{H}_\mu[X; 0,t] = (-1)^{n-\ell(\mu)} t^{n(\mu)+n-\ell(\mu)} C_\mu[X;t],
\]
which rearranges immediately to \eqref{eqhallC}.  The exponent \(n(\mu)+n-\ell(\mu)\) is nonnegative, with equality if and only if \(\mu=(1)\).

(2) By Haglund, Morse, and Zabrocki~\cite[pp.~828, 830]{HMZ12} and
\cite[Theorem~5.6]{HMZ12}, we have
\[
\omega \widetilde{H}_\mu[X;0,t^{-1}] = t^{-n(\mu)} \sum_{\beta \leq \overleftarrow{\mu}} t^{\doff_{\overleftarrow{\mu}}(\DP(\beta))} C_\beta[X;t] = \sum_{\beta \leq \overleftarrow{\mu}} t^{\doff_{\overleftarrow{\mu}}(\DP(\beta)) - n(\mu)} C_\beta[X;t],
\]
which is exactly identity \eqref{eqhallomegaC}.
\end{proof}

\subsection{Proof of Corollary \ref{thmmain1}}\label{subProofMain1}
In Haglund, Morse, and Zabrocki~\cite{HMZ12} and Carlsson and Mellit~\cite{CM18}, the area is weighted by $t$, the statistic $\dinv$ by $q$, and the operators $C_a$ carry the parameter $q$.  Since $\wtH_\mu[X;q,t]=\wtH_{\mu'}[X;t,q]$ and $T_\mu$ becomes $T_{\mu'}$ when $q$ and $t$ are interchanged, the operator $\nabla$ commutes with this interchange, so identities in their convention remain valid after $q$ and $t$ are interchanged.  We use the interchanged form, in which $C_\alpha$ carries the parameter $t$, because Conjectures~\ref{conjecture2} and~\ref{conjecture3} concern $\wtH_\mu[X;0,t]$ and $\wtH_\mu[X;0,t^{-1}]$.
The compositional $(km,kn)$-shuffle theorem proved by Mellit \cite{Mel21} takes the following form in our conventions.
\begin{theorem}[{\cite{Mel21}}]\label{kmknshuffle}
    For any composition $\alpha \models n$ and positive integer $r$, we have
    \begin{equation}
        \nabla^r C_\alpha[X;t] = \sum_{\substack{\pi \in \LD_{nr,n} \\ \comp(\pi) = \alpha}} q^{\area(\pi)} t^{\dinv(\pi)} F_{\ides(\pi)}.
    \end{equation}
\end{theorem}

For \(r=1\), Theorem~\ref{kmknshuffle} is the compositional shuffle theorem of Carlsson and Mellit~\cite{CM18}, conjectured by Haglund, Morse, and Zabrocki~\cite{HMZ12}, in our conventions.
The formula also yields Schur positivity.
\begin{proposition}[{\cite[Proposition~2.3]{QZ26}}]\label{QZ26schur}
    For any integer $r\geq 1$, every composition $\alpha \models n$, and any partition $\lambda \vdash n$, we have
    \[
    \langle \nabla^r C_\alpha[X;t], s_\lambda \rangle \in \mathbb{N}[q,t].
    \]
\end{proposition}

Applying the operator $\nabla^r$ to both sides of identity \eqref{eqhallC} and substituting the combinatorial expansion from Theorem \ref{kmknshuffle}, we derive the following analogous statement for modified Hall--Littlewood polynomials; for \(r=1\), it is Theorem~\ref{thmHMZ}.

\begin{corollary}[$=$ Corollary~\ref{thmmain1}]\label{kmknhall}
    For any partition $\mu \vdash n$ and positive integer $r$, we have
    \begin{equation}
        (-1)^{n-\ell(\mu)} \nabla^r \widetilde{H}_\mu[X;0,t] = t^{n(\mu)+n-\ell(\mu)} \sum_{\substack{\pi \in \LD_{nr,n} \\ \comp(\pi) = \mu}} q^{\area(\pi)} t^{\dinv(\pi)} F_{\ides(\pi)}.
    \end{equation}
\end{corollary}

\begin{corollary}\label{cor:hall-positivity}
    For any integer $r\geq 1$, any partition $\mu \vdash n$, and any partition $\lambda \vdash n$, we have
    \begin{equation*}
        (-1)^{n-\ell(\mu)} \langle \nabla^r \widetilde{H}_\mu[X;0,t], s_\lambda \rangle \in \mathbb{N}[q,t].
    \end{equation*}
\end{corollary}

\begin{proof}
By \eqref{eqhallC},
\[
(-1)^{n-\ell(\mu)}\nabla^r\widetilde{H}_\mu[X;0,t]
=
t^{n(\mu)+n-\ell(\mu)}\nabla^r C_\mu[X;t].
\]
Taking the Hall inner product with $s_\lambda$ and applying
Proposition~\ref{QZ26schur} proves the claim, since
$n(\mu)+n-\ell(\mu)\geq0$.
\end{proof}

\subsection{Proof of Theorem \ref{thmmain2}}
Using \eqref{eqhallomegaC} and Theorem \ref{kmknshuffle}, we obtain the following theorem.

\begin{theorem}[$=$ Theorem~\ref{thmmain1.5}]\label{thmnablawhall}
For any partition $\mu\vdash n$, we have
\begin{equation}\label{eqnablawhall}
\begin{aligned}
\nabla\omega\wtH_\mu[X;0,t^{-1}]=
\sum_{\beta\leq\overleftarrow{\mu}}
\sum_{\substack{\pi\in\LD_{n,n}\\
\comp(\pi)=\beta}}
q^{\area(\pi)}
t^{\doff_{\overleftarrow{\mu}}(\DP(\beta))
-n(\mu)+\dinv(\pi)}
F_{\ides(\pi)}.
\end{aligned}
\end{equation}
\end{theorem}

\begin{lemma}\label{lemmadoff}
Let $\mu\vdash n$, let $\beta\leq\overleftarrow{\mu}$, and let $\pi\in\LD_{n,n}$ satisfy $\comp(\pi)=\beta$.  Then
$\doff_{\overleftarrow{\mu}}(\DP(\beta))-n(\mu)+\dinv(\pi)\geq 0$.
\end{lemma}

\begin{proof}
For convenience, define
$\delta_\mu(\beta):=n(\mu)-\doff_\alpha(\DP(\beta))$, where
$\alpha:=\overleftarrow{\mu}$ and $\beta$ is a refinement of
$\overleftarrow{\mu}$.
Write
$\overleftarrow{\mu}=\alpha=(\alpha_1,\cdots,\alpha_m)$; then
$\alpha_1\leq\cdots\leq\alpha_m$.  Since $\beta$ is a refinement of
$\alpha$, we may write
\begin{equation*}
\begin{aligned}
\beta
&=
\left(\beta^{(1)}\mid\beta^{(2)}
\mid\cdots\mid\beta^{(m)}\right),
\qquad
\beta^{(k)}\models\alpha_k.
\end{aligned}
\end{equation*}
Let $\ell_k:=\ell\bigl(\beta^{(k)}\bigr)$.  Clearly,
$1\leq \ell_k\leq\alpha_k$.

For the path $\DP(\beta)$, the $k$th $\alpha$-block contains exactly
$\ell_k$ zeros.  Moreover, by the definition of $n(\mu)$ and the equality
$\alpha=\overleftarrow{\mu}$, we have
$n(\mu)=\sum_{k=1}^{m}(m-k)\alpha_k$.  It follows that
\begin{align*}
\delta_\mu(\beta)
&=
n(\mu)-\doff_\alpha(\DP(\beta))
\\
&=
\sum_{k=1}^{m}(m-k)\alpha_k
-
\sum_{k=1}^{m}(m-k)\ell_k
\\
&=
\sum_{k=1}^{m}
(m-k)(\alpha_k-\ell_k)
\geq 0.
\end{align*}
Consider two \(\alpha\)-blocks of the labeled path, say the
\(k\)th and \(k'\)th blocks with \(k<k'\).  Let the earlier block be
$A=(a_1,\cdots,a_s)$, with corresponding labels
$x=(x_1,\cdots,x_s)$, and let the later block be
$B=(b_1,\cdots,b_S)$, with corresponding labels
$y=(y_1,\cdots,y_S)$.  These blocks satisfy
\begin{equation}
\begin{aligned}
a_1=b_1=0,
\qquad
a_{i+1}\leq a_i+1,
\qquad
b_{j+1}\leq b_j+1,
\qquad
s=\alpha_k\leq\alpha_{k'}=S,
\end{aligned}
\end{equation}
where \(s\leq S\) because \(\alpha_1\leq\cdots\leq\alpha_m\).
Since \(\beta\) refines \(\alpha\), the path $D$ of $\pi$ meets the diagonal
at every boundary between \(\alpha\)-blocks.  If $u_i$ and $v_j$ are the
lower-endpoint $x$-coordinates of the north steps corresponding to
$a_i$ and $b_j$, respectively, then $0<v_j-u_i\leq n$.
For a square path, $a_{i+1}=a_i+1$ precisely when the corresponding
north steps lie in the same column; their labels therefore increase.
We use this observation at each first occurrence of a height below.

In the square case their ranks are
\[
n a_i+\frac{u_i}{n+1},
\qquad
n b_j+\frac{v_j}{n+1}.
\]
Thus equal heights give the rank direction from $A$ to $B$, whereas
heights differing by one give the rank direction from the lower entry
in $B$ to the higher entry in $A$:
\[
a_i=b_j\Longrightarrow
0<
\left(n b_j+\frac{v_j}{n+1}\right)
-
\left(n a_i+\frac{u_i}{n+1}\right)
<n,
\]
and
\[
a_i=b_j+1\Longrightarrow
0<
\left(n a_i+\frac{u_i}{n+1}\right)
-
\left(n b_j+\frac{v_j}{n+1}\right)
<n.
\]
For an $(n,n)$-parking function, we have
$\pdinv(D)=\maxtdinv(D)$ by Blasiak et al.~\cite[Proposition~8.4.10]{nonsymshuffle},
hence $\dinv(\pi)=\tdinv(\pi)$.  The rank computations above
show that every ordered pair constructed below is a temporary diagonal
inversion.

We claim that the two blocks contribute at least $s-z(A)$ terms to the
$\dinv$ statistic, where $z(A):=\#\{i:a_i=0\}$.

\medskip

\noindent
\textbf{Case 1.}
Suppose that the later block $B$ reaches the maximum height attained by
$A$.  For every positive-height position $a_i=h>0$ in $A$, let $j_h$
be the first position at which $B$ reaches height $h$.  Then
$b_{j_h}=h$, $b_{j_h-1}=h-1$, and
$y_{j_h-1}<y_{j_h}$.  We now compare $x_i$ with $y_{j_h}$.

If $x_i<y_{j_h}$, then $a_i=b_{j_h}=h$, so $(i,j_h)$ forms an
inversion.  If $x_i\geq y_{j_h}$, then, since the labels are distinct,
$x_i>y_{j_h-1}$.  Together with $a_i=b_{j_h-1}+1$, this shows that
$(j_h-1,i)$ forms an inversion.

Therefore, for every positive-height position in $A$, we can find an
inversion between the blocks $A$ and $B$.  Hence, the two blocks
contribute at least $s-z(A)$ terms to the $\dinv$ statistic.

\medskip

\noindent
\textbf{Case 2.}
Suppose that the later block $B$ does not reach the maximum height
attained by $A$.  In this case, for every position $b_j=h$ in $B$, the
earlier block $A$ reaches height $h+1$.  Let $i_{h+1}$ be the first
position at which $A$ reaches height $h+1$.  Then
$a_{i_{h+1}-1}=h$, $a_{i_{h+1}}=h+1$, and
$x_{i_{h+1}-1}<x_{i_{h+1}}$.  We now compare
$x_{i_{h+1}-1}$ with $y_j$.

If $x_{i_{h+1}-1}<y_j$, then $(i_{h+1}-1,j)$ contributes an inversion.
If $x_{i_{h+1}-1}\geq y_j$, then
$x_{i_{h+1}}>x_{i_{h+1}-1}\geq y_j$, so $(j,i_{h+1})$ contributes an
inversion.

Therefore, the blocks $A$ and $B$ contribute at least $S$ inversions.
Since $S\geq s\geq s-z(A)$, the claim follows.

Let $D$ denote the underlying path of $\pi$.  Fix the $k$th block of
$D$.  There are $m-k$ blocks after it.  For every later block, say the $k'$th block with $k'>k$, the
claim shows that the pair consisting of the $k$th and $k'$th blocks
contributes at least $\alpha_k-z_k(D)$ terms to the $\dinv$ statistic,
where $z_k(D)$ denotes the number of zeros in the $k$th block.

Consequently, the $k$th block together with all the blocks following
it contributes at least $(m-k)(\alpha_k-z_k(D))$ terms to the $\dinv$
statistic.  Since \(\comp(\pi)=\beta\), we have \(z_k(D)=\ell_k\) for every \(k\).  Therefore,
\begin{align*}
\dinv(\pi)
&\geq
\sum_{1\leq k<k'\leq m}
\bigl(\alpha_k-z_k(D)\bigr)
\\
&=
\sum_{k=1}^{m}
(m-k)\bigl(\alpha_k-z_k(D)\bigr)
\\
&=
\sum_{k=1}^{m}
(m-k)(\alpha_k-\ell_k)
\\
&=
\delta_\mu(\beta).
\end{align*}
Thus,
$\doff_{\overleftarrow{\mu}}(\DP(\beta))-n(\mu)+\dinv(\pi)
=\dinv(\pi)-\delta_\mu(\beta)\geq 0$,
and the lemma follows.
\end{proof}

The $(6,6)$-parking function in
Example~\ref{ex:parking-statistics-doff} attains this bound:
\[
\dinv(\pi)=n(\mu)-\doff_{\overleftarrow\mu}(D)=2.
\]

We can now prove the following corollary.

\begin{corollary}[$=$ Theorem~\ref{thmmain2}]\label{corohallschur}
For any partitions $\mu,\lambda\vdash n$, we have
\begin{equation}
\begin{aligned}
\left\langle
\nabla\omega\wtH_\mu[X;0,t^{-1}],
s_\lambda
\right\rangle
&\in
\mathbb{N}[q,t].
\end{aligned}
\end{equation}
\end{corollary}

\begin{proof}
We continue to use the notation from the proof of
Lemma \ref{lemmadoff}.  It suffices to prove that
\begin{equation}\label{eq19}
\begin{aligned}
\left\langle
t^{-\delta_\mu(\beta)}
\nabla C_\beta[X;t],
s_\lambda
\right\rangle
&\in
\mathbb{N}[q,t].
\end{aligned}
\end{equation}

First, by Proposition~\ref{QZ26schur}, we have
\begin{equation}\label{eq20}
\begin{aligned}
\left\langle
t^{-\delta_\mu(\beta)}
\nabla C_\beta[X;t],
s_\lambda
\right\rangle
&\in
\mathbb{N}[q,t, t^{-1}].
\end{aligned}
\end{equation}

Second, Theorem~\ref{kmknshuffle}, applied to this fixed composition
$\beta$, gives
\[
t^{-\delta_\mu(\beta)}\nabla C_\beta[X;t]
=
\sum_{\substack{\pi\in\LD_{n,n}\\
\comp(\pi)=\beta}}
q^{\area(\pi)}
t^{\dinv(\pi)-\delta_\mu(\beta)}
F_{\ides(\pi)}.
\]
Lemma~\ref{lemmadoff} makes every exponent of $t$ in this display
nonnegative.  Since every fundamental quasisymmetric function is
monomial-positive, we have
\begin{equation}
\left\langle
t^{-\delta_\mu(\beta)}
\nabla C_\beta[X;t],
h_\lambda
\right\rangle \in
\mathbb{N}[q,t].
\end{equation}

Since the Kostka matrix is unitriangular with integer entries, its inverse also has integer entries.  It follows that
\begin{equation}\label{eq22}
\begin{aligned}
\left\langle
t^{-\delta_\mu(\beta)}
\nabla C_\beta[X;t],
s_\lambda
\right\rangle
&\in
\mathbb{Z}[q,t].
\end{aligned}
\end{equation}
Combining \eqref{eq20} with \eqref{eq22}, we conclude \eqref{eq19}.
\end{proof}

\section{Hall--Littlewood polynomials and LLT polynomials}
In this section, we establish the connection between Hall--Littlewood
polynomials and LLT polynomials.  We first use results of Macdonald \cite{Mac95} and Jing \cite{Jin91} to express
the relevant modified Hall--Littlewood polynomials in terms of the
Catalan functions \(H_{p,\nu}[X;t]\) of Chen \cite{Chen10} and Blasiak, Morse, Pun, and Summers~\cite{BMPS19}.  We then use the
representation-theoretic interpretation of LLT polynomials introduced by
Grojnowski and Haiman \cite{GrojnowskiHaiman07}, in the form reformulated
by Blasiak et al.~\cite{dens}, to relate \(H_{p,\nu}[X;t]\) to LLT polynomials.

\subsection{Catalan functions}
Fix $R\geq |\mu|$, and append trailing zeros to every partition under
consideration so that it has length $R$.  Set
\begin{equation*}
    \Delta_R^+
    =
    \left\{
        \epsilon_a-\epsilon_b:1\leq a<b\leq R
    \right\},
    \qquad
    \alpha_p=\epsilon_p-\epsilon_{p+1}.
\end{equation*}
Here $\epsilon_i$ denotes the vector in $\mathbb{Z}^R$ whose $i$th
coordinate is $1$ and whose remaining coordinates are $0$.

For a subset $\Sigma\subseteq\Delta_R^+$, define the
$t$-Kostant partition function by
\begin{equation}\label{defntkostant-general}
    \prod_{\alpha\in\Sigma}
    \frac{1}{1-tz^\alpha}
    =
    \sum_{\eta\in\mathbb{Z}^R}
    P_\Sigma(\eta;t)z^\eta,
\end{equation}
where $z^\eta:=z_1^{\eta_1}\cdots z_R^{\eta_R}$.  
Equivalently, $P_\Sigma(\eta;t)=\sum t^{k}$, where the sum runs over all multisets of roots in $\Sigma$ whose sum is $\eta$, 
and $k$ is the size of the multiset; we call such a multiset a $\Sigma$-partition of $\eta$.
We use the convention that $P_\Sigma(\eta;t)=0$ unless $\eta$ is a nonnegative integral
combination of roots in $\Sigma$.  Thus we may let $\eta$
range over all of $\mathbb{Z}^R$ in the preceding sum.  

For $\gamma\in\mathbb{Z}^R$, define the generalized Schur function
$s_\gamma$ by
\begin{equation}\label{eq:generalized-schur}
    s_\gamma[X]
    :=
    \det
    \bigl(
        h_{\gamma_a-a+b}[X]
    \bigr)_{1\leq a,b\leq R},
\end{equation}
where $h_m[X]$ is the complete homogeneous symmetric function,
with $h_0[X]=1$ and $h_m[X]=0$ for $m<0$.

Generalized Schur functions satisfy the straightening relation
\begin{equation*}
    s_{w\cdot\gamma}[X]
    =
    (-1)^{\ell(w)}s_\gamma[X],
\end{equation*}
where $w\in\mathfrak{S}_R$, $w\cdot\gamma:=w(\gamma+\rho)-\rho$, and
$\rho=(R-1,R-2,\cdots,0)$.

\begin{definition}
    For $\gamma\in\mathbb{Z}^R$, define
    \begin{equation}\label{eqQgamma}
        Q_\gamma[X;t]
        :=
        \sum_{\eta\in\mathbb{Z}^R}
        P_{\Delta_R^+}(\eta;t)
        s_{\gamma+\eta}[X].
    \end{equation}
\end{definition}

For a partition $\gamma$, \eqref{eqQgamma} and the results of Macdonald
\cite[Chapter III, Section 2]{Mac95} and Jing \cite{Jin91} give
$Q_\gamma[X;t]=t^{n(\gamma)}\wtH_\gamma[X;0,1/t]$.  Here $Q_\gamma$ always denotes
the function defined in \eqref{eqQgamma}; it should not be confused
with Macdonald's Hall--Littlewood $Q$-function.

Similarly, for each simple root
$\alpha_p=\epsilon_p-\epsilon_{p+1}$, where $1\leq p<R$, set
$\Omega_p:=\Delta_R^+\setminus\{\alpha_p\}$ and define
\begin{equation*}
    H_{p,\gamma}[X;t]
    :=
    \sum_{\eta\in\mathbb{Z}^R}
    P_{\Omega_p}(\eta;t)
    s_{\gamma+\eta}[X].
\end{equation*}

\begin{remark}
    The set $\Omega_p$ is a root ideal in the sense of Blasiak, Morse, Pun, and Summers~\cite{BMPS19}, since $\alpha_p$ is a minimal positive root.
    Thus $H_{p,\gamma}[X;t]$ here is the Catalan function $H(\Omega_p;\gamma)=\prod_{\epsilon_a-\epsilon_b\in\Omega_p}(1-tR_{ab})^{-1}s_\gamma$ of Chen \cite{Chen10} and Blasiak, Morse, Pun, and Summers~\cite{BMPS19}, where the raising operator $R_{ab}$ here sends $s_\kappa$ to $s_{\kappa+\epsilon_a-\epsilon_b}$.
\end{remark}

The following theorem is the main theorem in this subsection.
\begin{theorem}\label{thmstep1}
    Let $\mu,\nu\vdash n$ and suppose that $\mu\gtrdot\nu$ in the
    dominance order; that is, $\mu$ covers $\nu$.  Write
    $\nu=\mu-\epsilon_i+\epsilon_j$ for some $i<j$.  Then, for every
    $p\in\{i,\cdots,j-1\}$, we have
    \begin{equation*}
        t^{n(\nu)}\wtH_\nu[X;0,1/t]
        -
        t^{j-i+n(\mu)}\wtH_\mu[X;0,1/t]
        =
        H_{p,\nu}[X;t].
    \end{equation*}
    Since $n(\nu)=n(\mu)+j-i$, we can write it as
    \begin{equation}\label{eqstep1}
        t^{n(\nu)}(\wtH_\nu[X;0,1/t] - \wtH_\mu[X;0,1/t])
        =
        H_{p,\nu}[X;t].
    \end{equation}
\end{theorem}

To prove the theorem, we begin with the following observation.

\begin{lemma}\label{lemma1step1}
    For $\gamma\in\mathbb{Z}^R$, we have
    \begin{equation*}
        H_{p,\gamma}[X;t]
        =
        Q_\gamma[X;t]
        -
        t Q_{\gamma+\alpha_p}[X;t].
    \end{equation*}
\end{lemma}

\begin{proof}
    By the definition of $\Omega_p$, we have
    \begin{equation*}
        \prod_{\alpha\in\Delta_R^+}
        \frac{1}{1-tz^\alpha}
        =
        \frac{1}{1-tz^{\alpha_p}}
        \prod_{\alpha\in\Omega_p}
        \frac{1}{1-tz^\alpha}.
    \end{equation*}
    Equivalently,
    \begin{equation*}
        \sum_{\eta\in\mathbb{Z}^R}
        P_{\Omega_p}(\eta;t)z^\eta
        =
        (1-tz^{\alpha_p})
        \sum_{\eta\in\mathbb{Z}^R}
        P_{\Delta_R^+}(\eta;t)z^\eta.
    \end{equation*}
    Comparing the coefficients of $z^\eta$ on both sides gives
    \begin{equation*}
        P_{\Omega_p}(\eta;t)
        =
        P_{\Delta_R^+}(\eta;t)
        -
        tP_{\Delta_R^+}(\eta-\alpha_p;t).
    \end{equation*}
    Therefore,
    \begin{align*}
        H_{p,\gamma}[X;t]
        &=
        \sum_{\eta\in\mathbb{Z}^R}
        P_{\Delta_R^+}(\eta;t)
        s_{\gamma+\eta}[X]
        -
        t\sum_{\eta\in\mathbb{Z}^R}
        P_{\Delta_R^+}(\eta-\alpha_p;t)
        s_{\gamma+\eta}[X]\\
        &=
        Q_\gamma[X;t]
        -
        t\sum_{\xi\in\mathbb{Z}^R}
        P_{\Delta_R^+}(\xi;t)
        s_{\gamma+\alpha_p+\xi}[X]\\
        &=
        Q_\gamma[X;t]
        -
        tQ_{\gamma+\alpha_p}[X;t],
    \end{align*}
    where the second equality follows by setting
    $\xi=\eta-\alpha_p$.
\end{proof}

\begin{lemma}\label{lemma1step2}
    Let $\gamma\in\mathbb{Z}^R$.  If $\gamma_p=a$ and
    $\gamma_{p+1}=a+1$, then
    \begin{equation*}
        Q_{(\cdots,a,a+1,\cdots)}[X;t]
        =
        t Q_{(\cdots,a+1,a,\cdots)}[X;t].
    \end{equation*}
\end{lemma}

\begin{proof}
    By Lemma~\ref{lemma1step1}, it is equivalent to show that
    $H_{p,\gamma}[X;t]=0$.
    Let $s_p\in\mathfrak{S}_R$ denote the transposition of the
    $p$th and $(p+1)$st coordinates.  Clearly the set $\Omega_p$ is invariant
    under $s_p$.  Applying $s_p$ to every root in an
    $\Omega_p$-partition of $\eta$ therefore gives an
    $\Omega_p$-partition of $s_p\eta$, while preserving the number of
    roots used.  Hence
    \begin{equation*}
        P_{\Omega_p}(s_p\eta;t)
        =
        P_{\Omega_p}(\eta;t).
    \end{equation*}

    Moreover,
    $(\gamma+\rho)_p=a+R-p=(\gamma+\rho)_{p+1}$.  Therefore
    $s_p(\gamma+\rho)=\gamma+\rho$, or
    $s_p\cdot\gamma=\gamma$.  For every $\eta\in\mathbb{Z}^R$, we thus
    have
    \begin{align*}
        s_p\cdot(\gamma+\eta)
        =
        s_p(\gamma+\eta+\rho)-\rho =
        s_p(\gamma+\rho)+s_p\eta-\rho =
        \gamma+s_p\eta.
    \end{align*}
    It follows from the straightening relation that
    $s_{\gamma+s_p\eta}[X]
    =s_{s_p\cdot(\gamma+\eta)}[X]
    =-s_{\gamma+\eta}[X]$.

    We now group the defining sum for $H_{p,\gamma}$ into orbits of the
    involution $\eta\mapsto s_p\eta$.  If $s_p\eta\neq\eta$, the two
    terms corresponding to the orbit $\{\eta,s_p\eta\}$ satisfy
    \begin{align*}
        P_{\Omega_p}(\eta;t)s_{\gamma+\eta}[X]
        +
        P_{\Omega_p}(s_p\eta;t)s_{\gamma+s_p\eta}[X]
        =
        P_{\Omega_p}(\eta;t)
        \left(
            s_{\gamma+\eta}[X]
            -
            s_{\gamma+\eta}[X]
        \right)
        =
        0.
    \end{align*}

    If $s_p\eta=\eta$, then
    $(\gamma+\eta+\rho)_p
    =(\gamma+\eta+\rho)_{p+1}
    =a+\eta_p+R-p$.  Thus, the $p$th and $(p+1)$st rows in the
    determinant \eqref{eq:generalized-schur} coincide, and hence
    $s_{\gamma+\eta}[X]=0$.  Therefore, every $s_p$-orbit contributes
    zero.  The grouping is valid coefficientwise in $t$, since
    $\Omega_p$ is finite and hence each power of $t$ occurs in $P_{\Omega_p}(\eta;t)$ for only finitely many $\eta$; consequently $H_{p,\gamma}[X;t]=0$.

    Lemma~\ref{lemma1step1} now gives
    \begin{equation*}
        Q_\gamma[X;t]
        =
        tQ_{\gamma+\alpha_p}[X;t].
    \end{equation*}
    Since adding $\alpha_p$ replaces the adjacent pair $(a,a+1)$ by
    $(a+1,a)$, the desired identity follows.
\end{proof}

\begin{lemma}\label{platformlemma}
    Suppose that $\mu,\nu\vdash n$, that $\mu\gtrdot\nu$, and that
    $\nu=\mu-\epsilon_i+\epsilon_j$, where $i+1<j$.  Then there exists
    an integer $c\geq0$ such that
    \begin{equation*}
        (\mu_i,\mu_{i+1},\cdots,\mu_{j-1},\mu_j)
        =
        (c+2,c+1,\cdots,c+1,c).
    \end{equation*}
    Consequently, $\nu_i=\nu_{i+1}=\cdots=\nu_j=c+1$.
\end{lemma}

\begin{proof}
    By Brylawski~\cite[Proposition~2.3]{Brylawski1973}, a cover of the
    above form must satisfy
    $j=i+1$ or $\nu_i=\nu_j$.  Since we are in the nonadjacent case
    $j>i+1$, it follows that $\nu_i=\nu_j$.

    Since $\nu$ is weakly decreasing, we have
    $\nu_i\geq\nu_{i+1}\geq\cdots\geq\nu_j$.  Together with
    $\nu_i=\nu_j$, this shows that all these parts are equal.  Write
    their common value as $c+1$.  Since $\nu_j=\mu_j+1\geq1$, we have
    $c\geq0$.  Finally, the relation
    $\nu=\mu-\epsilon_i+\epsilon_j$ gives
    \begin{equation*}
        (\mu_i,\mu_{i+1},\cdots,\mu_j)
        =
        (c+2,c+1,\cdots,c+1,c),
    \end{equation*}
    as required.
\end{proof}

Now we are ready to prove Theorem~\ref{thmstep1}.

\begin{proof}[Proof of Theorem~\ref{thmstep1}]
    By Lemma~\ref{lemma1step1},
    $H_{p,\nu}[X;t]
    =Q_\nu[X;t]-tQ_{\nu+\alpha_p}[X;t]$.
    Moreover, the definition of $n(\lambda)$ and the equality
    $\nu=\mu-\epsilon_i+\epsilon_j$ give
    $n(\nu)-n(\mu)=-(i-1)+(j-1)=j-i$.

    If $j=i+1$, then $p=i$ and $\nu+\alpha_i=\mu$.  Therefore,
    \begin{align*}
        H_{i,\nu}[X;t]
        &=
        Q_\nu[X;t]
        -
        tQ_\mu[X;t]\\
        &=
        t^{n(\nu)}\wtH_\nu[X;0,1/t]
        -
        t^{1+n(\mu)}\wtH_\mu[X;0,1/t],
    \end{align*}
    which is precisely \eqref{eqstep1}, since $j-i=1$.

    Suppose now that $j>i+1$.  By Lemma~\ref{platformlemma}, the parts
    of $\nu$ from positions $i$ through $j$ are all equal to $c+1$.
    Hence
    \begin{equation*}
        \bigl(
            (\nu+\alpha_p)_i,
            \cdots,
            (\nu+\alpha_p)_j
        \bigr)
        =
        (c+1,\cdots,c+1,c+2,c,c+1,\cdots,c+1),
    \end{equation*}
    where $c+2$ occurs in position $p$ and $c$ occurs in position
    $p+1$.

    Using Lemma~\ref{lemma1step2}, first move $c+2$ to the left across
    the $p-i$ preceding copies of $c+1$, and then move $c$ to the
    right across the $j-p-1$ following copies of $c+1$.  The resulting
    block is
    $(\mu_i,\cdots,\mu_j)=(c+2,c+1,\cdots,c+1,c)$.

    Altogether, the number of adjacent exchanges is
    $(p-i)+(j-p-1)=j-i-1$.  Each exchange contributes one factor of
    $t$, and therefore
    \begin{equation*}
        Q_{\nu+\alpha_p}[X;t]
        =
        t^{j-i-1}Q_\mu[X;t].
    \end{equation*}
    Consequently,
    \begin{align*}
        H_{p,\nu}[X;t]
        &=
        Q_\nu[X;t]
        -
        tQ_{\nu+\alpha_p}[X;t]\\
        &=
        Q_\nu[X;t]
        -
        t^{j-i}Q_\mu[X;t]\\
        &=
        t^{n(\nu)}\wtH_\nu[X;0,1/t]
        -
        t^{j-i+n(\mu)}\wtH_\mu[X;0,1/t],
    \end{align*}
    which proves \eqref{eqstep1}.
\end{proof}

\subsection{Connection with LLT polynomials}\label{subLLTconn}
Fix a positive integer $R$.  For the rest of this subsection, let
$\mu,\nu\vdash n$ satisfy $\mu\gtrdot\nu$, write
$\nu=\mu-\epsilon_i+\epsilon_j$ with $i<j$ (with $R$ enlarged if necessary).  Set
$\mathbf r=(1^{i-1},2,1^{R-i-1})$.  This lists the block
sizes of the Levi subgroup
\[
\operatorname{GL}_{\mathbf r}
:=
\operatorname{GL}_1^{i-1}\times \operatorname{GL}_2\times
\operatorname{GL}_1^{R-i-1}.
\]

Recall the notation
$\Delta_R^+=\{\epsilon_a-\epsilon_b:1\leq a<b\leq R\},
\Omega_i=\Delta_R^+\setminus\{\epsilon_i-\epsilon_{i+1}\}.$
Write $R^+(\operatorname{GL}_{\mathbf r})$ for the set of positive roots of $\operatorname{GL}_{\mathbf r}$, that is, the roots $\epsilon_a-\epsilon_b\in\Delta_R^+$ such that $a$ and $b$ lie in the same block of $\mathbf r$.
Since the only positive root inside the central $\operatorname{GL}_2$-block of $\operatorname{GL}_{\mathbf r}$ is $\epsilon_i-\epsilon_{i+1}$, we have
$$\Omega_i=\Delta_R^+\setminus R^+(\operatorname{GL}_{\mathbf r}).$$

Following the notation of Blasiak et al.~\cite{shuffleline,LLTcat,dens}, we write $\weyl_R$ for the
Weyl symmetrization operator for $\operatorname{GL}_R$:
\begin{equation}\label{defsigmaR}
\weyl_R(f)
=
\sum_{w\in \mathfrak{S}_R}
w
\left(
\frac{f}
{\displaystyle\prod_{\gamma\in\Delta_R^+}(1-z^{-\gamma})}
\right).
\end{equation}
The Weyl character formula applied to \eqref{defsigmaR} gives
$\weyl_R(z^\lambda)=\chi_\lambda$ for every dominant
$\lambda\in\mathbb Z^R$, where $\chi_\lambda$ is the irreducible
$\operatorname{GL}_R$-character of highest weight $\lambda$.
For arbitrary $\kappa\in\mathbb Z^R$, it also gives
$\weyl_R(z^\kappa)=0$ or $\weyl_R(z^\kappa)=\pm\chi_\lambda$
for some dominant $\lambda$.
For a formal linear combination $f=\sum_\lambda a_\lambda\chi_\lambda$ of irreducible $\operatorname{GL}_R$-characters, 
its polynomial part is the symmetric function $(f)_{\pol}=\sum_{\lambda\in\mathbb N^R}a_\lambda s_\lambda[X]$.
The same definition is used for $\operatorname{GL}_l$ with any $l$.
The preceding characterization and \eqref{eq:generalized-schur} give
\[
\bigl(\weyl_R(z^\kappa)\bigr)_{\pol}=s_\kappa[X]
\text{ for } \kappa \in\mathbb{Z}^R.
\]

For a composition $\mathbf a=(a_1,\cdots,a_k)$ of $R$, let
$R^+(\operatorname{GL}_{\mathbf a})$ be the set of positive roots whose
two indices lie in the same consecutive $\mathbf a$-block, and let
$\mathfrak{S}_{\mathbf a}$ be the corresponding Levi Weyl group.
Write $\weyl_{\operatorname{GL}_{\mathbf a}}$ for Weyl symmetrization
with $R^+(\operatorname{GL}_{\mathbf a})$ in place of $\Delta_R^+$.
The coset decomposition of $\mathfrak S_R$ gives
\begin{equation}\label{eq:weyl-transitivity}
\weyl_R\bigl(\weyl_{\operatorname{GL}_{\mathbf a}}(f)\bigr)=\weyl_R(f).
\end{equation}
Following
Blasiak et al.~\cite[Equation~(124)]{dens}, define
\begin{equation}\label{defHtsss}
\mathbb H_t^{\mathbf a}(f)
=
\weyl_R
\left(
\frac{f}
{\displaystyle\prod_{\gamma\in
\Delta_R^+\setminus R^+(\operatorname{GL}_{\mathbf a})}
(1-tz^\gamma)}
\right).
\end{equation}
The reciprocal factors in \eqref{defHtsss} are expanded
geometrically before applying $\weyl_R$; hence
$\mathbb H_t^{\mathbf a}(f)$ is a formal $\operatorname{GL}_R$-character
series.  We write $\mathbb H_q^{\mathbf a}$ for the same operator with $t$ replaced by $q$.

Let $\lambda=(\lambda_1,\cdots,\lambda_R)\in\mathbb Z^R$ be blockwise dominant for $\operatorname{GL}_{\mathbf r}$, namely assume $\lambda_i\geq\lambda_{i+1}$.  We denote by $\chi_\lambda(\operatorname{GL}_{\mathbf r})$ the irreducible rational character of $\operatorname{GL}_{\mathbf r}$ whose blockwise highest weight is $\lambda$.  Since a weight $a$ of $\operatorname{GL}_1$ gives the one-dimensional character $z^a$, while the central $\operatorname{GL}_2$-block contributes the character $s_{(a,b)}(z_i,z_{i+1})$, we have
\begin{equation}\label{eqchilambda}
\chi_\lambda(\operatorname{GL}_{\mathbf r})
=
\left(
\prod_{k<i}z_k^{\lambda_k}
\right)
s_{(\lambda_i,\lambda_{i+1})}(z_i,z_{i+1})
\left(
\prod_{k>i+1}z_k^{\lambda_k}
\right).
\end{equation}

For any $\operatorname{GL}_{\mathbf r}$-representation $V$, we write $V^*$ for its dual representation and $\chi_V^*$ for the corresponding dual character. 
If $\lambda$ is blockwise dominant, then $\chi_\lambda(\operatorname{GL}_{\mathbf r})^*=\chi_{-w_0^{\mathbf r}(\lambda)}(\operatorname{GL}_{\mathbf r})$, where $w_0^{\mathbf r}$ is the longest element of the Weyl group of $\operatorname{GL}_{\mathbf r}$.

Now define
\[
B=
(\nu_1-1,\cdots,\nu_{i-1}-1
\mid
\nu_i,\nu_{i+1}
\mid
\nu_{i+2},\cdots,\nu_R)
\]
and
\[
C=
(-1,\cdots,-1
\mid
0,0
\mid
0,\cdots,0).
\]
Then $B-C=\nu$.

\begin{lemma}\label{lemmaBC}
For $\mathbf r=(1^{i-1},2,1^{R-i-1})$ and the weights $B,C$ above,
we have
\[
\chi_B(\operatorname{GL}_{\mathbf r})
\chi_C(\operatorname{GL}_{\mathbf r})^*
=\chi_\nu(\operatorname{GL}_{\mathbf r}).
\]
\end{lemma}

\begin{proof}
Substituting $B$ into \eqref{eqchilambda}, we obtain
\[
\chi_B(\operatorname{GL}_{\mathbf r})
=
\left(
\prod_{k<i}z_k^{\nu_k-1}
\right)
s_{(\nu_i,\nu_{i+1})}(z_i,z_{i+1})
\left(
\prod_{k>i+1}z_k^{\nu_k}
\right).
\]
Substituting $C$ into \eqref{eqchilambda}, we obtain $\chi_C(\operatorname{GL}_{\mathbf r})=\prod_{k<i}z_k^{-1}$.  Hence its dual character is $\chi_C(\operatorname{GL}_{\mathbf r})^*=\prod_{k<i}z_k$.  Therefore
\[
\begin{aligned}
\chi_B(\operatorname{GL}_{\mathbf r})\chi_C(\operatorname{GL}_{\mathbf r})^*
&=
\left(
\prod_{k<i}z_k^{\nu_k-1}z_k
\right)
s_{(\nu_i,\nu_{i+1})}(z_i,z_{i+1})
\left(
\prod_{k>i+1}z_k^{\nu_k}
\right)
\\
&=
\left(
\prod_{k<i}z_k^{\nu_k}
\right)
s_{(\nu_i,\nu_{i+1})}(z_i,z_{i+1})
\left(
\prod_{k>i+1}z_k^{\nu_k}
\right)
\\
&=
\chi_\nu(\operatorname{GL}_{\mathbf r}).
\end{aligned}
\]
\end{proof}

Next, we express $H_{i,\nu}[X;t]$ using $\mathbb{H}_t^{\mathbf r}$.

\begin{proposition}\label{propHpnuLevi}
Let $\mathbf r=(1^{i-1},2,1^{R-i-1})$.  Assume that $\nu$ is
blockwise dominant for $\operatorname{GL}_{\mathbf r}$.  Then
\[
H_{i,\nu}[X;t]
=
\left(
\mathbb H_t^{\mathbf r}
\bigl(
\chi_\nu(\operatorname{GL}_{\mathbf r})
\bigr)
\right)_{\pol}.
\]
Consequently,
\[
H_{i,\nu}[X;t]
=
\left(
\mathbb H_t^{\mathbf r}
\bigl(
\chi_B(\operatorname{GL}_{\mathbf r})\chi_C(\operatorname{GL}_{\mathbf r})^*
\bigr)
\right)_{\pol}.
\]
\end{proposition}

\begin{proof}
Multiplying both sides of \eqref{defntkostant-general} by $z^\nu$, we obtain
\[
\frac{z^\nu}
{\displaystyle\prod_{\gamma\in\Omega_i}(1-tz^\gamma)}
=
\sum_{\eta\in\mathbb Z^R}
P_{\Omega_i}(\eta;t)z^{\nu+\eta}.
\]
Applying $\weyl_R$ to both sides and then taking the polynomial part gives
\[
\begin{aligned}
\left(
\weyl_R
\left(
\frac{z^\nu}
{\displaystyle\prod_{\gamma\in\Omega_i}(1-tz^\gamma)}
\right)
\right)_{\pol}
&=
\sum_{\eta\in\mathbb Z^R}
P_{\Omega_i}(\eta;t)
\bigl(
\weyl_R(z^{\nu+\eta})
\bigr)_{\pol}
\\
&=
\sum_{\eta\in\mathbb Z^R}
P_{\Omega_i}(\eta;t)s_{\nu+\eta}[X]
\\
&=
H_{i,\nu}[X;t].
\end{aligned}
\]
The set $\Omega_i$ is stable under the Levi Weyl group
$\mathfrak S_{\mathbf r}$, so $\prod_{\gamma\in\Omega_i}(1-tz^\gamma)$ is Levi-invariant.  Since $\nu$ is
$\operatorname{GL}_{\mathbf r}$-dominant, $\chi_\nu(\operatorname{GL}_{\mathbf r})=\weyl_{\operatorname{GL}_{\mathbf r}}(z^\nu)$, and hence
\[
\frac{\chi_\nu(\operatorname{GL}_{\mathbf r})}{\prod_{\gamma\in\Omega_i}(1-tz^\gamma)}
=
\weyl_{\operatorname{GL}_{\mathbf r}}\left(\frac{z^\nu}{\prod_{\gamma\in\Omega_i}(1-tz^\gamma)}\right).
\]
Applying \eqref{eq:weyl-transitivity} gives
\[
\mathbb H_t^{\mathbf r}\bigl(\chi_\nu(\operatorname{GL}_{\mathbf r})\bigr)
=
\weyl_R\left(\frac{\chi_\nu(\operatorname{GL}_{\mathbf r})}{\prod_{\gamma\in\Omega_i}(1-tz^\gamma)}\right)
=
\weyl_R\left(\frac{z^\nu}{\prod_{\gamma\in\Omega_i}(1-tz^\gamma)}\right).
\]
Taking polynomial parts and using the preceding computation proves the
first identity.  The second identity follows from Lemma~\ref{lemmaBC}.
\end{proof}

We shall recall the semi-symmetric Hall--Littlewood polynomials
introduced by Blasiak et al.~\cite[Section~5]{dens}. 
For \(\lambda=(\lambda_1,\cdots,\lambda_R)\in\mathbb Z^R\), write
$z^\lambda=z_1^{\lambda_1}\cdots z_R^{\lambda_R}$.
For \(1\leq i<R\), let \(s_i\) denote the operator interchanging
\(z_i\) and \(z_{i+1}\).  Define the operator \(T_i\) by
\[
T_i f
=
t\,s_i f
+
(t-1)
\frac{z_i}{z_i-z_{i+1}}
(f-s_i f).
\]
The operators \(T_i\) satisfy the braid relations and the Hecke
relations
$(T_i-t)(T_i+1)=0$.
Hence, for \(w\in \mathfrak{S}_R\), one may define
$T_w=T_{i_1}\cdots T_{i_\ell}$,
where $w=s_{i_1}\cdots s_{i_\ell}$
is a reduced expression.  This definition is independent of the
choice of reduced expression.

For \(\lambda\in\mathbb Z^R\), let \(\lambda^+\) denote the dominant
rearrangement of \(\lambda\), and choose \(w\in \mathfrak{S}_R\) such that $\lambda=w(\lambda^+)$.
The nonsymmetric Hall--Littlewood polynomial indexed by \(\lambda\)
is
\[
E_\lambda(z;t)
=
t^{-\ell(w)}T_wz^{\lambda^+}.
\]
If \(\lambda^+\) has a nontrivial stabilizer, the permutation \(w\)
need not be unique; nevertheless, the above definition is
independent of the choice of \(w\).  In particular, if \(\lambda\) is
dominant, then $E_\lambda(z;t)=z^\lambda$.

Write $w_0$ for the longest element of $\mathfrak{S}_R$.  The bar involution
inverts all variables and parameters, that is,
\(
\overline{f(z;t)}
=
f(z_1^{-1},\cdots,z_R^{-1};t^{-1}),
\)
and the same convention applies when the parameter is $q$.  Thus
$\overline{T_i}=T_i^{-1}$ and
$\overline{T_w}=T_{w^{-1}}^{-1}$.

For \(\xi\in \mathfrak{S}_R\), define
\[
\operatorname{Inv}(\xi^{-1})
=
\left\{
(a,b):a<b,\ \xi^{-1}(a)>\xi^{-1}(b)
\right\}.
\]
For a vector $v\in\mathbb R^R$, write $\operatorname{Inv}(v)=\{(a,b):a<b,\ v_a>v_b\}$; the set above is $\operatorname{Inv}(v)$ for $v=(\xi^{-1}(1),\cdots,\xi^{-1}(R))$.  Let $\rho=(R-1,R-2,\cdots,0)$,
and choose \(\varepsilon>0\) sufficiently small so that
\[
\operatorname{Inv}(\lambda+\varepsilon\rho)
=
\left\{
(a,b):a<b,\ \lambda_a\geq\lambda_b
\right\}.
\]
The twisted nonsymmetric Hall--Littlewood polynomials are defined by
\[
E_\lambda^\xi(z;t)
=
t^{
\left|
\operatorname{Inv}(\xi^{-1})
\cap
\operatorname{Inv}(\lambda+\varepsilon\rho)
\right|
}
\overline{T_\xi}\,E_{\xi^{-1}(\lambda)}(z;t),
\]
and
\[
F_\lambda^\xi(z;t)
=
\overline{E_{-\lambda}^{\xi w_0}(z;t)}.
\]

For a Levi composition
$\mathbf a=(a_1,\cdots,a_k)$ of $R$, let
$\operatorname{GL}_{\mathbf a}
=
\operatorname{GL}_{a_1}\times\cdots\times\operatorname{GL}_{a_k}$,
and let $\mathfrak{S}_{\mathbf a}$ be its Weyl group.  A weight is in
$X^+(\operatorname{GL}_{\mathbf a})$, respectively
$X^{++}(\operatorname{GL}_{\mathbf a})$, if it is weakly, respectively
strictly, decreasing within every $\mathbf a$-block.  Choose
$\rho_{\mathbf a}$ so that its consecutive entries differ by $1$
inside every block, that is, $(\rho_{\mathbf a})_p-(\rho_{\mathbf a})_{p+1}=1$ whenever $p$ and $p+1$ lie in the same block.  We write $w_0^{\mathbf a}$ for the longest
element of $\mathfrak{S}_{\mathbf a}$, and $w_0^k\in \mathfrak{S}_k$ for the longest block
permutation; thus
$w_0^k(\mathbf a)=(a_k,\cdots,a_1)$.

For a permutation \(\sigma\in \mathfrak{S}_k\), let
\(\widehat{\sigma}\in \mathfrak{S}_R\) denote the induced block permutation.  More
precisely,
\[
\widehat{\sigma}
\left(
a_{\sigma(1)}+\cdots+a_{\sigma(p-1)}+j
\right)
=
a_1+\cdots+a_{\sigma(p)-1}+j,
\]
for
$1\leq p\leq k,
1\leq j\leq a_{\sigma(p)}$.  The lift of $\sigma^{-1}$ is formed
relative to $\sigma^{-1}(\mathbf a)$; in particular,
$\widehat{\sigma}^{-1}$ is not the lift of $\sigma^{-1}$ formed
relative to the original composition $\mathbf a$.

The Levi semi-symmetrization operator is
\[
\delta_{\mathbf a}(f)
=
\frac{\displaystyle
\sum_{w\in \mathfrak{S}_{\mathbf a}}(-1)^{\ell(w)}w(f)}
{\displaystyle\prod_{\gamma\in R^+(\operatorname{GL}_{\mathbf a})}
(1-z^{-\gamma})}.
\]
For the fixed composition $\mathbf r=(1^{i-1},2,1^{R-i-1})$, the Levi Weyl group is
$\mathfrak{S}_{\mathbf r}=\langle s_i\rangle$, and the preceding formula becomes
\[
\delta_{\mathbf r}(f)
=
\frac{f-s_if}{1-z_{i+1}/z_i}.
\]
For $\lambda\in X^{++}(\operatorname{GL}_{\mathbf a})$, define
\[
E_{\mathbf a,\lambda}^{\sigma}(z;t)
=
\delta_{\mathbf a}E_\lambda^{\widehat{\sigma}}(z;t),
\qquad
F_{\mathbf a,\lambda}^{\sigma}(z;t)
=
\delta_{\mathbf a}F_\lambda^{\widehat{\sigma}}(z;t).
\]
If $\mathbf a$ is weak, delete its zero blocks.  More precisely, let
$i_1<\cdots<i_j$ be the indices of its nonzero blocks and put
$\mathbf a'=(a_{i_1},\cdots,a_{i_j})$.  Let $\tau\in \mathfrak{S}_j$ be the
unique permutation for which
$\tau^{-1}(1),\cdots,\tau^{-1}(j)$ has the same relative order as
$\sigma^{-1}(i_1),\cdots,\sigma^{-1}(i_j)$.  Then the preceding
notation means
\[
E_{\mathbf a,\lambda}^{\sigma}(z;t)
=
E_{\mathbf a',\lambda}^{\tau}(z;t),
\qquad
F_{\mathbf a,\lambda}^{\sigma}(z;t)
=
F_{\mathbf a',\lambda}^{\tau}(z;t).
\]
The following symmetrization identity relates the LLT series of Blasiak et al.~\cite{dens}, whose definition is recalled in the statement, to
semi-symmetric Hall--Littlewood polynomials.

\begin{proposition}[{\cite[Proposition~5.4.3]{dens}}]\label{prop:dens-symmetrization}
Let $\mathbf a=(a_1,\cdots,a_k)$
be a composition of \(R\), let \(\sigma\in \mathfrak{S}_k\), and let
$\alpha,\beta\in X^{++}(\operatorname{GL}_{\mathbf a})$.
Then
\[
\mathcal L_{\mathbf a,\beta/\alpha}^{\sigma}(z;t)
=
\mathbb H_t^{w_0^k(\mathbf a)}
\left(
w_0\left(
F_{\mathbf a,\beta}^{\sigma^{-1}}(z;t)
\,
\overline{
E_{\mathbf a,\alpha}^{\sigma^{-1}}(z;t)
}
\right)
\right).
\]
The LLT series on the left side is defined by
Blasiak et al.~\cite[Definition~5.4.1]{dens} through
\[
\left\langle\chi_\lambda\right\rangle
\mathcal L_{\mathbf a,\beta/\alpha}^{\sigma^{-1}}(z;t^{-1})
=
\left\langle
E_{\mathbf a,\beta}^{\sigma}(z;t)
\right\rangle
\chi_\lambda(z)
E_{\mathbf a,\alpha}^{\sigma}(z;t).
\]
Here \(\langle E_{\mathbf a,\beta}^{\sigma}(z;t)\rangle f\)
denotes the coefficient of \(E_{\mathbf a,\beta}^{\sigma}(z;t)\) in
the expansion of \(f\) in the semi-symmetric Hall--Littlewood basis, and \(\langle\chi_\lambda\rangle g\) denotes the coefficient of the irreducible character \(\chi_\lambda\) in the expansion of \(g\) in irreducible \(\operatorname{GL}_R\)-characters.
\end{proposition}

\begin{proposition}\label{prop:levi-character}
Let $1\leq i<R$, $\mathbf r=(1^{i-1},2,1^{R-i-1})$, and
$D=(D_1,\cdots,D_R)\in\mathbb Z^R$.
Assume that $D$ is blockwise dominant for $\operatorname{GL}_{\mathbf r}$;
that is, $D_i\geq D_{i+1}$.
Then
\[
z^{-\rho_{\mathbf r}}
\delta_{\mathbf r}
\left(
z^{D+\rho_{\mathbf r}}
\right)
=
\chi_D(\operatorname{GL}_{\mathbf r}).
\]
Here,
\(
\rho_{\mathbf r}
=
(-1,\cdots,-1\mid-1,-2\mid-1,\cdots,-1).
\)
\end{proposition}

\begin{proof}
Write $
(D_i,D_{i+1})=(a,b),
a\geq b.
$
The central \(\operatorname{GL}_2\)-block contributes
$$z_i z_{i+1}^2
\frac{
z_i^{a-1}z_{i+1}^{b-2}
-
z_i^{b-2}z_{i+1}^{a-1}
}
{1-z_{i+1}/z_i}
=
\frac{
z_i^{a+1}z_{i+1}^{b}
-
z_i^{b}z_{i+1}^{a+1}
}
{z_i-z_{i+1}}
=
s_{(a,b)}(z_i,z_{i+1}).
$$
Each \(\operatorname{GL}_1\)-block contributes the monomial prescribed by its
weight.  Therefore,
\[
z^{-\rho_{\mathbf r}}
\delta_{\mathbf r}
\left(
z^{D+\rho_{\mathbf r}}
\right)
=
\left(
\prod_{k<i}z_k^{D_k}
\right)
s_{(D_i,D_{i+1})}(z_i,z_{i+1})
\left(
\prod_{k>i+1}z_k^{D_k}
\right)
=\chi_D(\operatorname{GL}_{\mathbf r}).
\]
\end{proof}

\begin{lemma}\label{Platform}
Let $\mu\gtrdot\nu$ be a dominance cover, and write
$\nu=\mu-\epsilon_i+\epsilon_j$ with $i<j$.
Let
\(
\mathbf r=(1^{i-1},2,1^{R-i-1}),
B=(\nu_1-1,\cdots,\nu_{i-1}-1
\mid \nu_i,\nu_{i+1}
\mid \nu_{i+2},\cdots,\nu_R),
\)
and
\(
\rho_{\mathbf r}
=
(-1,\cdots,-1\mid-1,-2\mid-1,\cdots,-1).
\)
Then \(B+\rho_{\mathbf r}\in X^{++}(\operatorname{GL}_{\mathbf r})\) and
\[
E_{\mathbf r,B+\rho_{\mathbf r}}^{1}(z;t)
=
\delta_{\mathbf r}E_{B+\rho_{\mathbf r}}(z;t)
=
\delta_{\mathbf r}\left(z^{B+\rho_{\mathbf r}}\right)
=
z^{\rho_{\mathbf r}}\chi_B(\operatorname{GL}_{\mathbf r}).
\]
Here the superscript \(1\) denotes the identity permutation.
\end{lemma}

\begin{proof}
Since $\nu$ is a partition,
\[
(B+\rho_{\mathbf r})_i-(B+\rho_{\mathbf r})_{i+1}
=
\nu_i-\nu_{i+1}+1
\geq 1.
\]
Thus \(B+\rho_{\mathbf r}\in X^{++}(\operatorname{GL}_{\mathbf r})\).

First suppose that $j=i+1$.  In this case
$B+\rho_{\mathbf r}$ is globally dominant.  Indeed, when the indicated
parts exist, the inequalities \(\mu_{i-1}\geq\mu_i\) and \(\mu_{i+1}\geq\mu_{i+2}\) give
\(\nu_{i-1}\geq\nu_i+1\) and
\(\nu_{i+1}\geq\nu_{i+2}+1\); together with the monotonicity of
$\nu$, these verify the only boundary inequalities.  Hence
$E_{B+\rho_{\mathbf r}}(z;t)=z^{B+\rho_{\mathbf r}}$, so the assertion
follows from Proposition~\ref{prop:levi-character}.

Now suppose that $j>i+1$.  By Lemma~\ref{platformlemma}, there is a
$c\geq0$ such that
\(
\nu_i=\nu_{i+1}=\cdots=\nu_j=c+1.
\)
The inequalities \(\nu_{i-1}=\mu_{i-1}\geq\mu_i=c+2\) (if \(i>1\)) and \(\nu_{j+1}=\mu_{j+1}\leq\mu_j=c\) (if \(j<R\)) at the two ends of the string $\nu_i=\cdots=\nu_j$, together
with the monotonicity of $\nu$, show that $B+\rho_{\mathbf r}$ is weakly
decreasing except for the displayed string
\[
\bigl((B+\rho_{\mathbf r})_i,\cdots,(B+\rho_{\mathbf r})_j\bigr)
=
(c,c-1,c,\cdots,c).
\]
Its dominant rearrangement \(\bigl(B+\rho_{\mathbf r}\bigr)^+\) is
obtained by moving the displayed $c-1$ to position $j$.  Put
\(
w=s_{i+1}s_{i+2}\cdots s_{j-1}.
\)
Then $B+\rho_{\mathbf r}=w\bigl(\bigl(B+\rho_{\mathbf r}\bigr)^+\bigr)$ and $\ell(w)=j-i-1$.

For $i+1\leq a\leq j$, let $m_a$ be the monomial obtained from
$z^{(B+\rho_{\mathbf r})^+}$ by placing this distinguished $c-1$ in
position $a$.  Thus $m_{i+1}=z^{B+\rho_{\mathbf r}}$.  The
definition of $T_a$ above gives
\[
t^{-1}T_a m_{a+1}
=
m_a+(1-t^{-1})m_{a+1},
\qquad
T_a m_b=t\,m_b
\quad (b\geq a+2).
\]
An induction downwards on $a$ therefore yields
\[
E_{B+\rho_{\mathbf r}}(z;t)
=
t^{-\ell(w)}T_w z^{(B+\rho_{\mathbf r})^+}
=
m_{i+1}
+
(1-t^{-1})\sum_{a=i+2}^{j}m_a.
\]
For every $a\geq i+2$, the exponents of $z_i$ and $z_{i+1}$ in $m_a$
are both $c$.  Hence $s_i(m_a)=m_a$, and therefore
$\delta_{\mathbf r}(m_a)=0$.  It follows that
\[
\delta_{\mathbf r}E_{B+\rho_{\mathbf r}}(z;t)
=
\delta_{\mathbf r}m_{i+1}
=
\delta_{\mathbf r}\left(z^{B+\rho_{\mathbf r}}\right).
\]

Finally, Proposition~\ref{prop:levi-character}, applied to the
$\operatorname{GL}_{\mathbf r}$-dominant weight $B$, gives
\[
\delta_{\mathbf r}\left(z^{B+\rho_{\mathbf r}}\right)
=
z^{\rho_{\mathbf r}}\chi_B(\operatorname{GL}_{\mathbf r}).
\]
This proves the result.
\end{proof}

For brevity, set $\mathbf r^{\leftarrow}
=
w_0^{R-1}(\mathbf r)
=
(1^{R-i-1},2,1^{i-1}),
\sigma_0=w_0^{R-1}\in \mathfrak{S}_{R-1}$.
Let \(w_0^{\mathbf r^{\leftarrow}}\) denote the longest element of the Levi Weyl
group \(\mathfrak{S}_{\mathbf r^{\leftarrow}}\).
The permutation \(w_0w_0^{\mathbf r^{\leftarrow}}\) reverses the order of the Levi blocks while
preserving the order of the coordinates within each block.
Recall $
B=
(\nu_1-1,\cdots,\nu_{i-1}-1
\mid
\nu_i,\nu_{i+1}
\mid
\nu_{i+2},\cdots,\nu_R)
$
and
$
C=
(-1,\cdots,-1
\mid
0,0
\mid
0,\cdots,0).
$
Set
\(
B^{\leftarrow}=(w_0w_0^{\mathbf r^{\leftarrow}})^{-1}B,
C^{\leftarrow}=(w_0w_0^{\mathbf r^{\leftarrow}})^{-1}C,
\rho_{\mathbf r^{\leftarrow}}=(w_0w_0^{\mathbf r^{\leftarrow}})^{-1}\rho_{\mathbf r},
\)
and define
\(
\alpha=C^{\leftarrow}+\rho_{\mathbf r^{\leftarrow}},
\beta=B^{\leftarrow}+\rho_{\mathbf r^{\leftarrow}}.
\)
Both $B$ and $C$ lie in $X^+(\operatorname{GL}_{\mathbf r})$.  Consequently,
$B^{\leftarrow},C^{\leftarrow}\in X^+(\operatorname{GL}_{\mathbf r^{\leftarrow}})$.  Moreover,
the chosen staircase weight is
\(
\rho_{\mathbf r^{\leftarrow}}
=
(-1,\cdots,-1\mid-1,-2\mid-1,\cdots,-1),
\)
and is strictly decreasing within every \(\mathbf r^{\leftarrow}\)-block.  Hence
\[
\alpha,\beta\in X^{++}(\operatorname{GL}_{\mathbf r^{\leftarrow}}), \quad
\beta-\alpha = B^{\leftarrow}-C^{\leftarrow} = (w_0w_0^{\mathbf r^{\leftarrow}})^{-1}\nu \geq0.
\]
In particular, \(\alpha\leq\beta\) coordinatewise.

For a coordinate \(m\) lying in a block of \(\mathbf r^{\leftarrow}\), let \(j\) be
its position within that block.
The skew-tuple construction of Blasiak et al.~\cite[Equation~(133)]{dens} gives row
endpoints
\[
\alpha_m+j=C^{\leftarrow}_m,
\qquad
\beta_m+j=B^{\leftarrow}_m,
\]
since \((\rho_{\mathbf r^{\leftarrow}})_m=-j\).  Here, for weakly decreasing integer vectors $\mathbf b$ and $\mathbf c$ of the same length with $c_y\leq b_y$ for every $y$, we write $\mathbf b/\mathbf c$, or $\frac{\mathbf b}{\mathbf c}$, for the skew diagram whose $y$th row consists of the boxes $(x,y)$ with $c_y<x\leq b_y$; a vector with one entry is written without parentheses.  Applied to each block of $\mathbf r$, this gives the tuple $B/C$ of skew diagrams.  Thus the skew tuple associated
with \(\beta/\alpha\) is exactly the block-reversal of \(B/C\).  It follows that
\[
\sigma_0(\beta/\alpha)=\theta^{\mu,\nu},
\]
where
\[
\theta^{\mu,\nu}
=
\left(
\frac{\nu_1-1}{-1},
\cdots,
\frac{\nu_{i-1}-1}{-1},
\frac{(\nu_i,\nu_{i+1})}{(0,0)},
\frac{\nu_{i+2}}0,
\cdots,
\frac{\nu_R}0
\right).
\]
For an ordered skew tuple \(\boldsymbol{\nu}\) with $k$ components and a permutation \(\sigma\in\mathfrak S_k\),
let \(h_\sigma(\boldsymbol{\nu})\) denote the number of \(\sigma\)-triples in \(\boldsymbol{\nu}\) in the sense of
Blasiak et al.~\cite[Definition~5.5.1]{dens}, that is, triples $(a,b,c)$ of boxes such that $b$ lies in the $p$th component; for some $p'>p$, the box $a$ lies in or immediately left of a row of the $p'$th component, and $c$ is the box immediately right of $a$, lying in or immediately right of the same row (rows may be empty); and $b$ has the same content as $c$ if $\sigma(p)<\sigma(p')$, and the same content as $a$ if $\sigma(p)>\sigma(p')$.  We write $\sigma(\boldsymbol\nu)$ for the tuple whose $\sigma(p)$th component is the $p$th component of $\boldsymbol\nu$.  Whenever \(h_\sigma(\beta/\alpha)\) or
\(\sigma(\beta/\alpha)\) occurs below, the notation refers to the skew
tuple associated with the pair of weights \(\alpha\leq\beta\).

\begin{proposition}\label{prop:special-character-factor}
With $\mathbf r=(1^{i-1},2,1^{R-i-1})$ and
$\mathbf r^{\leftarrow}=(1^{R-i-1},2,1^{i-1})$ as above,
\[
w_0\left(
F_{\mathbf r^{\leftarrow},\beta}^{\sigma_0^{-1}}(z;t)
\,
\overline{
E_{\mathbf r^{\leftarrow},\alpha}^{\sigma_0^{-1}}(z;t)
}
\right)
=
\chi_B(\operatorname{GL}_{\mathbf r})\chi_C(\operatorname{GL}_{\mathbf r})^*.
\]
\end{proposition}

\begin{proof}
The preceding construction gives the required blockwise regularity of
\(\alpha\) and \(\beta\).  Thus the hypotheses required by
Blasiak et al.~\cite[Lemma~5.2.4]{dens} are satisfied.  In the
\(\mathbf r^{\leftarrow}\)-coordinate order, the weight \(\alpha\) has the form
$(-1,\cdots,-1
\mid
-1,-2
\mid
-2,\cdots,-2)$.
Thus \(\alpha\) is globally dominant, so
$E_\alpha^\xi(z;t)=z^\alpha$ for every permutation \(\xi\) by
Blasiak et al.~\cite[Section~5.2]{dens}.  Hence
\[
E_{\mathbf r^{\leftarrow},\alpha}^{\sigma_0^{-1}}(z;t)
=
\delta_{\mathbf r^{\leftarrow}}(z^\alpha)
=
z^{\rho_{\mathbf r^{\leftarrow}}}
\chi_{C^{\leftarrow}}(\operatorname{GL}_{\mathbf r^{\leftarrow}}),
\]
where the last equality is the same blockwise Weyl-character
calculation as in Proposition~\ref{prop:levi-character}, after
reversing the Levi blocks.
After applying the bar involution and then \(w_0\), we obtain
\[
w_0\left(
\overline{
E_{\mathbf r^{\leftarrow},\alpha}^{\sigma_0^{-1}}(z;t)
}
\right)
=
z^{-w_0\rho_{\mathbf r^{\leftarrow}}}
\chi_C(\operatorname{GL}_{\mathbf r})^*.
\]

By Blasiak et al.~\cite[Equation~(95)]{dens}, for every composition
$\mathbf a$ with $k$ parts, every $\sigma\in\mathfrak S_k$ and every
$\kappa\in X^{++}(\operatorname{GL}_{\mathbf a})$, we have
\[
F_{\mathbf a,\kappa}^{\sigma}(z;t)
=
z^{\rho_{\mathbf a}-w_0^{\mathbf a}(\rho_{\mathbf a})}
w_0\left(
E_{w_0^k(\mathbf a),\,w_0w_0^{\mathbf a}(\kappa)}^{w_0^k\sigma}(z;t^{-1})
\right).
\]
We apply this with $\mathbf a=\mathbf r^{\leftarrow}$, $k=R-1$, $\sigma=\sigma_0^{-1}$ and
$\kappa=\beta$.  Since $\sigma_0=w_0^{R-1}$ is an involution and
$\beta=(w_0w_0^{\mathbf r^{\leftarrow}})^{-1}(B+\rho_{\mathbf r})$, we have
\[
w_0^{R-1}(\mathbf r^{\leftarrow})=\mathbf r, w_0w_0^{\mathbf r^{\leftarrow}}(\beta)=B+\rho_{\mathbf r}, w_0^{R-1}\sigma_0^{-1}=1.
\]
Hence
\[
F_{\mathbf r^{\leftarrow},\beta}^{\sigma_0^{-1}}(z;t)
=
z^{\rho_{\mathbf r^{\leftarrow}}-w_0^{\mathbf r^{\leftarrow}}(\rho_{\mathbf r^{\leftarrow}})}
w_0\left(
E_{\mathbf r,B+\rho_{\mathbf r}}^{1}(z;t^{-1})
\right).
\]
Since
$w_0w_0^{\mathbf r^{\leftarrow}}(\rho_{\mathbf r^{\leftarrow}})=\rho_{\mathbf r}$,
Lemma \ref{Platform}, applied with parameter $t^{-1}$,
gives
\[
w_0\left(
F_{\mathbf r^{\leftarrow},\beta}^{\sigma_0^{-1}}(z;t)
\right)
=
z^{w_0\rho_{\mathbf r^{\leftarrow}}-\rho_{\mathbf r}}
E_{\mathbf r,B+\rho_{\mathbf r}}^{1}(z;t^{-1})
=
z^{w_0\rho_{\mathbf r^{\leftarrow}}}\chi_B(\operatorname{GL}_{\mathbf r}).
\]
Combining the two displayed identities yields
\[
w_0\left(
F_{\mathbf r^{\leftarrow},\beta}^{\sigma_0^{-1}}(z;t)
\,
\overline{
E_{\mathbf r^{\leftarrow},\alpha}^{\sigma_0^{-1}}(z;t)
}
\right)=
\left(
z^{w_0\rho_{\mathbf r^{\leftarrow}}}
\chi_B(\operatorname{GL}_{\mathbf r})
\right)
\left(
z^{-w_0\rho_{\mathbf r^{\leftarrow}}}
\chi_C(\operatorname{GL}_{\mathbf r})^*
\right)=
\chi_B(\operatorname{GL}_{\mathbf r})\chi_C(\operatorname{GL}_{\mathbf r})^*.
\]
\end{proof}

\begin{theorem}\label{thmHpnuLLT}
Let $\mu,\nu\vdash n$ satisfy $\mu\gtrdot\nu$.  Write
$\nu=\mu-\epsilon_i+\epsilon_j$ with $i<j$.
Write $\mathbf r=(1^{i-1},2,1^{R-i-1})$ and
$\mathbf r^{\leftarrow}=(1^{R-i-1},2,1^{i-1})$, and retain
$\sigma_0$, $\alpha$, $\beta$, and $\theta^{\mu,\nu}$ from above.  Then
\begin{equation}\label{eqHpnuLLT}
H_{i,\nu}[X;t]
=
t^{h_{\sigma_0}(\beta/\alpha)}
G_{\theta^{\mu,\nu}}[X;t^{-1}].
\end{equation}
\end{theorem}

\begin{proof}
The preceding construction gives
\(\alpha,\beta\in X^{++}(\operatorname{GL}_{\mathbf r^{\leftarrow}})\),
\(\alpha\leq\beta\), and
\(\sigma_0(\beta/\alpha)=\theta^{\mu,\nu}\).
Thus Blasiak et al.~\cite[Theorem~5.5.4]{dens} give
\[
\left(
\mathcal L_{\mathbf r^{\leftarrow},\beta/\alpha}^{\sigma_0}(z;t)
\right)_{\operatorname{pol}}
=
t^{h_{\sigma_0}(\beta/\alpha)}
G_{\theta^{\mu,\nu}}[X;t^{-1}].
\]

On the other hand, since
\(w_0^{R-1}(\mathbf r^{\leftarrow})=\mathbf r\),
Propositions~\ref{propHpnuLevi},
\ref{prop:special-character-factor}, and~\ref{prop:dens-symmetrization} give
\[
\begin{aligned}
H_{i,\nu}[X;t]
&=
\left(
\mathbb H_t^{\mathbf r}
\bigl(
\chi_B(\operatorname{GL}_{\mathbf r})
\chi_C(\operatorname{GL}_{\mathbf r})^*
\bigr)
\right)_{\operatorname{pol}}\\
&=
\left(
\mathbb H_t^{\mathbf r}
\left(
w_0\left(
F_{\mathbf r^{\leftarrow},\beta}^{\sigma_0^{-1}}(z;t)
\,\overline{
E_{\mathbf r^{\leftarrow},\alpha}^{\sigma_0^{-1}}(z;t)
}
\right)
\right)
\right)_{\operatorname{pol}}\\
&=
\left(
\mathcal L_{\mathbf r^{\leftarrow},\beta/\alpha}^{\sigma_0}(z;t)
\right)_{\operatorname{pol}}.
\end{aligned}
\]
Comparing the two displayed identities proves \eqref{eqHpnuLLT}.
\end{proof}

Combining Theorem~\ref{thmstep1} with Theorem~\ref{thmHpnuLLT},
we obtain the following result.
\begin{theorem}\label{thmHallLLT}
    Let \(\mu,\nu\vdash n\) satisfy $\mu\gtrdot\nu$.  Write
    $\nu=\mu-\epsilon_i+\epsilon_j$ with $i<j$.
    Then
\[
\omega\wtH_\mu[X;0,t^{-1}]
-
\omega\wtH_\nu[X;0,t^{-1}]
=
-
t^{
h_{\sigma_0}(\beta/\alpha)
-
A(\theta^{\mu,\nu})
-
n(\nu)
}
G_{(\theta^{\mu,\nu})^R}[X;t].
\]
\end{theorem}

\begin{proof}
Applying \(\omega\) to \eqref{eqstep1} and rearranging yields
\[
\omega\wtH_\mu[X;0,t^{-1}]
-
\omega\wtH_\nu[X;0,t^{-1}]
=
-t^{-n(\nu)}\omega H_{i,\nu}[X;t].
\]

Applying \(\omega\) to Theorem~\ref{thmHpnuLLT} and then using
Proposition~\ref{propomegaLLT} with \(t\) replaced by \(t^{-1}\), we obtain
\[
\omega H_{i,\nu}[X;t]
=
t^{h_{\sigma_0}(\beta/\alpha)}
\omega G_{\theta^{\mu,\nu}}[X;t^{-1}]
=
t^{h_{\sigma_0}(\beta/\alpha)-A(\theta^{\mu,\nu})}
G_{(\theta^{\mu,\nu})^R}[X;t].
\]
Substituting this identity into the first displayed equality proves
the theorem.
\end{proof}

For later use, we denote the exponent in Theorem~\ref{thmHallLLT} by
\begin{equation}\label{eq:emunu}
e^{\mu,\nu}:=h_{\sigma_0}(\beta/\alpha)-A(\theta^{\mu,\nu})-n(\nu).
\end{equation}

\begin{example}\label{ex:hall-llt-cover}
Let $\mu=(3,2,1)\gtrdot\nu=(2,2,2)$.  Here $i=1$, $j=3$,
and $\theta^{\mu,\nu}=((2,2),(2))$.  Its transpose--reversal is
$(\theta^{\mu,\nu})^\dagger=((1^2),(2,2))$.  The skew tuple and
its monomial expansion are displayed side by side.
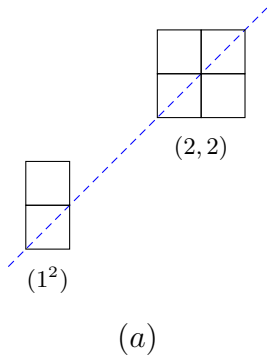
\begin{center}
\begin{minipage}[b]{.40\linewidth}\centering
\begin{tikzpicture}[x=.58cm,y=.58cm]
\draw[blue,densely dashed] (-.4,-.4)--(5.5,5.5);
\draw (0,0) rectangle (1,1);
\draw (0,1) rectangle (1,2);
\begin{scope}[shift={(3,3)}]
\foreach \x in {0,1}{
  \foreach \y in {0,1}{
    \draw (\x,\y) rectangle ++(1,1);
  }
}
\end{scope}
\node[font=\scriptsize] at (.5,-.7) {$(1^2)$};
\node[font=\scriptsize] at (4,2.3) {$(2,2)$};
\end{tikzpicture}\par\smallskip
$(a)$
\end{minipage}\hfill
\begin{minipage}[b]{.58\linewidth}\centering
{\small
\setlength{\jot}{3pt}
\[
\begin{aligned}
G_{((1^2),(2,2))}[X;t] & = t m_{(3^2)} +t(t+1)m_{(3,2,1)}\\
&+t(2t+1)m_{(3,1^3)} +t(2t+1)m_{(2^3)}\\
&+t(t^2+4t+2)m_{(2^2,1^2)}\\
&+t(3t^2+8t+3)m_{(2,1^4)}\\
&+t(9t^2+16t+5)m_{(1^6)}.
\end{aligned}
\]
}\par\smallskip
$(b)$
\end{minipage}
\captionof{figure}{A cover tuple and its monomial expansion.}
\label{fig:hall-cover-tuple}
\end{center}
A direct count gives $h_{\sigma_0}(\beta/\alpha)=3$ and
$A(\theta^{\mu,\nu})=4$, while $n(\nu)=6$.
Theorem~\ref{thmHallLLT} and Proposition~\ref{propomegaLLTdagger}
therefore give
\[
\omega\wtH_{(3,2,1)}[X;0,t^{-1}]
-\omega\wtH_{(2,2,2)}[X;0,t^{-1}]
=-t^{-7}G_{((1^2),(2,2))}[X;t].
\]
To see the subtraction explicitly, we multiply the difference by
$t^6$ and collect its monomial coefficients:
\[
\begin{aligned}
&t^6\bigl(\omega\wtH_{(3,2,1)}[X;0,t^{-1}]
-\omega\wtH_{(2,2,2)}[X;0,t^{-1}]\bigr)\\
={}&-m_{(3^2)}-\bigl[(t^2+t+1)-t^2\bigr]m_{(3,2,1)}\\
&-\bigl[(t^3+2t^2+2t+1)-t^2(t+2)\bigr]
  (m_{(3,1^3)}+m_{(2^3)})\\
&-\bigl[(t^4+3t^3+5t^2+4t+2)-t^2(t^2+3t+4)\bigr]
  m_{(2^2,1^2)}\\
&-\bigl[(t^5+4t^4+9t^3+11t^2+8t+3)
  -t^2(t^3+4t^2+9t+8)\bigr]m_{(2,1^4)}\\
&-\bigl[(t^6+5t^5+14t^4+24t^3+25t^2+16t+5)\\
&\qquad-t^2(t^4+5t^3+14t^2+24t+16)\bigr]m_{(1^6)}\\
={}&-\bigl(m_{(3^2)}+(t+1)m_{(3,2,1)}
 +(2t+1)(m_{(3,1^3)}+m_{(2^3)})\\
&+(t^2+4t+2)m_{(2^2,1^2)}
 +(3t^2+8t+3)m_{(2,1^4)}\\
&+(9t^2+16t+5)m_{(1^6)}\bigr)\\
={}&-t^{-1}G_{((1^2),(2,2))}[X;t].
\end{aligned}
\]
\end{example}

\section{LLT Catalanimals and Cauchy identity}\label{secCat}
In this section, we use the Catalanimal formula of LLT polynomials
and a Cauchy identity to obtain an
LLT-series expansion of the \((r,1)\)-Catalanimal
\(H_{\boldsymbol\nu_r}^{r,1}(z;q,t)\).  Taking polynomial parts of this
expansion and combining the result with Theorem \ref{thmHallLLT}, we prove
Theorem \ref{thmintroNablaHallLLT}, which in turn proves
Theorem \ref{thmmain3} and Corollary \ref{coromain4}.  We conclude
with a worked example of Theorem~\ref{thmNablaHallLLT}.
We use the following results of Blasiak et al.: \cite[Corollary~8.4.2]{LLTcat} and
\cite[Proposition~5.4.3, Theorems~5.5.4 and~6.1.3, and
Lemma~6.2.2]{dens}.
The new input is the verification of their hypotheses for the cover tuple, mainly in Lemmas~\ref{lemalmostmonotonicity}, \ref{lemrootkernel}, \ref{lemresidualweight},
\ref{lemtransport}, and~\ref{lemwinding}.  Table \ref{tabnotation} lists the notation of this section.
\begin{table}[ht]
\centering
\renewcommand{\arraystretch}{1.2}
\begin{tabular}{p{0.3\textwidth}p{0.62\textwidth}}
\hline
Notation & Meaning\\
\hline
$\boldsymbol\nu^{\mu,\nu}$ & cover tuple $(A_1,\cdots,A_u,E,B_1,\cdots,B_v)$; the $A_h$ and $B_k$ are columns, and $E=(2^{\nu_{i+1}},1^{\nu_i-\nu_{i+1}})$\\
\lighthline
$\boldsymbol\nu_r=(C_1,\cdots,C_k)$ & constant-offset $r$-stretch of $\boldsymbol\nu^{\mu,\nu}$; it has $\ell$ boxes\\
\lighthline
$D_{h,c}$, $g_{h,c}$ & set of boxes of $C_h$ of ordinary content $c$, and its size\\
\lighthline
$\operatorname{st}_{h,c}$, $\operatorname{en}_{h,c}$ & $1$ if $D_{h,c}$ contains the initial, respectively terminal, box of a row, and $0$ otherwise\\
\lighthline
$\boldsymbol\kappa^{\mathrm{cat}}$ & Catalanimal weight; it equals $\operatorname{st}_{h,c}-\operatorname{en}_{h,c}+\widetilde b_{h,c}$ on $D_{h,c}$\\
\lighthline
$\gamma$, $\mathbf d$ & lengths of the nonempty diagonals, listed by increasing, respectively decreasing, adjusted content\\
\lighthline
$\mathcal B$, $N$ & boundary blocks $(h,d)$ with $L_h-1\leq d\leq U_h$, and their number\\
\lighthline
$a_B$, $b_B$ & $g_{h,d+1}$ and $g_{h,d}$, for $B=(h,d)$\\
\lighthline
$m_B$, $n_B$, $M_B$ & block minima $-\operatorname{st}_{h,d+1}$ and $-\operatorname{en}_{h,d}$, and the common block maximum\\
\lighthline
$\sigma_r$, $\prec_r$ & permutation and order given by the $r$-order on $\mathcal B$\\
\lighthline
$\theta_r$, $\tau_r$ & standardized restrictions of $\sigma_r^{-1}$ to the blocks with $a_B>0$, respectively $b_B>0$\\
\lighthline
$\boldsymbol\lambda=(\lambda_B)_{B\in\mathcal B}$ & Cauchy index\\
\lighthline
$\alpha(\boldsymbol\lambda)$, $\beta^0(\boldsymbol\lambda)$, $\beta(\boldsymbol\lambda)$ & weights for $\operatorname{GL}_{\mathbf d}$ built from $\boldsymbol\lambda$\\
\lighthline
$\eta_D$, $\widehat\eta$ & $\eta_D=1$ if $c(D)\equiv1\pmod r$, and $\eta_D=0$ otherwise; $\widehat\eta=(\eta_D^{g_D})_D$\\
\hline
\end{tabular}
\caption{Notation of Section \ref{secCat}.}
\label{tabnotation}
\end{table}

\subsection{LLT Catalanimal}\label{subLLTcat}
Blasiak et al.~\cite{LLTcat} gave a Catalanimal formula for LLT polynomials.  Here we use only the special case needed in our setting.

Assume that $\mu\gtrdot\nu$.  Write
$\nu=\mu-\epsilon_i+\epsilon_j$ with $i<j$, and set $a:=\nu_i, b:=\nu_{i+1}$.
Recall
$$\theta^{\mu,\nu}
=
\left(
\frac{\nu_1-1}{-1},
\cdots,
\frac{\nu_{i-1}-1}{-1},
\frac{(a,b)}{(0,0)},
\frac{\nu_{i+2}}0,
\cdots,
\frac{\nu_R}0
\right)$$
and define the cover tuple
$$\boldsymbol\nu^{\mu,\nu} := (\theta^{\mu,\nu})^\dagger =
\Bigg(
\left(\frac{\nu_R}{0}\right)',
\cdots,
\left(\frac{\nu_{i+2}}0\right)',
\left(\frac{(a,b)}{(0,0)}\right)',
\left(\frac{\nu_{i-1}-1}{-1}\right)',
\cdots,
\left(\frac{\nu_1-1}{-1}\right)'
\Bigg).
$$
Empty components in this displayed tuple are omitted.  More precisely,
$$\boldsymbol\nu^{\mu,\nu} = 
(A_1,\cdots,A_u,E,B_1,\cdots,B_v),
$$
where each \(A_h\) and \(B_k\) is a column, $E=(2^b,1^{a-b})$,
$u=|\left\{k: i+2\leq k\leq R, \nu_k>0\right\}|$ and
$v=|\left\{ k: 1\leq k\leq i-1, \nu_k>0 \right\}|$.  Proposition~\ref{propomegaLLTdagger} gives
\begin{equation}\label{eq:cover-dagger}
G_{\boldsymbol\nu^{\mu,\nu}}[X;t]
=
G_{(\theta^{\mu,\nu})^R}[X;t].
\end{equation}

Apply the constant-offset $r$-stretch defined in Section \ref{subLLT} to every component of $\boldsymbol\nu^{\mu,\nu}$, and denote the resulting skew tuple by
$\boldsymbol\nu_r$.  Suppose that \(\boldsymbol\nu_r\) has \(\ell\) boxes, numbered by
$1,2,\cdots,\ell$ in LLT reading order.  We identify each box with its position in this reading order.  All roots and Weyl symmetrizations in this construction are for
\(\operatorname{GL}_\ell\); the symbol \(N\) used below denotes
the number of boundary blocks.  If \(a<b\), write
$\alpha_{ab}=\epsilon_a-\epsilon_b$.

The Catalanimal involves three types of roots:
\[
R_q
=
\left\{
\alpha_{ab}:
c_e(a)<c_e(b)
\right\},
R_t
=
\left\{
\alpha_{ab}:
c_e(a)+1\leq c_e(b)
\right\},
R_{qt}
=
\left\{
\alpha_{ab}:
c_e(a)+1<c_e(b)
\right\}.
\]
Let $D_{h,c}= \left\{x\in \boldsymbol\nu_r^{(h)}: c(x)=c \right\}, g_{h,c}=|D_{h,c}|$.
In other words, $D_{h,c}$ is the set of boxes in the \(h\)th component of
$\boldsymbol\nu_r$ with ordinary content $c$.
Read each row in the direction of increasing ordinary content.  Define
\[
\operatorname{st}_{h,c}
=
\begin{cases}
1,&D_{h,c}\text{ contains the initial box of a row},\\
0,&\text{otherwise},
\end{cases}
\]
and
\[
\operatorname{en}_{h,c}
=
\begin{cases}
1,&D_{h,c}\text{ contains the terminal box of a row},\\
0,&\text{otherwise}.
\end{cases}
\]
Let
\[
\widetilde b_{h,c}
=
\begin{cases}
1,&c\equiv1\pmod r,\\
0,&c\not\equiv1\pmod r.
\end{cases}
\]
Define the Catalanimal weight on the diagonal \(D_{h,c}\) by
$$\kappa_{h,c}^{\mathrm{cat}}
=
\operatorname{st}_{h,c}
-
\operatorname{en}_{h,c}
+
\widetilde b_{h,c}.
$$
Repeating this number \(g_{h,c}\) times gives a weight vector
\(\boldsymbol\kappa^{\mathrm{cat}}\) of length $|\boldsymbol\nu_r|$.  For a diagonal $D=D_{h,c}$, we also write $c(D)=c$, $g_D=g_{h,c}$, $\operatorname{st}_D=\operatorname{st}_{h,c}$, $\operatorname{en}_D=\operatorname{en}_{h,c}$ and $\widetilde b_D=\widetilde b_{h,c}$.  For an integer $v$ and $g\geq0$, we write $v^g$ for the vector $(v,\cdots,v)$ with $g$ entries.

The
\((r,1)\)-Catalanimal is defined by
\[
H_{\boldsymbol\nu_r}^{r,1}(z;q,t)
=
\boldsymbol{\sigma}_\ell
\left(
z^{\boldsymbol\kappa^{\mathrm{cat}}}
\frac{
\displaystyle\prod_{\alpha\in R_{qt}}
(1-qtz^\alpha)
}{
\displaystyle
\prod_{\alpha\in R_q}
(1-qz^\alpha)
\prod_{\alpha\in R_t}
(1-tz^\alpha)
}
\right).
\]

Delete all diagonals of length zero.  Order the remaining diagonals by increasing adjusted content, and write
$\gamma = \left(g_{D_1}, \cdots, g_{D_{\ell(\gamma)}} \right)$, where $D_i$ denotes the \(i\)th diagonal.
For a tuple \(\boldsymbol\mu\) of skew diagrams, let \(D\) range over
its nonempty adjusted diagonals and set
\[
\operatorname{mag}(\boldsymbol\mu)
:=
\sum_D |D|\,\mathbf 1_{\{D\text{ contains no initial box of a row}\}},
\qquad
n_0(\gamma):=\sum_{a=1}^{\ell(\gamma)}\binom{\gamma_a}{2}.
\]
Boxes in the same adjusted diagonal form a block of the Levi subgroup
$\GL_\gamma
=
\GL_{g_{D_1}}
\times\cdots\times
\GL_{g_{D_{\ell(\gamma)}}}$.
Thus \(R_q
=
R^+\setminus R^+(\GL_\gamma)\), where \(R^+=\{\epsilon_u-\epsilon_v:1\leq u<v\leq\ell\}\) is the set of positive roots of \(\operatorname{GL}_\ell\).
Consequently, \(H_{\boldsymbol\nu_r}^{r,1}(z;q,t)\) may be written as
\[H_{\boldsymbol\nu_r}^{r,1}(z;q,t)
=
\mathbb H_q^\gamma
\left(
z^{\boldsymbol\kappa^{\mathrm{cat}}}
\frac{
\displaystyle\prod_{\alpha\in R_{qt}}
(1-qtz^\alpha)
}{
\displaystyle\prod_{\alpha\in R_t}
(1-tz^\alpha)
}
\right).\]

We now specialize the Catalanimal identity to the cover tuple
$\boldsymbol\nu^{\mu,\nu}$.

\begin{theorem}\label{thmLLTcat}
Let $\mu,\nu\vdash n$ satisfy $\mu\gtrdot\nu$.  Write
$\nu=\mu-\epsilon_i+\epsilon_j$ with $i<j$.
For every positive integer \(r\), we have
\begin{equation}\label{eqLLTcat}
\nabla^r
G_{\boldsymbol\nu^{\mu,\nu}}[X;t]
=
-
(qt)^{
1+r(\nu_{i+1}-1)
}
t^{A_r(\boldsymbol\nu^{\mu,\nu})}
\cdot
\omega
\left(
\left(
H_{\boldsymbol\nu_r}^{r,1}(z;t,q)
\right)_{\operatorname{pol}}
\right).
\end{equation}
Here $\nu_{i+1}-1$ is determined by the central component of the
cover tuple, and \(A_r(\boldsymbol\nu^{\mu,\nu})\) denotes the weighted
attacking-pair statistic of Section \ref{subLLT}, computed on
\(\boldsymbol\nu^{\mu,\nu}\) before stretching.
\end{theorem}

\begin{proof}
After interchanging the parameters \(q\) and \(t\),
Blasiak et al.~\cite[Corollary~8.4.2]{LLTcat} give
$$
\omega
\nabla^r
G_{\boldsymbol\nu^{\mu,\nu}}[X;t]
=
(-1)^{
\operatorname{mag}(\boldsymbol\nu^{\mu,\nu})
}
(qt)^{
\operatorname{mag}(\boldsymbol\nu^{\mu,\nu})
+
n_0(\gamma)
}
\cdot
t^{A_r(\boldsymbol\nu^{\mu,\nu})}
\left(
H_{\boldsymbol\nu_r}^{r,1}(z;t,q)
\right)_{\operatorname{pol}}.
$$

We now evaluate the two statistics appearing in this formula.
All components of \(\boldsymbol\nu^{\mu,\nu}\) other than the central
component are columns, and hence every one of their diagonals contains
the initial box of a row.  Recall that the central component is
$E=\left(
2^{\nu_{i+1}},
1^{\nu_i-\nu_{i+1}}
\right)$.
Among the diagonals of \(E\), precisely one diagonal does not contain
the initial box of a row, and this diagonal has length \(1\).  Therefore,
\(\operatorname{mag}(\boldsymbol\nu^{\mu,\nu}) = 1.\)

Moreover, all diagonals in the column components have length \(1\).
The central component \(E\) has precisely $\nu_{i+1}-1$
diagonals of length \(2\), while all of its remaining diagonals have
length \(1\).  Hence, before stretching, the statistic \(n_0\) is equal
to $\nu_{i+1}-1$.  Under the constant-offset \(r\)-stretch, each such diagonal occurs
\(r\) times.  Consequently, $n_0(\gamma) = r\bigl(\nu_{i+1}-1\bigr)$.

Substituting these two evaluations into the preceding formula and then
applying \(\omega\) to both sides proves \eqref{eqLLTcat}.
\end{proof}

\subsection{Cauchy identity}
We next use a Cauchy identity to expand \eqref{eqLLTcat}.  We first
state the conventions needed for it.

For a weak composition $\mathbf r=(r_1,\cdots,r_k)$, a staircase
weight $\rho_{\mathbf r}$ has, in every nonempty $i$th block, the
form $(M_i,M_i-1,\cdots,m_i)$.  If $r_i=0$, the $i$th block is empty, and we still choose integers
$M_i<m_i$.  We call $M_i$ and $m_i$ the block maximum and
minimum, also for an empty block.

For $\sigma\in \mathfrak{S}_k$, a sequence $(m_1,\cdots,m_k)$ is
$\sigma$-almost decreasing if
\[
m_i
\geq
m_j-\mathbf 1_{\{\sigma^{-1}(i)>\sigma^{-1}(j)\}}
\qquad\text{for all }i<j,
\]
and it is $\sigma$-almost increasing if
\[
m_i
\leq
m_j+\mathbf 1_{\{\sigma^{-1}(i)<\sigma^{-1}(j)\}}
\qquad\text{for all }i<j.
\]
With these conventions, the identity takes the following form.

\begin{theorem}[{\cite[Theorem~6.1.3]{dens}}]\label{thmdenscauchy}
    Let $\mathbf r = (r_1,\cdots,r_k)$ and $\mathbf s = (s_1,\cdots,s_k)$ be weak compositions, let
    $\sigma\in \mathfrak{S}_k$, and let $\rho_{\mathbf r}, \rho_{\mathbf s}$ be staircase weights satisfying \cite[Convention~6.1.2]{dens}.
    For each \(i\), let \(M_i\) denote the block maximum of the \(i\)th block of both \(\rho_{\mathbf r}\) and \(\rho_{\mathbf s}\), and let \(m_i\) and \(n_i\) denote their respective block minima.
    Assume that $\rho_{\mathbf r}$ and $\rho_{\mathbf s}$ have the same block maxima \(M_i\), 
    that $(m_1,\cdots,m_k)$ is \(\sigma\)-almost decreasing, and that $(n_1,\cdots,n_k)$ is \(\sigma\)-almost increasing.  Then
    \begin{equation}\label{eqdenscauchy}
    \frac{ \displaystyle\prod_{i<j} \Omega[-qtX_iY_j]}{\displaystyle\prod_{i\leq j}\Omega[-tX_iY_j]}
    =
    \sum_{\boldsymbol\lambda} t^{|\boldsymbol\lambda|}
x^{-\rho_{\mathbf r}}
E_{\mathbf r,\,
\boldsymbol\lambda+\rho_{\mathbf r}}^\sigma
(x;q^{-1})
\,
y^{-\rho_{\mathbf s}}
F_{\mathbf s,\,
\boldsymbol\lambda+\rho_{\mathbf s}}^\sigma
(y;q),
    \end{equation}
    where $\boldsymbol\lambda = \left(
\lambda^{(1)},\cdots,\lambda^{(k)}
\right)$ ranges over all tuples of partitions satisfying
$
\ell\left(\lambda^{(i)}\right)
\leq
\min(r_i,s_i).
$
Here $x=(x_1,\cdots,x_{|\mathbf r|})$ and $y=(y_1,\cdots,y_{|\mathbf s|})$, where $|\mathbf r|=\sum_ir_i$ and $|\mathbf s|=\sum_is_i$.  The alphabet $X_i$ is the sum of the $r_i$ variables $x_a$ with $r_1+\cdots+r_{i-1}<a\leq r_1+\cdots+r_i$, and $Y_j$ is defined in the same way from $y$ and $\mathbf s$.  In $\boldsymbol\lambda+\rho_{\mathbf r}$, each $\lambda^{(i)}$ is padded with zeros to length $r_i$, these vectors are concatenated, and $\rho_{\mathbf r}$ is added; $\boldsymbol\lambda+\rho_{\mathbf s}$ is formed in the same way.  Finally, $|\boldsymbol\lambda|=\sum_i|\lambda^{(i)}|$.
\end{theorem}

We now construct the required data from the stretched cover tuple $\boldsymbol\nu_r$.

Let \(\boldsymbol\nu_r = (C_1,\cdots,C_k)\).  For each box \(x\) in a skew component $C_h$, we consider only its ordinary content $c(x)$.
Let \(L_h=\min_{x\in C_h}c(x), U_h=\max_{x\in C_h}c(x)\).  For every integer $d$, recall
$$D_{h,d} =
\left\{
x\in C_h:c(x)=d
\right\},
g_{h,d}
=
|D_{h,d}|.$$
Define $$\mathcal B_h = \left\{(h,d): L_h-1\leq d\leq U_h\right\}\text{ and }
\mathcal B = \bigcup_h\mathcal B_h.$$ We call the elements of $\mathcal B$ boundary blocks; the block $(h,d)$ lies between the diagonals $D_{h,d}$ and $D_{h,d+1}$.  Let $N=|\mathcal B|$.  The natural order on $\mathcal{B}$ is the order by decreasing adjusted content $c_e(B)=d+h\epsilon$ of $B=(h,d)$, and we number the blocks $1,\cdots,N$ in this order.

For $B=(h,d)\in\mathcal B$, let its source and target be
$D_{h,d}, D_{h,d+1}$, respectively.
Every nonempty diagonal occurs exactly once as a source and once as a
target; the two outer boundary diagonals of each component are empty.
Consequently, both weak compositions below have total size
$|\boldsymbol\nu_r|$.
Set $a_B = g_{h,d+1}, b_B = g_{h,d}, m_B = -\operatorname{st}_{h,d+1}, n_B = -\operatorname{en}_{h,d}, M_B = a_B+m_B-1$.
With respect to the order on $\mathcal{B}$, define the weak compositions
$\mathbf r
=
(a_B)_{B\in\mathcal B},
\mathbf s
=
(b_B)_{B\in\mathcal B}$,
and take the \(B\)-blocks of \(\rho_{\mathbf r}\) and \(\rho_{\mathbf s}\) to be
$$(M_B,M_B-1,\cdots,m_B),\qquad (M_B,M_B-1,\cdots,n_B).$$

\begin{lemma}\label{lemcommonmaxima}
    For every \(B\in\mathcal B\), we have $ a_B+m_B = b_B+n_B$. 
    Consequently, $M_B = b_B+n_B-1$.
\end{lemma}

\begin{proof}
    Since $C_h$ is a skew diagram, the initial boxes of its rows have distinct contents, and so do the terminal boxes.  Thus $\operatorname{st}_{h,c}$ and $\operatorname{en}_{h,c}$ are the numbers of rows of $C_h$ whose initial, respectively terminal, box has content $c$.  When the ordinary content changes from \(d\) to \(d+1\), the rows which enter are precisely those whose initial boxes have content \(d+1\), while the rows which leave are precisely those whose terminal boxes have content \(d\).
    Therefore, $g_{h,d+1}-g_{h,d} = \operatorname{st}_{h,d+1} - \operatorname{en}_{h,d}$, and hence
    $$a_B+m_B = g_{h,d+1} - \operatorname{st}_{h,d+1} = g_{h,d} - \operatorname{en}_{h,d} = b_B+n_B. $$
    In particular, $M_B=m_B-1<m_B$ if $a_B=0$, and $M_B=n_B-1<n_B$ if $b_B=0$, as required by Blasiak et al.~\cite[Convention~6.1.2]{dens}.  This proves the result.
\end{proof}

We define another order on $\mathcal{B}$ for $\boldsymbol\nu_r$, called the \(r\)-order.  For $B = (h,c) \in \mathcal{B}_h$, its ordinary content is defined by $c(B)=c$.

If $r=1$, first order by decreasing ordinary content.  For equal ordinary contents, order by the component containing the block, in the order \(E,\ A_u,\cdots,A_1,\ B_v,\cdots,B_1\).
If $r>1$, let $s\in\{0,\cdots,r-1\}$ be the residue of $c$ modulo $r$.  First order by $s$, in the order $0,\ r-1,\ r-2,\cdots,1$.  Next order by decreasing $\lfloor c/r\rfloor$, and finally order by the component containing the block, in the order \(E,\ A_u,\cdots,A_1,\ B_v,\cdots,B_1\).
Define \(\sigma_r\in \mathfrak{S}_N\) by requiring that, for $1\leq p\leq N$, the $p$th block in the \(r\)-order is the $\sigma_r(p)$th block in the natural order.
For \(P,Q\in\mathcal B\), write $P\prec_r Q$ if \(P\) precedes \(Q\) in the \(r\)-order; equivalently,
$\sigma_r^{-1}(P)<\sigma_r^{-1}(Q)$, where we identify a block with its number in the natural order.

\begin{lemma}\label{lemalmostmonotonicity}
    The block minima of \(\rho_{\mathbf r}\) are \(\sigma_r\)-almost decreasing, and the block minima of \(\rho_{\mathbf s}\) are \(\sigma_r\)-almost increasing.
\end{lemma}

\begin{proof}
    For $P=(h,d)\in\mathcal B$ we have $m_P=-\operatorname{st}_{h,d+1}$ and $n_P=-\operatorname{en}_{h,d}$, so all block minima belong to $\{-1,0\}$.  Thus, \(\sigma_r\)-almost decreasing can fail only for a pair $P,Q$ such that $P$ precedes $Q$ in the natural order, $P\prec_r Q$, and $(m_P,m_Q)=(-1,0)$.  Similarly, \(\sigma_r\)-almost increasing can fail only for a pair $P,Q$ such that $P$ precedes $Q$ in the natural order, $Q\prec_r P$, and $(n_P,n_Q)=(0,-1)$.

    We first find the blocks with a zero minimum.  For a component $Z$ of $\boldsymbol\nu^{\mu,\nu}$, we write $(Z,d)$ for the block $(h,d)$, where $C_h$ is the stretch of $Z$; if $Z$ is a column, let $\nu(Z)$ be the part of $\nu$ that gives $Z$.  Recall that $a=\nu_i\geq b=\nu_{i+1}\geq\nu_j\geq1$.  Since $\nu_k=\mu_k\geq\mu_i=a+1$ for $k\leq i-1$, and $\nu_k\leq b$ for $k\geq i+2$, we have
    \[
    \nu(B_k)\geq a+1,
    \qquad
    \nu(A_h)\leq b\leq a.
    \]
    After stretching, $A_h$ is a column with contents $1-r\nu(A_h),\cdots,0$, $B_k$ is a column with contents $r(1-\nu(B_k))+1,\cdots,r$, and $E$ consists of a column with contents $1-ra,\cdots,0$ and a column with contents $r(1-b)+1,\cdots,r$.  Two boxes of $E$ lie in the same row exactly when their contents are $c-1$ and $c$ with $r(1-b)+1\leq c\leq1$.  Every box of a column is both the initial and the terminal box of its row.  In $E$, the only nonempty diagonal without an initial box is the one of content $1$, and the only nonempty diagonal without a terminal box is the one of content $r(1-b)$.  Therefore
    \begin{align*}
    m_P=0
    &\Longleftrightarrow
    P\in\{(A_h,0),\ (E,0),\ (E,r),\ (B_k,r)\},\\
    n_P=0
    &\Longleftrightarrow
    P\in\{(A_h,-r\nu(A_h)),\ (E,-ra),\ (E,r(1-b)),\ (B_k,r(1-\nu(B_k)))\},
    \end{align*}
    where $h$ and $k$ range over all column components.  In particular, every block with a zero minimum has content divisible by $r$.

    If $P$ precedes $Q$ in the natural order, then $c(P)\geq c(Q)$.  For blocks of equal content, the natural order lists the components as $B_v,\cdots,B_1,E,A_u,\cdots,A_1$, while the $r$-order lists them as $E,A_u,\cdots,A_1,B_v,\cdots,B_1$.  Hence the two orders differ on a pair of blocks of equal content exactly when one block lies in some $B_k$ and the other lies in $E$ or in some $A_h$.  For $r=1$, the $r$-order is the rule for $r>1$ with $r=1$, since every content has residue $0$.

    We now prove that $(m_P)_{P\in\mathcal B}$ is $\sigma_r$-almost decreasing.  Suppose that $P$ precedes $Q$ in the natural order and $(m_P,m_Q)=(-1,0)$.  Every block has content at most $r$, and every block of content $r$ lies in $E$ or in some $B_k$ and has zero minimum.  Since $m_P=-1$ and $c(Q)\leq c(P)\leq r$, we get $c(P)<r$ and hence $c(Q)=0$; so $Q$ lies in $E$ or in some $A_h$.  If $c(P)=0$, then $P=(B_k,0)$ for some $k$, because the other blocks of content $0$ have zero minimum; hence $Q\prec_r P$.  If $c(P)>0$, then $1\leq c(P)\leq r-1$.  In particular $r\geq2$, the residue of $c(P)$ is nonzero, and the residue of $c(Q)$ is $0$; hence again $Q\prec_r P$.  So no pair violates the condition.

    Next we prove that $(n_P)_{P\in\mathcal B}$ is $\sigma_r$-almost increasing.  Suppose that $P$ precedes $Q$ in the natural order, $Q\prec_r P$, and $(n_P,n_Q)=(0,-1)$.  Since $r$ divides $c(P)$, the block $P$ lies in the first residue class of the $r$-order.  Since $Q\prec_r P$, also $r$ divides $c(Q)$, and $\lfloor c(Q)/r\rfloor\geq\lfloor c(P)/r\rfloor$.  Together with $c(P)\geq c(Q)$, this gives $c(P)=c(Q)$.  Hence $P$ lies in some $B_k$ and $Q$ lies in $E$ or in some $A_h$.  As $n_P=0$, we get $P=(B_k,r(1-\nu(B_k)))$, and so $c(Q)=r(1-\nu(B_k))\leq-ra$.  The blocks of $E$ have contents at least $-ra$, and the blocks of $A_h$ have contents at least $-r\nu(A_h)\geq-ra$.  Therefore $c(Q)=-ra$, and $Q$ is either $(E,-ra)$ or $(A_h,-r\nu(A_h))$ with $\nu(A_h)=a$.  In both cases $n_Q=0$, a contradiction.  This completes the proof.
\end{proof}

For a list of $m$ distinct integers, \(\operatorname{std}\) denotes its
standardization, that is, the list of $1,\cdots,m$ in the same relative order; restrictions of inverse lists below are read in the induced
natural order on the indicated blocks.

Define
$\mathcal B_x = \{B\in\mathcal B:a_B>0\}, \mathcal B_y = \{B\in\mathcal B:b_B>0\},$
and
$
\theta_r
=
\operatorname{std}
\left(
\sigma_r^{-1}\big|_{\mathcal B_x}
\right),
\tau_r
=
\operatorname{std}
\left(
\sigma_r^{-1}\big|_{\mathcal B_y}
\right).
$
Use the notation \(\gamma
=
\gamma(\boldsymbol\nu_r)
=
(g_{D_1},\cdots,g_{D_{\ell(\gamma)}}),\) where \(D_1,\cdots,D_{\ell(\gamma)}\) are the nonempty diagonals of \(\boldsymbol\nu_r\), ordered by increasing adjusted content.
Let
$
\mathbf d
=
w_0^{\ell(\gamma)}(\gamma)$, where $w_0^{\ell(\gamma)}$ reverses only the \(\ell(\gamma)\) diagonal blocks.

For every nonempty diagonal $D=D_{h,d}$,
set
\(
\lambda_D^- = \lambda_{(h,d-1)}, \lambda_D^+ = \lambda_{(h,d)},
\)
where \(\boldsymbol\lambda=(\lambda_B)_{B\in\mathcal B}\) is a tuple of partitions with $\ell(\lambda_B)\leq\min(a_B,b_B)$ for every $B\in\mathcal B$, as in Theorem \ref{thmdenscauchy}; we call such a tuple a Cauchy index.  Pad both partitions with zeros to length \(g_D\).

Define
\[
\alpha_D(\boldsymbol\lambda)
=
\lambda_D^-
+
\left(
g_D-\operatorname{st}_D-1,
\cdots,
-\operatorname{st}_D
\right),
\beta_D^0(\boldsymbol\lambda)
=
\lambda_D^+
+
\left(
g_D-\operatorname{en}_D-1,
\cdots,
-\operatorname{en}_D
\right).
\]
Concatenating the blocks in decreasing adjusted-content order gives
\[
\alpha(\boldsymbol\lambda),\
\beta^0(\boldsymbol\lambda)
\in
X^{++}(\GL_{\mathbf d}).
\]

Define further
\[
\beta(\boldsymbol\lambda) = \beta^0(\boldsymbol\lambda)+\widehat\eta,
\text{ where }
\eta=(\eta_D)_D,
\eta_D=\mathbf 1_{\{c(D)\equiv1\pmod r\}},
\widehat\eta=(\eta_D^{g_D})_D.
\]

We shall apply Theorem \ref{thmdenscauchy} with $N$ blocks, indexed by $\mathcal B$ in the natural order, and with $\mathbf r$, $\mathbf s$, $\rho_{\mathbf r}$ and $\rho_{\mathbf s}$ as above.  Since $\sum_Ba_B=\sum_Bb_B=\ell$, the variables of that theorem are $x=(x_1,\cdots,x_\ell)$ and $y=(y_1,\cdots,y_\ell)$.  For $B\in\mathcal B$, we write $X_B$ and $Y_B$ for the alphabets of Theorem \ref{thmdenscauchy} attached to the block $B$; they have $a_B$ and $b_B$ variables, respectively.  In products such as $\prod_{B<C}$, the order is the natural order on $\mathcal B$.  Under the substitution $x=w_0(z^{-1})$, $y=w_0(z)$, the variables of $X_B$ become the $z_v^{-1}$ with $v$ in the target of $B$, and the variables of $Y_B$ become the $z_u$ with $u$ in the source of $B$.  To apply Theorem \ref{thmdenscauchy}, we make the following reductions.

\begin{lemma}\label{lemrootkernel}
    Under the substitution $ x=w_0(z^{-1}), y=w_0(z)$, we have
\[
\frac{
\displaystyle\prod_{B<C}
\Omega[-qtX_BY_C]
}{
\displaystyle\prod_{B\leq C}
\Omega[-tX_BY_C]
}
=
\frac{
\displaystyle\prod_{\alpha\in R_{qt}}
(1-qtz^\alpha)
}{
\displaystyle\prod_{\alpha\in R_t}
(1-tz^\alpha)
}.
\]
\end{lemma}

\begin{proof}
    Let
\(u\in D_{h,d},
v\in D_{k,e+1}.
\)
Under the stated substitution, the corresponding Cauchy denominator factor is
\(
1-tz_u/z_v.
\)
It occurs precisely when
\(
(k,e)\leq(h,d)
\)
in decreasing adjusted boundary-content order.  This is equivalent to
\(
e+k\epsilon
\geq
d+h\epsilon,
\)
or, equivalently,
\(
c_e(u)+1
\leq
c_e(v).
\)
Thus the denominator roots are exactly \(R_t\).

Replacing the weak inequality by the strict inequality gives
\(
c_e(u)+1<c_e(v),
\)
which is exactly the defining condition for \(R_{qt}\).  Hence the numerator roots are precisely \(R_{qt}\).
\end{proof}

\begin{lemma}\label{lemresidualweight}
Let $\zeta\in\mathbb Z^\ell$ be the weight, with blocks given by $\mathbf d$, whose block indexed by a
nonempty diagonal $D$ is $(\operatorname{en}_D-\operatorname{st}_D)^{g_D}$,
where the blocks are read in decreasing adjusted-content order.  Under
$x=w_0(z^{-1})$ and $y=w_0(z)$, we have
\[
x^{-\rho_{\mathbf r}}y^{-\rho_{\mathbf s}}
=
w_0\left(z^{\zeta}\right).
\]
Consequently,
\[
z^{\boldsymbol\kappa^{\mathrm{cat}}}
x^{-\rho_{\mathbf r}}y^{-\rho_{\mathbf s}}
=
w_0\left(z^{\widehat\eta}\right).
\]
\end{lemma}

\begin{proof}
Put $u=w_0(z)$, so that $x=u^{-1}$ and $y=u$.  On the block
corresponding to $D$, the target and source staircases are respectively
\[
(g_D-\operatorname{st}_D-1,\cdots,-\operatorname{st}_D), (g_D-\operatorname{en}_D-1,\cdots,-\operatorname{en}_D).
\]
Their difference is
$(\operatorname{en}_D-\operatorname{st}_D)^{g_D}$, which proves the
first identity.  When the Catalanimal weight is read in the same
$\mathbf d$-order, its $D$-block is
$(\operatorname{st}_D-\operatorname{en}_D+\widetilde b_D)^{g_D}$.
Since the original $z$-coordinates are in increasing adjusted-content
order, the monomial $z^{\boldsymbol\kappa^{\mathrm{cat}}}$ is the $w_0$-image of the monomial whose exponent is this $\mathbf d$-ordered
weight.  Adding the two blockwise weights leaves
$(\widetilde b_D)^{g_D}$, namely $\widehat\eta$.  This proves the
second identity.
\end{proof}

\begin{lemma}\label{lemtransport}
For every Cauchy index $\boldsymbol\lambda$, the substitution
$x=w_0(z^{-1})$, $y=w_0(z)$ gives
\begin{align*}
E_{\mathbf r,\,\boldsymbol\lambda+\rho_{\mathbf r}}^{\sigma_r}
\left(w_0(z^{-1});q^{-1}\right)
&=
w_0\left(
\overline{E_{\mathbf d,\,\alpha(\boldsymbol\lambda)}^{\theta_r^{-1}}(z;q)}
\right),\\
F_{\mathbf s,\,\boldsymbol\lambda+\rho_{\mathbf s}}^{\sigma_r}
\left(w_0(z);q\right)
&=
w_0\left(
F_{\mathbf d,\,\beta^0(\boldsymbol\lambda)}^{\tau_r^{-1}}(z;q)
\right).
\end{align*}
\end{lemma}

\begin{proof}
After deleting the zero target blocks, respectively zero source
blocks, the resulting Levi compositions are both $\mathbf d$.  The
restrictions of $\boldsymbol\lambda+\rho_{\mathbf r}$ and
$\boldsymbol\lambda+\rho_{\mathbf s}$ are respectively
$\alpha(\boldsymbol\lambda)$ and $\beta^0(\boldsymbol\lambda)$.
Moreover, by the definition of standardization,
$\theta_r^{-1}$ and $\tau_r^{-1}$ are precisely the permutations
induced by $\sigma_r$ after these two deletions.  The zero-block
convention for Levi semi-symmetric Hall--Littlewood polynomials
therefore identifies the two factors before specialization with
$E_{\mathbf d,\alpha(\boldsymbol\lambda)}^{\theta_r^{-1}}$ and
$F_{\mathbf d,\beta^0(\boldsymbol\lambda)}^{\tau_r^{-1}}$.
The displayed formulas now follow from the definitions of the bar
involution and of $w_0$.
\end{proof}

\begin{lemma}\label{lemwinding}
    For every Cauchy index $\boldsymbol\lambda$, we have
\[
z^{\widehat\eta}
F_{
\mathbf d,
\beta^0(\boldsymbol\lambda)
}^{\tau_r^{-1}}(z;q)
=
F_{
\mathbf d,
\beta(\boldsymbol\lambda)
}^{\theta_r^{-1}}(z;q).
\]
\end{lemma}

\begin{proof}
For a nonempty diagonal $D=D_{h,c}$, write $D^-=(h,c-1)$ and $D^+=(h,c)$.  Then $D\mapsto D^-$ and $D\mapsto D^+$ are bijections from the set of nonempty diagonals onto $\mathcal B_x$ and $\mathcal B_y$, respectively, and they carry the order by decreasing adjusted content to the natural order.  We index the blocks of $\mathbf d$ by the nonempty diagonals.  By the definition of $\theta_r$ and $\tau_r$, for nonempty diagonals $D_1\neq D_2$ we have
\[
\theta_r(D_1)<\theta_r(D_2)\Longleftrightarrow D_1^-\prec_r D_2^-,
\qquad
\tau_r(D_1)<\tau_r(D_2)\Longleftrightarrow D_1^+\prec_r D_2^+.
\]
Since $\eta$ takes values in $\{0,1\}$,
the lemma of Blasiak et al.~\cite[Lemma~6.2.2]{dens} applies once the following two conditions are checked:
\begin{enumerate}
\item[(a)] if $\eta_{D_1}=\eta_{D_2}$, then $\tau_r(D_1)<\tau_r(D_2)$ if and only if $\theta_r(D_1)<\theta_r(D_2)$;
\item[(b)] if $\eta_{D_1}-\eta_{D_2}=1$, then $\tau_r(D_1)>\tau_r(D_2)$ and $\theta_r(D_1)<\theta_r(D_2)$.
\end{enumerate}

If $r=1$, then $\eta_D=1$ for every $D$, and (b) is empty.  Moreover, $D^+$ is obtained from $D^-$ by increasing the content by $1$ in the same component, which does not change the $1$-order.  Hence $\theta_1=\tau_1$, and (a) holds.

Now let $r\geq2$.  For a block $P$, let $\operatorname{res}(P)\in\{0,\cdots,r-1\}$ be the residue of $c(P)$ modulo $r$, and put $\operatorname{rk}(0)=0$ and $\operatorname{rk}(s)=r-s$ for $1\leq s\leq r-1$.  Thus the $r$-order compares first $\operatorname{rk}(\operatorname{res}(P))$ in increasing order, then $\lfloor c(P)/r\rfloor$ in decreasing order, and then the components.  Note that $c(D^+)=c(D^-)+1$, and that $D^-$ and $D^+$ lie in the same component.

Suppose that $\eta_{D_1}=\eta_{D_2}=1$.  Then $\operatorname{res}(D_1^-)=\operatorname{res}(D_2^-)=0$ and $\operatorname{res}(D_1^+)=\operatorname{res}(D_2^+)=1$, while
\[
\lfloor c(D_1^+)/r\rfloor=\lfloor c(D_1^-)/r\rfloor,
\qquad
\lfloor c(D_2^+)/r\rfloor=\lfloor c(D_2^-)/r\rfloor.
\]
So the comparison of $D_1^+$ and $D_2^+$ uses the same data as that of $D_1^-$ and $D_2^-$, and (a) holds.

Suppose that $\eta_{D_1}=\eta_{D_2}=0$.  Then $\operatorname{res}(D_1^-)$ and $\operatorname{res}(D_2^-)$ lie in $\{1,\cdots,r-1\}$.  For $1\leq s\leq r-1$, the residue of $s+1$ modulo $r$ has rank $\operatorname{rk}(s)-1$.  Hence
\[
\operatorname{rk}(\operatorname{res}(D_1^+))=\operatorname{rk}(\operatorname{res}(D_1^-))-1,
\qquad
\operatorname{rk}(\operatorname{res}(D_2^+))=\operatorname{rk}(\operatorname{res}(D_2^-))-1.
\]
Moreover, for $D=D_1,D_2$, the difference $\lfloor c(D^+)/r\rfloor-\lfloor c(D^-)/r\rfloor$ equals $1$ if $\operatorname{res}(D^-)=r-1$ and $0$ otherwise.  So if $\operatorname{res}(D_1^-)=\operatorname{res}(D_2^-)$, then the two quotients are shifted by the same amount.  In all cases the comparison of $D_1^+$ and $D_2^+$ agrees with that of $D_1^-$ and $D_2^-$, and (a) holds.

Suppose that $\eta_{D_1}=1$ and $\eta_{D_2}=0$.  Then $\operatorname{res}(D_1^-)=0\neq\operatorname{res}(D_2^-)$, so $D_1^-\prec_r D_2^-$ and $\theta_r(D_1)<\theta_r(D_2)$.  Also $\operatorname{res}(D_1^+)=1\neq\operatorname{res}(D_2^+)$, and residue $1$ is the last residue class of the $r$-order; so $D_2^+\prec_r D_1^+$ and $\tau_r(D_1)>\tau_r(D_2)$.  This proves (b).

Finally, $\widehat\eta$ is constant on each block of $\mathbf d$ and $\beta^0(\boldsymbol\lambda)\in X^{++}(\GL_{\mathbf d})$, so $\beta(\boldsymbol\lambda)=\beta^0(\boldsymbol\lambda)+\widehat\eta\in X^{++}(\GL_{\mathbf d})$.  Now the lemma of Blasiak et al.~\cite[Lemma~6.2.2]{dens}, applied with $\theta=\theta_r$, $\tau=\tau_r$ and the weight $\beta(\boldsymbol\lambda)$, gives
\[
F_{
\mathbf d,
\beta(\boldsymbol\lambda)
}^{\theta_r^{-1}}(z;q)
=
z^{\widehat\eta}
F_{
\mathbf d,
\beta(\boldsymbol\lambda)-\widehat\eta
}^{\tau_r^{-1}}(z;q)
=
z^{\widehat\eta}
F_{
\mathbf d,
\beta^0(\boldsymbol\lambda)
}^{\tau_r^{-1}}(z;q).
\qedhere
\]
\end{proof}

\begin{theorem}\label{thmmains}
    For every dominance cover $\mu\gtrdot\nu$ and every positive integer $r$, we have
\begin{equation}\label{eqmains}
H_{\boldsymbol\nu_r}^{r,1}(z;q,t)
=
\sum_{\boldsymbol\lambda}
t^{|\boldsymbol\lambda|}
\mathcal L_{
w_0^{\ell(\gamma)}(\gamma),
\beta(\boldsymbol\lambda)/
\alpha(\boldsymbol\lambda)
}^{\theta_r}(z;q),
\end{equation}
where the sum ranges over all Cauchy indices
$\boldsymbol\lambda = (\lambda_B)_{B\in\mathcal B}$, that is, all tuples of partitions such that $\ell(\lambda_B) \leq \min(a_B,b_B)$ for every \(B\in\mathcal B\).
\end{theorem}

\begin{proof}
We apply Theorem~\ref{thmdenscauchy} with $k=N=|\mathcal B|$, whose $i$th block
is indexed by the $i$th boundary block in the natural order.  Recall that, for $B=(h,d)\in\mathcal B$, the
alphabet $X_B$ is attached to the target diagonal $D_{h,d+1}$, and
the alphabet $Y_B$ is attached to the source diagonal
$D_{h,d}$.  Thus $X_B$ has $a_B=g_{h,d+1}$ variables and $Y_B$ has
$b_B=g_{h,d}$ variables.  We allow either alphabet to be empty when its
corresponding block size is zero.  The relations $B<C$ and $B\leq C$
refer to the natural order on $\mathcal B$.

The two staircase weights are those constructed above: the $B$-block
of $\rho_{\mathbf r}$ is $(M_B,M_B-1,\cdots,m_B)$, while the $B$-block
of $\rho_{\mathbf s}$ is $(M_B,M_B-1,\cdots,n_B)$.  By Lemma
\ref{lemcommonmaxima},
$M_B=a_B+m_B-1=b_B+n_B-1$, so these two blocks have the same maximum.
Lemma \ref{lemalmostmonotonicity} gives the required
$\sigma_r$-almost decreasing condition for the minima $(m_B)_{B\in
\mathcal B}$ and the $\sigma_r$-almost increasing condition for the
minima $(n_B)_{B\in\mathcal B}$.  Hence Theorem
\ref{thmdenscauchy} applies with
$\mathbf r=(a_B)_{B\in\mathcal B}$,
$\mathbf s=(b_B)_{B\in\mathcal B}$, and
$\sigma=\sigma_r$.  It gives
$$
\frac{
\displaystyle\prod_{B<C}
\Omega[-qtX_BY_C]
}{
\displaystyle\prod_{B\leq C}
\Omega[-tX_BY_C]
}
=
\sum_{\boldsymbol\lambda}
t^{|\boldsymbol\lambda|}
x^{-\rho_{\mathbf r}}
E_{\mathbf r,\,
\boldsymbol\lambda+\rho_{\mathbf r}}^{\sigma_r}
(x;q^{-1})
\,
y^{-\rho_{\mathbf s}}
F_{\mathbf s,\,
\boldsymbol\lambda+\rho_{\mathbf s}}^{\sigma_r}
(y;q).
$$

Make the substitution $x=w_0(z^{-1}), y=w_0(z)$,
by Lemmas~\ref{lemrootkernel}, \ref{lemresidualweight}, and~\ref{lemtransport}, we obtain
$$
z^{\boldsymbol\kappa^{\mathrm{cat}}}
\frac{
\displaystyle\prod_{\alpha\in R_{qt}}
(1-qtz^\alpha)
}{
\displaystyle\prod_{\alpha\in R_t}
(1-tz^\alpha)
}
=
\sum_{\boldsymbol\lambda}
t^{|\boldsymbol\lambda|}
w_0
\left(
z^{\widehat\eta}
F_{
\mathbf d,
\beta^0(\boldsymbol\lambda)
}^{\tau_r^{-1}}(z;q)
\,
\overline{
E_{
\mathbf d,
\alpha(\boldsymbol\lambda)
}^{\theta_r^{-1}}(z;q)
}
\right).
$$
By Lemma \ref{lemwinding}, the right-hand side is equal to
$$
\begin{aligned}
\sum_{\boldsymbol\lambda}
t^{|\boldsymbol\lambda|}
w_0
\left(
F_{
\mathbf d,
\beta(\boldsymbol\lambda)
}^{\theta_r^{-1}}(z;q)
\,
\overline{
E_{
\mathbf d,
\alpha(\boldsymbol\lambda)
}^{\theta_r^{-1}}(z;q)
}
\right).
\end{aligned}
$$
On the other hand, $R_q=R^+\setminus R^+(\GL_\gamma)$,
so the definition of the Catalanimal gives
$$
\begin{aligned}
H_{\boldsymbol\nu_r}^{r,1}(z;q,t)
=
\mathbb H_q^\gamma
\left(
z^{\boldsymbol\kappa^{\mathrm{cat}}}
\frac{
\displaystyle\prod_{\alpha\in R_{qt}}
(1-qtz^\alpha)
}{
\displaystyle\prod_{\alpha\in R_t}
(1-tz^\alpha)
}
\right).
\end{aligned}
$$
Therefore,
$$
H_{\boldsymbol\nu_r}^{r,1}(z;q,t)
=
\sum_{\boldsymbol\lambda}
t^{|\boldsymbol\lambda|}
\mathbb H_q^\gamma
\left(
w_0
\left(
F_{
\mathbf d,
\beta(\boldsymbol\lambda)
}^{\theta_r^{-1}}(z;q)
\,
\overline{
E_{
\mathbf d,
\alpha(\boldsymbol\lambda)
}^{\theta_r^{-1}}(z;q)
}
\right)
\right).
$$
Finally,
$
\mathbf d
=
w_0^{\ell(\gamma)}(\gamma)
$
satisfies
$
w_0^{\ell(\gamma)}(\mathbf d)
=
\gamma.
$
By Proposition~\ref{prop:dens-symmetrization}, each summand is equal to
$
\mathcal L_{
w_0^{\ell(\gamma)}(\gamma),
\beta(\boldsymbol\lambda)/
\alpha(\boldsymbol\lambda)
}^{\theta_r}(z;q).
$
Hence
$$
H_{\boldsymbol\nu_r}^{r,1}(z;q,t)
=
\sum_{\boldsymbol\lambda}
t^{|\boldsymbol\lambda|}
\mathcal L_{
w_0^{\ell(\gamma)}(\gamma),
\beta(\boldsymbol\lambda)/
\alpha(\boldsymbol\lambda)
}^{\theta_r}(z;q).
$$
That completes the proof.
\end{proof}

Taking polynomial parts on both sides of Theorem \ref{thmmains} and
applying the theorem of Blasiak et al.~\cite[Theorem~5.5.4]{dens}, we obtain the following positive
LLT-polynomial expansion.

\begin{theorem}\label{thmmainpol}
Let $\mu\gtrdot\nu$ be a dominance cover, and let $r\geq1$ be a
positive integer.  Then
\[
\left(
H_{\boldsymbol\nu_r}^{r,1}(z;q,t)
\right)_{\operatorname{pol}}
=
\sum_{
\substack{
\boldsymbol\lambda\\
\alpha(\boldsymbol\lambda)
\leq
\beta(\boldsymbol\lambda)
}
}
t^{|\boldsymbol\lambda|}
q^{h_{\theta_r}(\beta(\boldsymbol\lambda)/\alpha(\boldsymbol\lambda))}
G_{\theta_r(\beta(\boldsymbol\lambda)/\alpha(\boldsymbol\lambda))}
[X;q^{-1}].
\]
Here $\boldsymbol\lambda=(\lambda_B)_{B\in\mathcal B}$ ranges over the
Cauchy indices satisfying
$\ell(\lambda_B)\leq\min(a_B,b_B)$ for every $B\in\mathcal B$, and the
inequality
$\alpha(\boldsymbol\lambda)\leq\beta(\boldsymbol\lambda)$ is understood
coordinatewise.  The sum is finite.
\end{theorem}

\begin{proof}
By Theorem \ref{thmmains},
\[
H_{\boldsymbol\nu_r}^{r,1}(z;q,t)
=
\sum_{\boldsymbol\lambda}
t^{|\boldsymbol\lambda|}
\mathcal L_{
w_0^{\ell(\gamma)}(\gamma),
\beta(\boldsymbol\lambda)/
\alpha(\boldsymbol\lambda)
}^{\theta_r}(z;q).
\]

The displayed series is interpreted coefficientwise in the \(t\)-adic
completion.  For a fixed power \(t^d\), only the Cauchy indices with
\(\lvert\boldsymbol\lambda\rvert=d\) occur, and there are finitely many
of them because \(\mathcal B\) is finite and every
\(\lambda_B\) has bounded length.  Thus Weyl symmetrization and the
polynomial-part operation may be applied term by term.

For every Cauchy index $\boldsymbol\lambda$, the weights
$\alpha(\boldsymbol\lambda)$ and $\beta(\boldsymbol\lambda)$ are
blockwise regular dominant by construction.  Hence
the theorem of Blasiak et al.~\cite[Theorem~5.5.4]{dens} applies to every summand.  If
$\alpha(\boldsymbol\lambda)
\not\leq
\beta(\boldsymbol\lambda)$,
then the polynomial part of the corresponding summand is zero.  If
$\alpha(\boldsymbol\lambda)
\leq
\beta(\boldsymbol\lambda)$,
then
\[
\left(
\mathcal L_{
\mathbf d,
\beta(\boldsymbol\lambda)/
\alpha(\boldsymbol\lambda)
}^{\theta_r}(z;q)
\right)_{\operatorname{pol}}
=
q^{h_{\theta_r}(\beta(\boldsymbol\lambda)/\alpha(\boldsymbol\lambda))}
G_{\theta_r(\beta(\boldsymbol\lambda)/\alpha(\boldsymbol\lambda))}
[X;q^{-1}].
\]
Taking polynomial parts term by term in the preceding identity yields
the asserted formula.

Finally, only finitely many Cauchy indices satisfy $\alpha(\boldsymbol\lambda)\leq\beta(\boldsymbol\lambda)$, so the sum is finite.  Indeed, let $D=D_{h,d}$ be a nonempty diagonal.  Since $\eta_D=\widetilde b_{h,d}$, the $D$-blocks of $\alpha(\boldsymbol\lambda)$ and $\beta(\boldsymbol\lambda)$ show that $\alpha(\boldsymbol\lambda)\leq\beta(\boldsymbol\lambda)$ holds if and only if
\[
\lambda_{(h,d-1),k}\leq\lambda_{(h,d),k}+\kappa^{\mathrm{cat}}_{h,d}
\qquad(1\leq k\leq g_{h,d})
\]
for every such $D$.  If $D_{h,d}$ is empty, then $a_{(h,d-1)}=0$, and hence $\lambda_{(h,d-1)}=\emptyset$; in particular $\lambda_{(h,U_h)}=\emptyset$.  By descending induction on $d$, every part of every $\lambda_B$ is bounded.
\end{proof}

Combining Theorems~\ref{thmHallLLT}, \ref{thmLLTcat}, and~\ref{thmmainpol}, together with Propositions~\ref{propomegaLLT} and~\ref{propomegaLLTdagger}, we obtain the following formula.
\begin{theorem}\label{thmNablaHallLLT}
Let $\mu,\nu\vdash n$ satisfy $\mu\gtrdot\nu$.  Write
$\nu=\mu-\epsilon_i+\epsilon_j$ with $i<j$.
For every positive integer $r$, we have
\begin{equation}\label{eqNablaHallLLT}
\begin{aligned}
&\nabla^r
\left(
\omega\wtH_\mu[X;0,t^{-1}]
-
\omega\wtH_\nu[X;0,t^{-1}]
\right)\\
&\quad=
(qt)^{
1+r(\nu_{i+1}-1)
}
t^{e^{\mu,\nu}+A_r(\boldsymbol\nu^{\mu,\nu})}
\sum_{
\substack{
\boldsymbol\lambda\\
\alpha(\boldsymbol\lambda)
\leq
\beta(\boldsymbol\lambda)
}
}
q^{|\boldsymbol\lambda|}
t^{\varphi(\boldsymbol\lambda)}
G_{
\left(
\theta_r
\left(
\beta(\boldsymbol\lambda)/
\alpha(\boldsymbol\lambda)
\right)
\right)^R
}
[X;t],
\end{aligned}
\end{equation}
where $e^{\mu,\nu}$ is defined in \eqref{eq:emunu} and
\[
\varphi(\boldsymbol\lambda)
:=
h_{\theta_r}
\left(
\beta(\boldsymbol\lambda)/
\alpha(\boldsymbol\lambda)
\right)
-
A\left(
\theta_r
\left(
\beta(\boldsymbol\lambda)/
\alpha(\boldsymbol\lambda)
\right)
\right).
\]
Here $A\left(
\theta_r
\left(
\beta(\boldsymbol\lambda)/
\alpha(\boldsymbol\lambda)
\right)
\right)$
denotes the total number of attacking pairs of the indicated LLT tuple.
\end{theorem}

\begin{proof}
By Theorem~\ref{thmHallLLT}, we have
\[
\nabla^r
\left(
\omega\wtH_\mu[X;0,t^{-1}]
-
\omega\wtH_\nu[X;0,t^{-1}]
\right)
=
-t^{h_{\sigma_0}(\beta/\alpha)-A(\theta^{\mu,\nu})-n(\nu)}
\nabla^r G_{(\theta^{\mu,\nu})^R}[X;t].
\]
By \eqref{eq:cover-dagger}, this becomes
\[
\nabla^r
\left(
\omega\wtH_\mu[X;0,t^{-1}]
-
\omega\wtH_\nu[X;0,t^{-1}]
\right)
=
-t^{h_{\sigma_0}(\beta/\alpha)-A(\theta^{\mu,\nu})-n(\nu)}
\nabla^r G_{\boldsymbol\nu^{\mu,\nu}}[X;t].
\]
Applying Theorem \ref{thmLLTcat}, the two minus signs cancel, and we
obtain
\begin{multline*}
\nabla^r\!\Bigl(
  \omega\wtH_\mu[X;0,t^{-1}]
  - \omega\wtH_\nu[X;0,t^{-1}]
\Bigr)
= (qt)^{1+r(\nu_{i+1}-1)}
t^{h_{\sigma_0}(\beta/\alpha)-A(\theta^{\mu,\nu})-n(\nu)}
\\
t^{A_r(\boldsymbol\nu^{\mu,\nu})}
\omega\!\Bigl(
  \bigl(H_{\boldsymbol\nu_r}^{r,1}(z;t,q)\bigr)_{\operatorname{pol}}
\Bigr).
\end{multline*}
We now apply Theorem \ref{thmmainpol} after interchanging the
parameters $q$ and $t$.  It gives
\[
\left(
H_{\boldsymbol\nu_r}^{r,1}(z;t,q)
\right)_{\operatorname{pol}}
=
\sum_{
\substack{
\boldsymbol\lambda\\
\alpha(\boldsymbol\lambda)
\leq
\beta(\boldsymbol\lambda)
}
}
q^{|\boldsymbol\lambda|}
t^{
h_{\theta_r}
\left(
\beta(\boldsymbol\lambda)/
\alpha(\boldsymbol\lambda)
\right)
}
G_{
\theta_r
\left(
\beta(\boldsymbol\lambda)/
\alpha(\boldsymbol\lambda)
\right)
}
[X;t^{-1}].
\]

Finally, Proposition~\ref{propomegaLLT}, with $t$ replaced by $t^{-1}$,
gives
$$
\omega
G_{
\theta_r
\left(
\beta(\boldsymbol\lambda)/
\alpha(\boldsymbol\lambda)
\right)
}
[X;t^{-1}]
=
t^{
-
A\left(
\theta_r
\left(
\beta(\boldsymbol\lambda)/
\alpha(\boldsymbol\lambda)
\right)
\right)
}
G_{
\left(
\theta_r
\left(
\beta(\boldsymbol\lambda)/
\alpha(\boldsymbol\lambda)
\right)
\right)^R
}
[X;t].
$$
Substituting the preceding two identities proves
\eqref{eqNablaHallLLT}.
\end{proof}

\begin{proof}[Proof of Theorem~\ref{thmintroNablaHallLLT}]
Part~(1) follows from Theorem~\ref{thmHallLLT} and
\eqref{eq:cover-dagger}.  Part~(2) follows from
Theorem~\ref{thmNablaHallLLT}: the sum there is finite by
Theorem~\ref{thmmainpol}, and every exponent of $q$ is at least
$1+r(\nu_{i+1}-1)\geq1$.
\end{proof}

\begin{corollary}[$=$ Theorem~\ref{thmmain3}]\label{corohalldominanceschur}
For any partitions $\mu,\nu,\lambda\vdash n$ satisfying
$\mu\unrhd\nu$, and every positive integer $r$, we have
\begin{equation}
\begin{aligned}
\left\langle
\nabla^r\omega\wtH_\mu[X;0,t^{-1}]
-
\nabla^r\omega\wtH_\nu[X;0,t^{-1}],
s_\lambda
\right\rangle
&\in
\mathbb{N}[q,t].
\end{aligned}
\end{equation}
\end{corollary}

\begin{proof}
We first consider the case $\mu\gtrdot\nu$.  Theorem~\ref{thmNablaHallLLT},
together with the Schur positivity of LLT polynomials
established by Grojnowski and Haiman~\cite{GrojnowskiHaiman07}, gives
\begin{equation}\label{eqnablahalllaurent}
\begin{aligned}
\left\langle
\nabla^r\omega\wtH_\mu[X;0,t^{-1}]
-
\nabla^r\omega\wtH_\nu[X;0,t^{-1}],
s_\lambda
\right\rangle
&\in
\mathbb{N}[q,t,t^{-1}].
\end{aligned}
\end{equation}

To prove the required integrality, first suppose that \(r=1\).
Then the difference has coefficients in \(\mathbb Z[q,t]\) directly by
Corollary~\ref{corohallschur}.  Now assume \(r\geq2\).  By Bergeron et al.~\cite{BGHT99},
for every partition $\kappa\vdash n$, we have
\(
\left\langle
\nabla^{r-1}s_\kappa,
s_\lambda
\right\rangle
\in
\mathbb{Z}[q,t].\)

Moreover, Corollary \ref{corohallschur} gives, for every
partition $\kappa\vdash n$,
\begin{equation}
\left\langle
\nabla\omega\wtH_\mu[X;0,t^{-1}],
s_\kappa
\right\rangle
\in
\mathbb{N}[q,t],
\left\langle
\nabla\omega\wtH_\nu[X;0,t^{-1}],
s_\kappa
\right\rangle
\in
\mathbb{N}[q,t].
\end{equation}
Hence
\begin{equation}
\left\langle
\nabla\omega\wtH_\mu[X;0,t^{-1}]
-
\nabla\omega\wtH_\nu[X;0,t^{-1}], s_\kappa\right\rangle \in \mathbb{Z}[q,t].
\end{equation}
Since
$\nabla^{r-1}s_\kappa$ has Schur coefficients in
$\mathbb{Z}[q,t]$, it follows that
\begin{equation}\label{eqnablahallintegral}
\left\langle
\nabla^r\omega\wtH_\mu[X;0,t^{-1}]
-
\nabla^r\omega\wtH_\nu[X;0,t^{-1}],
s_\lambda
\right\rangle
\in
\mathbb{Z}[q,t].
\end{equation}

Combining \eqref{eqnablahalllaurent} with
\eqref{eqnablahallintegral}, we obtain
\begin{equation}
\left\langle
\nabla^r\omega\wtH_\mu[X;0,t^{-1}]
-
\nabla^r\omega\wtH_\nu[X;0,t^{-1}],
s_\lambda
\right\rangle
\in
\mathbb{N}[q,t,t^{-1}]
\cap
\mathbb{Z}[q,t]
=
\mathbb{N}[q,t].
\end{equation}

Now let $\mu\unrhd\nu$.  If $\mu=\nu$, the result is immediate.
Otherwise, choose a saturated chain
$
\mu
=
\mu^{(0)}
\gtrdot
\mu^{(1)}
\gtrdot
\cdots
\gtrdot
\mu^{(m)}
=
\nu.
$
Applying the cover case to every consecutive pair in this chain, we
obtain
\begin{equation}
\begin{aligned}
&
\left\langle
\nabla^r\omega\wtH_\mu[X;0,t^{-1}]
-
\nabla^r\omega\wtH_\nu[X;0,t^{-1}],
s_\lambda
\right\rangle
\\
= &
\sum_{k=1}^{m}
\left\langle
\nabla^r\omega\wtH_{\mu^{(k-1)}}[X;0,t^{-1}]
-
\nabla^r\omega\wtH_{\mu^{(k)}}[X;0,t^{-1}],
s_\lambda
\right\rangle
\in
\mathbb{N}[q,t].
\end{aligned}
\end{equation}
This proves the corollary.
\end{proof}

We finally prove Corollary \ref{coromain4}.

\begin{corollary}[$=$ Corollary~\ref{coromain4}]\label{corohallallpowers}
For any partitions $\mu,\lambda\vdash n$ and every positive integer $r$, we have
\begin{equation}
\left\langle
\nabla^r\omega\wtH_\mu[X;0,t^{-1}],
s_\lambda
\right\rangle
\in
\mathbb{N}[q,t].
\end{equation}
\end{corollary}

\begin{proof}
The only composition $\beta$ with $\beta\leq(1^n)$ is $(1^n)$, and
\[
\doff_{(1^n)}(\DP(1^n))=\sum_{k=1}^{n}(n-k)=\binom n2=n(1^n).
\]
Hence \eqref{eqhallomegaC} gives $\omega\wtH_{(1^n)}[X;0,t^{-1}]=C_{(1^n)}[X;t]$, and Proposition \ref{QZ26schur} gives
\[
\left\langle
\nabla^r\omega\wtH_{(1^n)}[X;0,t^{-1}],
s_\lambda
\right\rangle
\in
\mathbb{N}[q,t].
\]
Since $\mu\unrhd(1^n)$, Corollary \ref{corohalldominanceschur} with $\nu=(1^n)$ shows that
\[
\left\langle
\nabla^r\omega\wtH_\mu[X;0,t^{-1}]
-
\nabla^r\omega\wtH_{(1^n)}[X;0,t^{-1}],
s_\lambda
\right\rangle
\in
\mathbb{N}[q,t].
\]
Adding the two displays proves the corollary.
\end{proof}

\subsection{Example}\label{subsec:cover-llt-example}
We continue Example~\ref{ex:hall-llt-cover} with $r=1$ and apply
Theorem~\ref{thmNablaHallLLT}.  In the notation of
Section~\ref{subLLTcat}, the tuple in
Figure~\ref{fig:hall-cover-tuple} is
$\boldsymbol\nu^{\mu,\nu}=(A_1,E)$, where $A_1=(1^2)$ and
$E=(2,2)$.  Its constant-offset stretch is unchanged.  The boxes
of $A_1$ have contents $-1,0$; those of $E$ have contents
$-1,0,0,1$.  Figure~\ref{fig:four-llt-tuples} shows the four LLT
tuples in the expansion.
Ordering the nonempty diagonals by increasing adjusted content gives
$\gamma=(1,1,1,2,1)$, hence $\mathbf d=(1,2,1,1,1)$.
The boundary blocks in natural order are
\[
(E,1),(E,0),(A_1,0),(E,-1),(A_1,-1),(E,-2),(A_1,-2),
\]
with corresponding $(a_B,b_B)$ equal to
$(0,1),(1,2),(0,1),(2,1),(1,1),(1,0),(1,0)$.
The $1$-order agrees with this order, so
$\sigma_1=\theta_1=\mathrm{id}$ and $\eta_D=1$ on every
nonempty diagonal.

In the Cauchy sum, the coordinatewise inequality
$\alpha(\boldsymbol\lambda)\leq\beta(\boldsymbol\lambda)$
forces every boundary partition to be empty except
$\lambda_{(E,-1)}=(x)$ and $\lambda_{(A_1,-1)}=(y)$, where
$x,y\in\{0,1\}$.  Indeed, the block $(E,1)$ forces
$\lambda_{(E,0)}=\emptyset$, while the blocks $(E,0)$ and
$(A_1,0)$ give $x\leq1$ and $y\leq1$; all other boundary
partitions have length bound zero.  Reading the nonempty diagonals
in decreasing adjusted-content order, namely
$(E,1),(E,0),(A_1,0),(E,-1),(A_1,-1)$, gives
\[
\alpha(\boldsymbol\lambda)=(0,x,-1,y-1,-1,-1),
\qquad
\beta(\boldsymbol\lambda)=(0,1,0,0,x+1,y).
\]
For each block of length $g_D$, the $j$th row of the associated
skew diagram has outer length $\beta_{D,j}+j$ and inner length
$\alpha_{D,j}+j$, by the row-endpoint rule in
Section~\ref{subLLTconn}.  Thus the tuple before applying
$R$ is
\[
\theta_1\bigl(\beta(\boldsymbol\lambda)/
\alpha(\boldsymbol\lambda)\bigr)
=
\bigl((1)/(1),\ (2,2)/(x+1,1),\ (1)/(y),\
(x+2),\ (y+1)\bigr).
\]
We retain the empty component to display all five components
produced by the construction.
\begin{center}
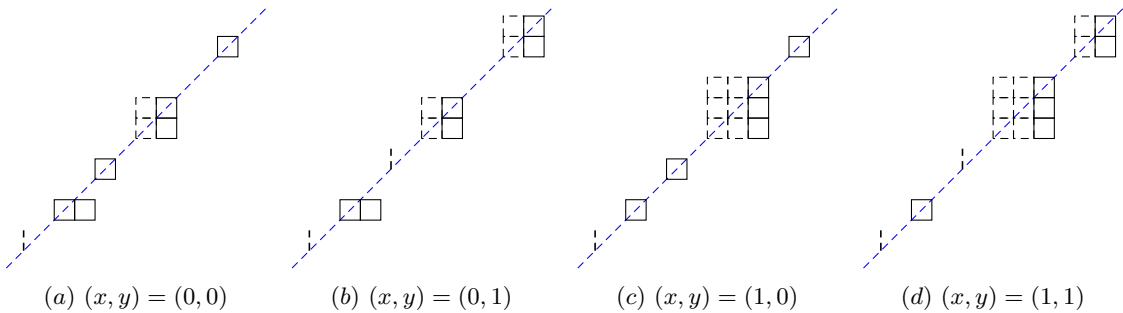

\begin{tikzpicture}[x=.27cm,y=.27cm]
\foreach \xx/\yy/\base/\panel in
  {0/0/0/a,0/1/14/b,1/0/28/c,1/1/42/d}{
\begin{scope}[shift={(\base,0)}]
\draw[blue,densely dashed] (-.35,-.35)--(12.4,12.4);
\draw[black,densely dashed,semithick] (.5,.5)--(.5,1.5);
\begin{scope}[shift={(2,2)}]
\draw (0,0) rectangle ++(1,1);
\ifnum\xx=0 \draw (1,0) rectangle ++(1,1);\fi
\end{scope}
\begin{scope}[shift={(4,4)}]
\ifnum\yy=0
  \draw (0,0) rectangle ++(1,1);
\else
  \draw[black,densely dashed,semithick] (.5,.5)--(.5,1.5);
\fi
\end{scope}
\begin{scope}[shift={(6,6)}]
\ifnum\xx=0
  \foreach \row in {0,1}{
    \draw[black,densely dashed] (0,\row) rectangle ++(1,1);
    \draw (1,\row) rectangle ++(1,1);}
\else
  \foreach \row in {0,1,2}{
    \foreach \col in {0,1}{
      \draw[black,densely dashed] (\col,\row) rectangle ++(1,1);}
    \draw (2,\row) rectangle ++(1,1);}
\fi
\end{scope}
\begin{scope}[shift={(10,10)}]
\ifnum\yy=0
  \draw (0,0) rectangle ++(1,1);
\else
  \foreach \row in {0,1}{
    \draw[black,densely dashed] (0,\row) rectangle ++(1,1);
    \draw (1,\row) rectangle ++(1,1);}
\fi
\end{scope}
\node[font=\scriptsize] at (6,-1.8)
  {$(\panel)\ (x,y)=(\xx,\yy)$};
\end{scope}}
\end{tikzpicture}
\captionof{figure}{Four LLT skew tuples.}\label{fig:four-llt-tuples}
\end{center}

The remaining weights are small enough to calculate directly.
By Example~\ref{ex:hall-llt-cover} and
$A_1(\boldsymbol\nu^{\mu,\nu})=4$, the exponent
$e^{\mu,\nu}+A_1(\boldsymbol\nu^{\mu,\nu})$ in Theorem~\ref{thmNablaHallLLT} equals $-3$, so the prefactor is
$(qt)^2t^{-3}=q^2t^{-1}$.  For $x=y=0$, the six component pairs
$2<3,2<4,2<5,3<4,3<5,4<5$ contribute
$2,2,2,1,1,2$ to $h_{\theta_1}$, giving $10$; their attacking-pair
contributions are $2,3,2,1,1,2$, giving $11$.
The coefficient for each $(x,y)$ is
$q^{2+|\boldsymbol\lambda|}
t^{-1+h_{\theta_1}(\beta(\boldsymbol\lambda)/\alpha(\boldsymbol\lambda))
-A(\theta_1(\beta(\boldsymbol\lambda)/\alpha(\boldsymbol\lambda)))}$.
The other values, obtained in the same way, are
\[
\begin{array}{cc|c|c|c|c}
x&y&|\boldsymbol\lambda|&h_{\theta_1}&\#\text{attacking pairs}&
\text{coefficient}\\ \hline
0&0&0&10&11&q^2t^{-2}\\
0&1&1&7&9&q^3t^{-3}\\
1&0&1&7&7&q^3t^{-1}\\
1&1&2&5&6&q^4t^{-2}
\end{array}
\]
We therefore obtain the four-term
LLT expansion
\begin{equation}\label{eq:worked-cover-llt}
\begin{aligned}
&\nabla\bigl(
\omega\wtH_{(3,2,1)}[X;0,t^{-1}]
-\omega\wtH_{(2,2,2)}[X;0,t^{-1}]
\bigr)=q^2t^{-2}
G_{(\emptyset,(2),(1),(2,2)/(1,1),(1))}[X;t]\\
&\quad+q^3t^{-3}
G_{(\emptyset,(2),\emptyset,(2,2)/(1,1),(2,2)/(1,1))}[X;t]+q^3t^{-1}
G_{(\emptyset,(1),(1),(3,3,3)/(2,2,2),(1))}[X;t]\\
&\quad+q^4t^{-2}G_{(\emptyset,(1),\emptyset,(3,3,3)/(2,2,2),(2,2)/(1,1))}[X;t].
\end{aligned}
\end{equation}
The four coefficients $q^at^b$ in this LLT expansion all have $b<0$, as
Theorem~\ref{thmNablaHallLLT} permits; nevertheless, the Schur coefficients of the total
sum lie in $\mathbb N[q,t]$ by Corollary~\ref{corohalldominanceschur}.

\section{Generalization to the Schiffmann algebra setting}\label{sec:schiffmann-coprime}
This section extends Corollary~\ref{thmmain1} and Theorem~\ref{thmmain1.5}
to the Schiffmann algebra setting.  It also extends
Theorems~\ref{thmHallLLT} and~\ref{thmNablaHallLLT}.
Lemma~\ref{lem:coprime-pointwise-bound} supplies
the integrality needed for the generalizations of
Theorem~\ref{thmmain3} and Corollary~\ref{coromain4}.

\subsection{Generalization of Theorem~\ref{thmmain1.5}}

Let $r,s$ be coprime positive integers.  We use the alphabet
$X^{r,s}$ of Blasiak et al.~\cite[Section~3.2]{LLTcat}: a symmetric function $f[X]$
determines an element $f[-M X^{r,s}]$ of the Schiffmann algebra (the
elliptic Hall algebra of Burban and Schiffmann~\cite{BS12}), which
acts on $1\in\Lambda(X)$.  This construction is linear in $f$, and
$f[-M X^{r,s}]\cdot1$ has symmetric-function degree $s\deg(f)$.
Blasiak et al.~\cite[Proposition~2.3.3]{dens} give
\[
f[-M X^{r,1}]\cdot1=\nabla^r f[X].
\]
Thus every identity in Section~\ref{subLLTconn} may first be read in
the Schiffmann algebra.

Bergeron, Garsia, Leven, and Xin~\cite[Conjecture~3.3]{BGLX16}
formulated the compositional $(kr,ks)$-shuffle conjecture, and Mellit~\cite{Mel21}
proved it.  As in Theorem~\ref{kmknshuffle}, we state it with $q$ and $t$ interchanged.
\begin{proposition}[{\cite{BGLX16,Mel21}}]\label{prop:coprime-compositional-shuffle}
For every composition $\alpha\models k$,
\begin{equation}\label{eq:coprime-compositional-shuffle}
C_\alpha[-M X^{r,s};t]\cdot1
=
\sum_{\substack{\pi\in\LD_{kr,ks}\\
\comp(\pi)=(s\alpha_1,\ldots,s\alpha_{\ell(\alpha)})}}
q^{\area(\pi)}t^{\dinv(\pi)}F_{\ides(\pi)}[X].
\end{equation}
For every partition $\lambda\vdash ks$, the coefficient of $s_\lambda[X]$
in \eqref{eq:coprime-compositional-shuffle} belongs to $\mathbb N[q,t]$.
\end{proposition}

The positivity assertion follows as in the proof of Corollary~\ref{corohallschur}.  For a fixed path $D$, the sum over its labelings in \eqref{eq:coprime-compositional-shuffle} is $q^{\area(D)}t^{\pdinv(D)-\maxtdinv(D)}$ times an LLT polynomial, so the Schur positivity of LLT polynomials established by Grojnowski and Haiman~\cite{GrojnowskiHaiman07} puts the coefficient of $s_\lambda[X]$ in $\mathbb N[q,t,t^{-1}]$.  Every $\pi\in\LD_{kr,ks}$ satisfies $\comp(\pi)=(s\beta_1,\ldots,s\beta_{\ell(\beta)})$ for some $\beta\models k$, and $\beta\leq(k)$.  Since $n((k))=0$ and $\doff_{(k)}(\DP(\beta))=0$, Lemma~\ref{lem:coprime-pointwise-bound} below, applied with $\mu=(k)$, gives $\dinv(\pi)\geq0$.  Hence the fundamental quasisymmetric expansion has monomial coefficients in $\mathbb N[q,t]$, and the integrality of the inverse Kostka matrix puts the coefficient of $s_\lambda[X]$ in $\mathbb Z[q,t]$.  Since $\mathbb N[q,t,t^{-1}]\cap\mathbb Z[q,t]=\mathbb N[q,t]$, the assertion follows.

Applying Proposition~\ref{prop:coprime-compositional-shuffle} to \eqref{eqhallC} and \eqref{eqhallomegaC}
is shorthand for
applying $f\mapsto f[-M X^{r,s}]\cdot1$ to each side; the resulting formulas are
recorded explicitly below.

\begin{theorem}\label{thm:coprime-hall-parking}
For every partition $\mu\vdash k$,
\begin{equation}
(-1)^{k-\ell(\mu)}
\wtH_\mu[-M X^{r,s};0,t]\cdot1
=t^{n(\mu)+k-\ell(\mu)}
\sum_{\substack{\pi\in\LD_{kr,ks}\\
\comp(\pi)=(s\mu_1,\ldots,s\mu_{\ell(\mu)})}}
q^{\area(\pi)}t^{\dinv(\pi)}F_{\ides(\pi)}[X].
\end{equation}
\begin{equation}\label{eq:coprime-signed-hall-parking}
\bigl(\omega\wtH_\mu[X;0,t^{-1}]\bigr)[-M X^{r,s}]\cdot1
=
\sum_{\mathclap{\substack{\beta\leq\overleftarrow\mu \\ \pi\in\LD_{kr,ks}\\
\comp(\pi)=(s\beta_1,\ldots,s\beta_{\ell(\beta)})}}}
q^{\area(\pi)}
t^{\doff_{\overleftarrow\mu}(\DP(\beta))-n(\mu)+\dinv(\pi)}
F_{\ides(\pi)}[X].
\end{equation}
In particular,
\[
\left\langle
(-1)^{k-\ell(\mu)}
\wtH_\mu[-M X^{r,s};0,t]\cdot1,s_\lambda
\right\rangle\in\mathbb N[q,t]. \]
\end{theorem}

\begin{proof}
Apply $f\mapsto f[-M X^{r,s}]\cdot1$ to \eqref{eqhallC} and
\eqref{eqhallomegaC}, and then substitute
\eqref{eq:coprime-compositional-shuffle}.  The exponent
$n(\mu)+k-\ell(\mu)$ is nonnegative.  The last assertion therefore
follows from Proposition~\ref{prop:coprime-compositional-shuffle}.
\end{proof}

\subsection{Generalization of Theorem~\ref{thmNablaHallLLT}}
\label{subsec:general-schiffmann-slope}

Applying $f\mapsto f[-M X^{r,s}]\cdot1$ to Theorem~\ref{thmHallLLT} gives the
following identity, which we use below.

\begin{proposition}\label{prop:coprime-cover-transport}
If $\mu\gtrdot\nu$ and $\boldsymbol\nu^{\mu,\nu}$ is the cover tuple
from Section~\ref{subLLTcat}, then for every coprime $(r,s)$,
\begin{equation}\label{eq:coprime-cover-transport}
\bigl(
\omega\wtH_\mu[X;0,t^{-1}]
-\omega\wtH_\nu[X;0,t^{-1}]
\bigr)[-M X^{r,s}]\cdot1
=-t^{e^{\mu,\nu}}
G_{\boldsymbol\nu^{\mu,\nu}}[-M X^{r,s};t]\cdot1,
\end{equation}
where $e^{\mu,\nu}$ is defined in \eqref{eq:emunu}.
\end{proposition}

\begin{proof}
Apply the linear map $f\mapsto f[-M X^{r,s}]\cdot1$ to
Theorem~\ref{thmHallLLT}, using
\eqref{eq:cover-dagger} to identify the LLT tuple.
\end{proof}

Fix a dominance cover $\mu\gtrdot\nu$, with
$\nu=\mu-\epsilon_i+\epsilon_j$ and $i<j$, and let $r,s$ be
coprime positive integers.  We retain the cover tuple
$\boldsymbol\nu^{\mu,\nu}$, its constant-offset $r$-stretch
$\boldsymbol\nu_r$, and all the boundary-block and Cauchy data of
Section~\ref{secCat}.  The parameter $s$ enters the following two
conventions.  We set
\begin{equation}\label{eq:coprime-slope-weight}
\widetilde b_{h,c}
=
\left\lceil\frac{cs}{r}\right\rceil
-
\left\lceil\frac{(c-1)s}{r}\right\rceil,
\end{equation}
and order the boundary blocks first by increasing residue of
$-s\,c(B)$ modulo $r$, taken in $\{0,\ldots,r-1\}$, then by
decreasing $\lfloor c(B)/r\rfloor$, and finally by the same
component order as before.  The formulas defining
$\boldsymbol\kappa^{\mathrm{cat}}$, $\eta_D=\widetilde b_D$,
$\widehat\eta$, $\prec_r$, $\sigma_r$, $\theta_r$, $\tau_r$,
and $\beta(\boldsymbol\lambda)$ are read with these
conventions.  At $s=1$ they agree with the preceding definitions;
when $r=1$, the order is the original $1$-order and
$\eta_D=s$ on every diagonal.  The LLT Catalanimal
$H_{\boldsymbol\nu_r}^{r,s}(z;q,t)$ of Blasiak et al.~\cite[Definition~8.2.1]{LLTcat}
has the same root sets $R_q,R_t,R_{qt}$ as
$H_{\boldsymbol\nu_r}^{r,1}$ and the weight
$\boldsymbol\kappa^{\mathrm{cat}}$ specified by
\eqref{eq:coprime-slope-weight}.

\begin{lemma}\label{lem:coprime-slope-order}
With these conventions, the block minima of
$\rho_{\mathbf r}$ and $\rho_{\mathbf s}$ are, respectively,
$\sigma_r$-almost decreasing and $\sigma_r$-almost increasing.
The diagonal weights $\eta_D$ and the permutations
$\theta_r,\tau_r$ satisfy the hypotheses of
Blasiak et al.~\cite[Lemma~6.2.2]{dens}, after the same constant is subtracted from every $\eta_D$.
\end{lemma}

\begin{proof}
The proof of Lemma~\ref{lemalmostmonotonicity} identifies every
block with a zero minimum; its content is divisible by $r$.
As $\gcd(r,s)=1$, these blocks lie in residue class $0$,
which consists precisely of the contents divisible by $r$.
For an almost-decreasing
violation, the argument of that lemma forces the zero-minimum
block to have content $0$ and the other block to have content
between $0$ and $r-1$.  A positive content has nonzero
residue and follows the zero-minimum block; when both contents
are $0$, the comparison of components gives the same conclusion.  For an
almost-increasing violation, the first block has content
divisible by $r$.  The order can reverse the two blocks only
when the second content is also divisible by $r$ and is at
least as large.  The natural order gives the opposite
inequality, so the contents agree; the final component
comparison is the same contradiction as in
Lemma~\ref{lemalmostmonotonicity}.

For $r>1$, the entries in \eqref{eq:coprime-slope-weight}
are $\lfloor s/r\rfloor$ or $\lfloor s/r\rfloor+1$.
Writing residues in $\{0,\ldots,r-1\}$, the larger value
occurs exactly when the residue of $-s(c(D)-1)$ is smaller
than the residue of $s$.  On passing from the boundary
block before $D$ to the one after $D$, its residue changes
by minus the residue of $s$ modulo $r$.  Thus the
larger-weight diagonals precede the smaller-weight
diagonals in $\theta_r$ and follow them in $\tau_r$.
Within either weight class this translation preserves
residue order; for equal residues, both contents have
the same quotient change, and the comparison of components is
preserved.  Thus conditions (a) and (b) in the proof of Lemma~\ref{lemwinding}
hold.  Moreover, $\widehat\eta$
is constant on each Levi block, so
$\beta(\boldsymbol\lambda)=
\beta^0(\boldsymbol\lambda)+\widehat\eta$
remains blockwise regular dominant.  If $r=1$, all diagonal
weights equal $s$ and $\theta_1=\tau_1$, so the same
conditions hold.  For the weak-composition form needed here, see
Blasiak et al.~\cite[Remark~6.2.3]{dens}.

Finally, the lemma of Blasiak et al.~\cite[Lemma~6.2.2]{dens} is stated for weights with values in $\{0,1\}$.  Put $c_0=\lfloor s/r\rfloor$.  By the preceding paragraph, every $\eta_D$ lies in $\{c_0,c_0+1\}$ (and $\eta_D=c_0=s$ if $r=1$).  Thus the weights $\eta_D-c_0$ take values in $\{0,1\}$, and they satisfy (a) and (b), because these conditions involve only differences of weights.  Since $(z_1\cdots z_\ell)^{c_0}$ is symmetric, multiplication by it commutes with every $T_i$, every $T_i^{-1}$, and $\delta_{\mathbf d}$.  Adding $(c_0,\ldots,c_0)$ to a weight $\lambda$ does not change $\operatorname{Inv}(\lambda+\varepsilon\rho)$, and the bar involution sends $(z_1\cdots z_\ell)^{-c_0}$ to $(z_1\cdots z_\ell)^{c_0}$.  Hence the definitions in Section~\ref{subLLTconn} give
\[
F_{\mathbf d,\kappa+(c_0,\ldots,c_0)}^{\xi}(z;q)=(z_1\cdots z_\ell)^{c_0}F_{\mathbf d,\kappa}^{\xi}(z;q)
\]
for every permutation $\xi$ and every $\kappa\in X^{++}(\GL_{\mathbf d})$.  Applying the lemma to the weights $\eta_D-c_0$ and the weight $\beta(\boldsymbol\lambda)-(c_0,\ldots,c_0)$, and then multiplying by $(z_1\cdots z_\ell)^{c_0}$, we obtain
\[
z^{\widehat\eta}
F_{\mathbf d,\beta^0(\boldsymbol\lambda)}^{\tau_r^{-1}}(z;q)
=
F_{\mathbf d,\beta(\boldsymbol\lambda)}^{\theta_r^{-1}}(z;q),
\]
which is the identity used in the proof of Theorem~\ref{thm:coprime-slope-llt}.
\end{proof}

\begin{theorem}\label{thm:coprime-slope-llt}
For every dominance cover $\mu\gtrdot\nu$ and coprime
positive pair $(r,s)$, we have 
\begin{equation}\label{eq:coprime-slope-llt}
\begin{aligned}
G_{\boldsymbol\nu^{\mu,\nu}}[-M X^{r,s};q]\cdot1
=
-(qt)^{1+r(\nu_{i+1}-1)}
q^{A_r(\boldsymbol\nu^{\mu,\nu})}
\sum_{\substack{\boldsymbol\lambda\\
  \alpha(\boldsymbol\lambda)\leq\beta(\boldsymbol\lambda)}}
t^{|\boldsymbol\lambda|}
q^{\varphi(\boldsymbol\lambda)}
G_{\left(\theta_r(\beta(\boldsymbol\lambda)/
                     \alpha(\boldsymbol\lambda))\right)^R}[X;q],
\end{aligned}
\end{equation}
where $\varphi(\boldsymbol\lambda)$ is defined as in Theorem~\ref{thmNablaHallLLT}, now with the conventions of this section.
Here $\boldsymbol\lambda$ ranges over the Cauchy indices
defined above, and the inequality is coordinatewise.
\end{theorem}

\begin{proof}
We first prove
\begin{equation}\label{eq:coprime-slope-series}
H_{\boldsymbol\nu_r}^{r,s}(z;q,t)
=
\sum_{\boldsymbol\lambda}
t^{|\boldsymbol\lambda|}
\mathcal L_{\mathbf d,\,
  \beta(\boldsymbol\lambda)/\alpha(\boldsymbol\lambda)}
^{\theta_r}(z;q),
\end{equation}
coefficientwise in the $t$-adic completion.  The
block sizes and staircase weights in the Cauchy construction
depend only on $\boldsymbol\nu_r$.  Lemma~\ref{lem:coprime-slope-order}
supplies the remaining ordering conditions in
Theorem~\ref{thmdenscauchy}.  The calculation in
Lemma~\ref{lemrootkernel} is unchanged.
Under $x=w_0(z^{-1})$ and $y=w_0(z)$, the staircase
weight cancels the $\operatorname{st}_D-\operatorname{en}_D$
part of $\boldsymbol\kappa^{\mathrm{cat}}$, leaving
$w_0(z^{\widehat\eta})$ as in
Lemma~\ref{lemresidualweight}.  Deleting the zero
source and target blocks gives the two Levi factors
with $\tau_r$ and $\theta_r$, by the argument of
Lemma~\ref{lemtransport}.  Finally, by Blasiak et al.~\cite[Lemma~6.2.2]{dens} and
Lemma~\ref{lem:coprime-slope-order},
\[
z^{\widehat\eta}
F_{\mathbf d,\beta^0(\boldsymbol\lambda)}^{\tau_r^{-1}}(z;q)
=
F_{\mathbf d,\beta(\boldsymbol\lambda)}^{\theta_r^{-1}}(z;q).
\]
Weyl symmetrization, together with
Proposition~\ref{prop:dens-symmetrization}, proves
\eqref{eq:coprime-slope-series}.  For each fixed
$t$-degree, only finitely many Cauchy indices occur.

Taking polynomial parts term by term and applying
the theorem of Blasiak et al.~\cite[Theorem~5.5.4]{dens}, we obtain
\begin{equation}\label{eq:coprime-slope-polynomial-part}
\bigl(H_{\boldsymbol\nu_r}^{r,s}(z;q,t)\bigr)_{\operatorname{pol}}
=
\sum_{\substack{\boldsymbol\lambda\\
  \alpha(\boldsymbol\lambda)\leq\beta(\boldsymbol\lambda)}}
t^{|\boldsymbol\lambda|}
q^{h_{\theta_r}(\beta(\boldsymbol\lambda)/
               \alpha(\boldsymbol\lambda))}
G_{\theta_r(\beta(\boldsymbol\lambda)/
             \alpha(\boldsymbol\lambda))}[X;q^{-1}].
\end{equation}
The sum is finite: for every nonempty diagonal $D_{h,c}$,
the displayed inequality implies
\[
\lambda_{(h,c-1),j}
\leq
\lambda_{(h,c),j}
+\kappa^{\mathrm{cat}}_{h,c}
\qquad(1\leq j\leq g_{h,c}).
\]
The terminal boundary partition is empty, so descending
induction on $c$ bounds every part of every
$\lambda_B$.

By Blasiak et al.~\cite[Theorem~8.3.1 and Remark~8.3.2]{LLTcat},
$H_{\boldsymbol\nu_r}^{r,s}(z;q,t)$ represents a
scalar multiple of
$G_{\boldsymbol\nu^{\mu,\nu}}[-M X^{r,s};q]$.
Its magic number is $1$ and
$n_0(\gamma)=r(\nu_{i+1}-1)$, as computed in the
proof of Theorem~\ref{thmLLTcat}.  Thus the scalar multiplying
$G_{\boldsymbol\nu^{\mu,\nu}}[-M X^{r,s};q]$ in this
representative is
$-(qt)^{-1-r(\nu_{i+1}-1)}q^{-A_r(\boldsymbol\nu^{\mu,\nu})}$.
By Blasiak et al.~\cite[Proposition~2.3.2]{dens}, the polynomial part in
\eqref{eq:coprime-slope-polynomial-part} is the $\omega$-image of
this action; thus
\[
G_{\boldsymbol\nu^{\mu,\nu}}[-M X^{r,s};q]\cdot1
=
-(qt)^{1+r(\nu_{i+1}-1)}
q^{A_r(\boldsymbol\nu^{\mu,\nu})}
\omega\bigl(
  (H_{\boldsymbol\nu_r}^{r,s}(z;q,t))_{\operatorname{pol}}
\bigr).
\]
Applying Proposition~\ref{propomegaLLT} then gives
\eqref{eq:coprime-slope-llt}.
\end{proof}

\begin{theorem}\label{thm:coprime-slope-hall-llt}
Let $\mu,\nu\vdash k$ satisfy $\mu\gtrdot\nu$, and
write $\nu=\mu-\epsilon_i+\epsilon_j$ with $i<j$.
For every coprime positive pair $(r,s)$, the
Hall--Littlewood cover difference has the finite
LLT expansion
\begin{equation}\label{eq:coprime-slope-hall-llt}
\begin{aligned}
&\bigl(
  \omega\wtH_\mu[X;0,t^{-1}]
 -\omega\wtH_\nu[X;0,t^{-1}]
 \bigr)[-M X^{r,s}]\cdot1\\
&\quad=
(qt)^{1+r(\nu_{i+1}-1)}
\sum_{\substack{\boldsymbol\lambda\\
  \alpha(\boldsymbol\lambda)\leq\beta(\boldsymbol\lambda)}}
q^{|\boldsymbol\lambda|}
t^{e^{\mu,\nu}+A_r(\boldsymbol\nu^{\mu,\nu})+\varphi(\boldsymbol\lambda)}
G_{\left(\theta_r(\beta(\boldsymbol\lambda)/
                \alpha(\boldsymbol\lambda))\right)^R}[X;t].
\end{aligned}
\end{equation}
Here $e^{\mu,\nu}$ is defined in \eqref{eq:emunu}, and
$\varphi(\boldsymbol\lambda)$ is as in Theorem~\ref{thm:coprime-slope-llt}.
At $s=1$,
the left-hand side is
$\nabla^r\bigl(
\omega\wtH_\mu[X;0,t^{-1}]
-\omega\wtH_\nu[X;0,t^{-1}]
\bigr)$, and the formula is
Theorem~\ref{thmNablaHallLLT}.
\end{theorem}

\begin{proof}
By Proposition~\ref{prop:coprime-cover-transport}, the left-hand side of
\eqref{eq:coprime-slope-hall-llt} equals
\[
-t^{e^{\mu,\nu}}
G_{\boldsymbol\nu^{\mu,\nu}}[-M X^{r,s};t]\cdot1.
\]
Interchange $q$ and $t$ in
Theorem~\ref{thm:coprime-slope-llt} and multiply by
$-t^{e^{\mu,\nu}}$.  The signs cancel and yield
\eqref{eq:coprime-slope-hall-llt}.  For $s=1$, apply
the result of Blasiak et al.~\cite[Proposition~2.3.3]{dens}.
\end{proof}

To generalize Theorem~\ref{thmmain3} and
Corollary~\ref{coromain4} to the Schiffmann algebra setting,
we need the following lemma.

\begin{lemma}\label{lem:coprime-pointwise-bound}
Let $r,s$ be coprime positive integers, let $\mu\vdash k$, and let
$\beta\leq\overleftarrow\mu$.  If $\pi\in\LD_{kr,ks}$ satisfies
$\comp(\pi)=(s\beta_1,\ldots,s\beta_{\ell(\beta)})$, then
\[
\dinv(\pi)\geq
n(\mu)-\doff_{\overleftarrow\mu}(\DP(\beta)).
\]
\end{lemma}

This is the $(kr,ks)$-version of Lemma~\ref{lemmadoff}.
Its technical proof is given in
Appendix~\ref{app:coprime-pointwise-bound}.

\begin{proposition}\label{prop:coprime-hall-integrality}
For every coprime positive pair $(r,s)$, every $\mu\vdash k$,
and every $\lambda\vdash ks$,
\[
\left\langle
\bigl(\omega\wtH_\mu[X;0,t^{-1}]\bigr)[-M X^{r,s}]\cdot1,
s_\lambda
\right\rangle\in\mathbb Z[q,t].
\]
\end{proposition}

\begin{proof}
Lemma~\ref{lem:coprime-pointwise-bound} shows that every power of
$t$ in \eqref{eq:coprime-signed-hall-parking} is nonnegative.
As in the proof of Corollary~\ref{corohallschur}, the
fundamental-quasisymmetric expansion has monomial coefficients in
$\mathbb N[q,t]$, and the integral inverse Kostka matrix puts its
Schur coefficients in $\mathbb Z[q,t]$.
\end{proof}

\begin{corollary}\label{cor:coprime-hall-polynomial-schur}
For every coprime positive pair $(r,s)$, every
$\mu,\nu\vdash k$ with $\mu\unrhd\nu$, and every
$\lambda\vdash ks$,
\[
\left\langle
\bigl(
\omega\wtH_\mu[X;0,t^{-1}]
-\omega\wtH_\nu[X;0,t^{-1}]
\bigr)[-M X^{r,s}]\cdot1,\,
s_\lambda
\right\rangle\in\mathbb N[q,t].
\]
\end{corollary}

\begin{proof}
We first consider the case $\mu\gtrdot\nu$.
Theorem~\ref{thm:coprime-slope-hall-llt}, together with the
Schur positivity of LLT polynomials established by Grojnowski and Haiman~\cite{GrojnowskiHaiman07}, gives
\begin{equation}\label{eq:coprime-hall-laurent}
\left\langle
\bigl(\omega\wtH_\mu[X;0,t^{-1}]
-\omega\wtH_\nu[X;0,t^{-1}]\bigr)[-M X^{r,s}]\cdot1,
s_\lambda
\right\rangle\in\mathbb N[q,t,t^{-1}].
\end{equation}
Proposition~\ref{prop:coprime-hall-integrality} gives
\[
\begin{aligned}
\left\langle
\bigl(\omega\wtH_\mu[X;0,t^{-1}]\bigr)[-M X^{r,s}]\cdot1,
s_\lambda
\right\rangle&\in\mathbb Z[q,t],
\left\langle
\bigl(\omega\wtH_\nu[X;0,t^{-1}]\bigr)[-M X^{r,s}]\cdot1,
s_\lambda
\right\rangle&\in\mathbb Z[q,t].
\end{aligned}
\]
By linearity, therefore,
\begin{equation}\label{eq:coprime-hall-integral}
\left\langle
\bigl(\omega\wtH_\mu[X;0,t^{-1}]
-\omega\wtH_\nu[X;0,t^{-1}]\bigr)[-M X^{r,s}]\cdot1,
s_\lambda
\right\rangle\in\mathbb Z[q,t].
\end{equation}
Combining \eqref{eq:coprime-hall-laurent} with
\eqref{eq:coprime-hall-integral}, we obtain
\[
\left\langle
\bigl(\omega\wtH_\mu[X;0,t^{-1}]
-\omega\wtH_\nu[X;0,t^{-1}]\bigr)[-M X^{r,s}]\cdot1,
s_\lambda
\right\rangle
\in\mathbb N[q,t,t^{-1}]\cap\mathbb Z[q,t]
=\mathbb N[q,t].
\]
Now let $\mu\unrhd\nu$.  If $\mu=\nu$, the result is immediate.
Otherwise, choose a saturated dominance chain, as in the proof of
Corollary~\ref{corohalldominanceschur}:
\(
\mu=\mu^{(0)}\gtrdot\mu^{(1)}\gtrdot\cdots
\gtrdot\mu^{(m)}=\nu.
\)
Applying the cover case to every consecutive pair gives
\[
\begin{aligned}
&\left\langle
\bigl(\omega\wtH_\mu[X;0,t^{-1}]
-\omega\wtH_\nu[X;0,t^{-1}]\bigr)[-M X^{r,s}]\cdot1,
s_\lambda\right\rangle\\
&\quad=\sum_{j=1}^{m}
\left\langle
\bigl(\omega\wtH_{\mu^{(j-1)}}[X;0,t^{-1}]
-\omega\wtH_{\mu^{(j)}}[X;0,t^{-1}]\bigr)
[-M X^{r,s}]\cdot1,
s_\lambda\right\rangle\in\mathbb N[q,t].
\end{aligned}
\]
This proves the corollary.
\end{proof}

\begin{corollary}\label{cor:coprime-individual-hall-schur}
For every coprime positive pair $(r,s)$, every $\mu\vdash k$,
and every $\lambda\vdash ks$,
\[
\left\langle
\bigl(\omega\wtH_\mu[X;0,t^{-1}]\bigr)[-M X^{r,s}]\cdot1,
s_\lambda
\right\rangle\in\mathbb N[q,t].
\]
\end{corollary}

\begin{proof}
As in the proof of Corollary~\ref{corohallallpowers},
\eqref{eqhallomegaC} gives
$\omega\wtH_{(1^k)}[X;0,t^{-1}]=C_{(1^k)}[X;t]$.
Proposition~\ref{prop:coprime-compositional-shuffle} therefore gives
\[
\left\langle
C_{(1^k)}[-M X^{r,s};t]\cdot1,s_\lambda
\right\rangle\in\mathbb N[q,t].
\]
Since $\mu\unrhd(1^k)$, Corollary~\ref{cor:coprime-hall-polynomial-schur}
with $\nu=(1^k)$ gives
\[
\left\langle
\bigl(\omega\wtH_\mu[X;0,t^{-1}]
-\omega\wtH_{(1^k)}[X;0,t^{-1}]\bigr)[-M X^{r,s}]\cdot1,
s_\lambda
\right\rangle\in\mathbb N[q,t].
\]
Adding the two displays proves the corollary.
\end{proof}

\appendix
\section{Proof of Lemma~\ref{lem:coprime-pointwise-bound}}
\label{app:coprime-pointwise-bound}

In this appendix, we prove Lemma~\ref{lem:coprime-pointwise-bound}
by bounding the contributions within and between complete Dyck segments.

\begin{proof}
A complete Dyck segment is a part of the path that starts and ends on the diagonal; it is primitive if it meets the diagonal only at its two ends.  For a segment with east steps numbered $1,2,\ldots$ in order, its $i$th run is the (possibly empty) sequence of north steps immediately before the $i$th east step.
Let $A$ be an earlier primitive Dyck segment of size $(ra,sa)$ and
$B$ a later complete Dyck segment of size $(rb,sb)$, with $a\le b$.
Measure height by $ry-sx$ relative to each segment's start.  If $e_i$
and $f_j$ are their east-step starting heights, set
$e_0=f_0=s$, $p_i=e_{i-1}-s$, and $q_j=f_{j-1}-s$.  The north steps
in the corresponding runs start at heights
$p_i,p_i+r,\ldots,e_i-r$ and $q_j,q_j+r,\ldots,f_j-r$.

Write $P_{ij},S_{ij}$ for primary and secondary attacks between these
runs.  For heights $x$ in $A$ and $y$ in $B$, these mean respectively
$0\le y-x<r$ and $0<x-y\le r$.  The attacks alternate; since labels
increase within each run, pairing consecutive attacks yields at least
$\min(P_{ij},S_{ij})$ temporary inversions.  By Hicks and Leven~\cite[Lemma~1]{HL15}, the labeling
whose labels increase with the rank satisfies the column condition and realizes every attack.
Thus
\[
\tdinv_{\rm cross}\ge\sum_{i,j}\min(P_{ij},S_{ij}),
\qquad
\maxtdinv_{\rm cross}=\sum_{i,j}(P_{ij}+S_{ij}).
\]

By Hicks and Leven~\cite[Theorem~1]{HL15}, an east step of $A$ at height $e$ and
a north step of $B$ at height $h$ contribute to $\pdinv$ precisely
when $0<e-h\le r+s$: for their cell $c$,
$s a(c)-r\ell(c)=e-h-s$.  Because $A$ starts and ends at height
zero, its upward and downward crossings of $h+r+\tfrac12$ balance.
Hence its north starts $h<x\le h+r$ (the secondary attacks) are
equinumerous with east starts $h+r<e\le h+r+s$.  Summing the
remaining intervals $(h,h+r]$ over the $j$th run gives
\[
\pdinv_{\rm cross}-\sum_{i,j}S_{ij}
=\sum_{i,j}\mathbf1_{\{q_j<e_i\le f_j\}}.
\]
For fixed $y$ in that run, the possible primary and secondary
partners in the $i$th run have heights $x$ and $x+r$, where
$x\equiv p_i\pmod r$ and $y-r<x\le y$.  Interior partners cancel,
leaving the endpoints:
\[
P_{ij}-S_{ij}
=\mathbf1_{\{q_j<e_i\le f_j\}}
-\mathbf1_{\{q_j<p_i\le f_j\}}.
\]
Since $e_i\ge p_i$ and $p_i-q_j=e_{i-1}-f_{j-1}$, these identities give
\[
\begin{aligned}
\dinv_{\rm cross}
&\ge\pdinv_{\rm cross}-\sum_{i,j}\max(P_{ij},S_{ij})\\
&=\#\{(i,j):q_j<p_i\le e_i\le f_j\}\\
&=\#\{(i,j):e_{i-1}>f_{j-1},\ e_i\le f_j\}.
\end{aligned}
\]
Primitivity gives $e_i>s$ for $i<ra$ and $e_{ra}=s$, whereas
$f_j\ge s$.  For each $1\le i<ra$, the diagonal list
\(
(e_i,f_0),(e_{i+1},f_1),\ldots,(e_{ra},f_{ra-i})
\)
exists since $rb\ge ra$, and changes from $>$ to $\le$.  These
disjoint lists supply at least $ra-1\ge a-1$ pairs to the count.

The contribution internal to any complete Dyck segment is
nonnegative.  Indeed, apply the preceding crossing balance to the
prefix before each north step at height $h$; its initial and final
heights, $0$ and $h$, are below $h+r+\tfrac12$.  With $e_0=s$ and
$p_i=e_{i-1}-s$ for this segment, the same endpoint cancellation gives
\[
\pdinv_{\rm int}-\sum_{i<j}S_{ij}
=\sum_{i<j}\mathbf1_{\{p_j<e_i\le e_j\}},
\qquad
P_{ij}-S_{ij}
=\mathbf1_{\{p_j<e_i\le e_j\}}
-\mathbf1_{\{p_j<p_i\le e_j\}}.
\]
The same bounds on $\tdinv$ and $\maxtdinv$ therefore give
\[
\dinv_{\rm int}
\ge\#\{(i,j):i<j,\ p_i>p_j,\ e_i\le e_j\}\ge0.
\]

Now put $\alpha=\overleftarrow\mu=(\alpha_1,\ldots,\alpha_m)$ and
let $\ell_i$ be the number of primitive segments in its $i$th
block.  For blocks $i<j$, apply the cross bound to each primitive
segment of the earlier block against the whole later block: its
size $a$ satisfies $a\le\alpha_i\le\alpha_j$.  Their total
contribution is at least $\alpha_i-\ell_i$.  The three defining
terms of $\dinv$ split over internal and cross step pairs, and
all internal contributions are nonnegative.  Therefore, using the
block calculation in the proof of Lemma~\ref{lemmadoff},
\[
\dinv(\pi)\ge\sum_{i<j}(\alpha_i-\ell_i)
=\sum_{i=1}^m(m-i)(\alpha_i-\ell_i)
=n(\mu)-\doff_\alpha(\DP(\beta)).
\]
\end{proof}

\section*{Use of AI tools}
The theoretical derivations and the proofs in this paper were developed with the help of ChatGPT 5.6 Sol.  The authors studied and checked all of this material, and wrote the paper themselves.  We also checked the main identities by computer for small $n$.  The authors take full responsibility for the content of the paper.

\providecommand{\bysame}{\leavevmode\hbox to3em{\hrulefill}\thinspace}


\begin{thebibliography}{99}

\bibitem{BG99}
F. Bergeron and A. M. Garsia,
Science fiction and Macdonald's polynomials,
in Algebraic methods and \(q\)-special functions (Montr\'eal, QC, 1996),
CRM Proc. Lecture Notes, vol. 22,
Amer. Math. Soc., Providence, RI, 1999, 1--52.

\bibitem{BGHT99}
F. Bergeron, A. M. Garsia, M. Haiman, and G. Tesler,
Identities and positivity conjectures for some remarkable operators in the
theory of symmetric functions,
Methods Appl. Anal. 6 (1999), no. 3, 363--420.

\bibitem{BGLX16}
F. Bergeron, A. M. Garsia, E. Leven, and G. Xin,
Compositional \((km,kn)\)-shuffle conjectures,
Int. Math. Res. Not. IMRN 2016, no. 14, 4229--4270.

\bibitem{shuffleline}
J. Blasiak, M. Haiman, J. Morse, A. Pun, and G. H. Seelinger,
A shuffle theorem for paths under any line,
Forum Math. Pi 11 (2023), e5.

\bibitem{LLTcat}
J. Blasiak, M. Haiman, J. Morse, A. Pun, and G. H. Seelinger,
LLT polynomials in the Schiffmann algebra,
J. Reine Angew. Math. 811 (2024), 93--133.

\bibitem{dens}
J. Blasiak, M. Haiman, J. Morse, A. Pun, and G. H. Seelinger,
Dens, nests and the Loehr--Warrington conjecture,
J. Amer. Math. Soc. 38 (2025), no. 4, 1049--1106.

\bibitem{nonsymshuffle}
J. Blasiak, M. Haiman, J. Morse, A. Pun, and G. H. Seelinger,
The nonsymmetric shuffle theorem,
arXiv:2509.24040, 2025.

\bibitem{BMPS19}
J. Blasiak, J. Morse, A. Pun, and D. Summers,
Catalan functions and $k$-Schur positivity,
J. Amer. Math. Soc. 32 (2019), no. 4, 921--963.

\bibitem{Brylawski1973}
T. H. Brylawski,
The lattice of integer partitions,
Discrete Math. 6 (1973), no. 3, 201--219.

\bibitem{BS12}
I. Burban and O. Schiffmann,
On the Hall algebra of an elliptic curve, I,
Duke Math. J. 161 (2012), no. 7, 1171--1231.

\bibitem{CM18}
E. Carlsson and A. Mellit,
A proof of the shuffle conjecture,
J. Amer. Math. Soc. 31 (2018), no. 3, 661--697.

\bibitem{Chen10}
L.-C. Chen,
Skew-linked partitions and a representation-theoretic model for $k$-Schur functions,
Ph.D. thesis, University of California, Berkeley, 2010.

\bibitem{GH96}
A. M. Garsia and M. Haiman,
A remarkable \(q,t\)-Catalan sequence and \(q\)-Lagrange inversion,
J. Algebraic Combin. 5 (1996), no. 3, 191--244.

\bibitem{GN15}
E. Gorsky and A. Negu\c{t},
Refined knot invariants and Hilbert schemes,
J. Math. Pures Appl. (9) 104 (2015), no. 3, 403--435.

\bibitem{GrojnowskiHaiman07}
I. Grojnowski and M. Haiman,
Affine Hecke algebras and positivity of LLT and Macdonald polynomials,
unpublished manuscript, 2007.

\bibitem{haglund-book}
J. Haglund,
The \(q,t\)-Catalan numbers and the space of diagonal harmonics,
University Lecture Series, vol. 41,
American Mathematical Society, Providence, RI, 2008.
With an appendix on the combinatorics of Macdonald polynomials.

\bibitem{HHL05}
J. Haglund, M. Haiman, and N. Loehr,
A combinatorial formula for Macdonald polynomials,
J. Amer. Math. Soc. 18 (2005), no. 3, 735--761.

\bibitem{HHLRU05}
J. Haglund, M. Haiman, N. Loehr, J. B. Remmel, and A. Ulyanov,
A combinatorial formula for the character of the diagonal coinvariants,
Duke Math. J. 126 (2005), no. 2, 195--232.

\bibitem{HMZ12}
J. Haglund, J. Morse, and M. Zabrocki,
A compositional shuffle conjecture specifying touch points of the Dyck path,
Canad. J. Math. 64 (2012), no. 4, 822--844.

\bibitem{Hai02}
M. Haiman,
Vanishing theorems and character formulas for the Hilbert scheme of points
in the plane,
Invent. Math. 149 (2002), no. 2, 371--407.

\bibitem{HL15}
A. Hicks and E. Leven,
A simpler formula for the number of diagonal inversions of an
$(m,n)$-parking function and a returning fermionic formula,
Discrete Math. 338 (2015), no. 3, 48--65.

\bibitem{Jin91}
N. Jing,
Vertex operators and Hall--Littlewood symmetric functions,
Adv. Math. 87 (1991), no. 2, 226--248.

\bibitem{LLT97}
A. Lascoux, B. Leclerc, and J.-Y. Thibon,
Ribbon tableaux, Hall--Littlewood functions, quantum affine algebras, and
unipotent varieties,
J. Math. Phys. 38 (1997), no. 2, 1041--1068.

\bibitem{LW08}
N. A. Loehr and G. S. Warrington,
Nested quantum Dyck paths and \(\nabla(s_\lambda)\),
Int. Math. Res. Not. IMRN 2008, Art. ID rnm157.

\bibitem{Mac95}
I. G. Macdonald,
Symmetric functions and Hall polynomials,
2nd ed., Oxford Mathematical Monographs,
Oxford University Press, Oxford, 1995.

\bibitem{Mel21}
A. Mellit,
Toric braids and \((m,n)\)-parking functions,
Duke Math. J. 170 (2021), no. 18, 4123--4169.

\bibitem{QZ26}
D. Qiu and M. Zhang,
The Schur positivity of \(\nabla m_\mu\),
arXiv:2607.00940, 2026.

\bibitem{Qu26}
M. Qu,
Schur positivity of the nabla operator on two-column modified
Hall--Littlewood polynomials,
arXiv:2605.20954, 2026.


\end{thebibliography}
\end{document}